\documentclass[a4paper,reqno,11pt]{amsart}
\usepackage{amssymb}
\usepackage{amstext}
\usepackage{amsmath}
\usepackage{amscd}
\usepackage{amsthm}
\usepackage{amsfonts}
\usepackage{enumerate}
\usepackage{graphicx}
\usepackage{latexsym}
\usepackage{mathrsfs}

\input xy
\xyoption{all}
\usepackage{pstricks}
\usepackage{lscape}
\usepackage{comment}

\usepackage[pagewise]{lineno}

\newtheorem{theorem}{Theorem}[section]

\newtheorem{corollary}[theorem]{Corollary}
\newtheorem{lemma}[theorem]{Lemma}
\newtheorem{definition-theorem}[theorem]{Definition-Theorem}
\newtheorem{proposition}[theorem]{Proposition}

\theoremstyle{definition}
\newtheorem{definition}[theorem]{Definition}
\newtheorem{remark}[theorem]{Remark}
\newtheorem{example}[theorem]{Example}

\newtheorem*{example*}{Example}
\numberwithin{equation}{section}
\newcommand{\A}{{\mathcal A}}

\newcommand{\C}{{\mathsf C}}

\newcommand{\D}{{\mathsf D}}

\renewcommand{\H}{{\mathsf H}}

\newcommand{\KK}{\mathsf{K}}

\renewcommand{\S}{{\mathsf S}}
\newcommand{\T}{{\mathsf T}}

\newcommand{\U}{{\mathsf U}}

\newcommand{\X}{{\mathcal X}}
\newcommand{\Y}{{\mathcal Y}}

\newcommand{\tT}{{\sf T}}

\newcommand{\tD}{{\sf D}}
\newcommand{\ZZ}{\mathbb{Z}}

\newcommand{\add}{\mathsf{add}\hspace{.01in}}
\newcommand{\Fac}{\mathsf{Fac}\hspace{.01in}}

\newcommand{\Kb}{\mathsf{K}^{\rm b}}
\newcommand{\Db}{\mathsf{D}^{\rm b}}

\newcommand{\Dfd}{\mathsf{D}_{\rm fd}}
\renewcommand{\mod}{\mathsf{mod}\hspace{.01in}}
\newcommand{\Mod}{\mathsf{Mod}\hspace{.01in}}
\newcommand{\proj}{\mathsf{proj}\hspace{.01in}}
\newcommand{\per}{\mathsf{per}\hspace{.01in}}

\newcommand{\thick}{\mathsf{thick}\hspace{.01in}}

\newcommand{\dctilt}{{\it d}\strut\kern-.2em\operatorname{-ctilt}\nolimits}
\newcommand{\Ga}{\Gamma}

\newcommand{\fd}{\mathsf{fd}}

\newcommand{\Le}{\mathsf{L}}
\newcommand{\Ri}{\mathsf{R}}

\newcommand{\ten}{\otimes}
\newcommand{\Hom}{\operatorname{Hom}\nolimits}
\newcommand{\RHom}{\mathbf{R}\strut\kern-.2em\operatorname{Hom}\nolimits}
\newcommand{\HHom}{\mathcal{H}\strut\kern-.2em\operatorname{om}\nolimits}

\newcommand{\End}{\operatorname{End}\nolimits}
\newcommand{\silt}{\mbox{\rm silt}\hspace{.01in}}

\newcommand{\tstr}{\mbox{$t$-{\rm str}}\hspace{.01in}}

\newcommand{\tors}{\mbox{\rm tors}\hspace{.01in}}
\newcommand{\ftors}{\mbox{\rm ftors}\hspace{.01in}}

\newcommand{\ctilt}[1]{{#1}\operatorname{-ctilt}}
\newcommand{\ssilt}[1]{{#1}\operatorname{-silt}}

\newcommand{\xto}{\xrightarrow}

\newcommand{\eg}{\emph{e.g.\ }}
\newcommand{\ie}{\emph{i.e.\ }}

\newcommand{\ca}{\mathcal{A}}
\newcommand{\cb}{\mathcal{B}}
\newcommand{\std}{{\rm std}}
\newcommand{\dg}{{\mathrm{dg}}}
\newcommand{\pc}{{\mathrm{pc}}}
\newcommand{\hf}{\mathrm{hf}}
\newcommand{\chf}{\mathrm{chf}}
\newcommand{\coh}{\mathrm{coh}}
\newcommand{\Qcoh}{\mathrm{Qcoh}}

\newcommand{\cf}{\mathcal{F}}
\newcommand{\cg}{\mathcal{G}}

\begin{document}
\title[Silting-discreteness and t-discreteness]{Discreteness of silting objects and $t$-structures in triangulated categories}

\author{Takahide Adachi}
\address{T. Adachi: Graduate school of Science, Osaka Prefecture University, 1-1 Gakuen-cho, Nakaku, Sakai, Osaka, 599-8531, Japan}
\email{adachi@mi.s.osakafu-u.ac.jp}
\thanks{T.~Adachi is supported by Grant-in-Aid for JSPS Research Fellow 17J05537.} 
\author{Yuya Mizuno}
\address{Y. Mizuno: Department of Mathematics, Faculty of Science, Shizuoka University, 
836 Ohya, Suruga-ku, Shizuoka, 422-8529, Japan}
\email{yuya.mizuno@shizuoka.ac.jp}
\thanks{Y. Mizuno is supported by Grant-in-Aid for JSPS Research Fellow 17J00652.} 
\author{Dong Yang}
\address{D. Yang: Department of Mathematics, Nanjing University, 22 Hankou Road, Nanjing 210093, P. R.
China }
\email{yangdong@nju.edu.cn}
\date{\today}
\begin{abstract}
We introduce the notion of ST-pairs of triangulated subcategories, a prototypical example of which is the pair of the bound homotopy category and the bound derived category of a finite-dimensional algebra.  For an ST-pair $(\C,\D)$, we construct an injective order-preserving map from silting objects in $\C$ to bounded $t$-structures on $\D$ and show that the map is bijective if and only if $\C$ is silting-discrete if and only if $\D$ is $t$-discrete.
Based on a work of Qiu and Woolf, the above result is applied to show that if $\C$ is silting-discrete then the stability space of $\D$ is contractible. This is used to obtain the contractibility of the stability spaces of some Calabi--Yau triangulated categories associated to Dynkin quivers. 
\end{abstract}
\maketitle
\tableofcontents

\section{Introduction}

Silting objects were introduced in \cite{KV88} as a generalisation of tilting objects in order to parametrise bounded $t$-structures on derived categories of Dynkin quivers: there is a one-to-one correspondence between isomorphism classes of basic silting objects and bounded $t$-structures. Later such a correspondence was studied in the setting of finite-dimensional algebras in \cite{KoY14}, in the setting of non-positive differential graded (=dg) algebras with finite-dimensional total cohomology in \cite{SY16} and in the setting of homologically smooth non-positive dg algebras with finite-dimensional zeroth cohomology in \cite{KN2}. 
In recent years the connection between silting theory and the theory of $t$-structures (and co-$t$-structures) has received much attention \cite{HKM02,BR07,Bo10,MSSS13,AI12,AHMV,PV15}, and
silting theory has been playing an increasingly important role in the study of triangulated categories of algebraic origin, due to such a connection as well as its connection to cluster-tilting theory \cite{Am09,Gu11,BRT11,KN2,BY13,IYa14}, to stability spaces \cite{BPP2,QW14} and 
to the theory of universal localisation of rings \cite{AHMV2,MS16}.

\smallskip
This paper is a continuation of \cite{KV88,KoY14,SY16,KN2}. In these papers, pairs of triangulated categories appear as the natural home of silting objects and $t$-structures, respectively. Motivated by this fact,
we introduce the notion of ST-pairs (Definition~\ref{defn:ST-pair}) and study the correspondence between silting objects of $\C$ and $t$-structures on $\D$, where $(\C,\D)$ is an ST-pair. This bivariant viewpoint puts the above papers into a uniform framework.  
A prototypical example of an ST-pair is
\begin{itemize}
\item[(1)]  (Lemma~\ref{lem:ST-pair-for-fd-algebra}) $(\Kb(\proj\Lambda),\Db(\mod \Lambda))$, where $\Lambda$ is a finite-dimensional algebra, $\Kb(\proj\Lambda)$ is the bounded homotopy category of finitely generated projective $\Lambda$-modules and $\Db(\mod\Lambda)$ is the bounded derived category of finite-dimensional $\Lambda$-modules.
\end{itemize}
Other relevant examples include
\begin{itemize}
\item[(2)] (Lemma~\ref{lem:ST-pair-for-fd-non-positive-dg-algebra}) $(\per(A),\Dfd(A))$, where $A$ is a non-positive dg algebra with finite-dimensional total cohomology, $\per(A)$ is the perfect derived category of dg $A$-modules and $\Dfd(A)$ is the finite-dimensional derived category of dg $A$-modules;
\item[(3)] (Lemma~\ref{lem:standard-ST-pair-for-smooth-non-positive-dg-algebra}) $(\per(\Gamma),\Dfd(\Gamma))$, where $\Gamma$ is a homologically smooth non-positive dg algebra such that $H^0(\Gamma)$ is finite-dimensional;
\item[(4)] (Corollaries~\ref{cor:ST-pair-for-variety-with-silting}~and~\ref{cor:uniqueness-for-variety}) $(\per(X),\Db(\coh(X)))$, where $X$ is a projective scheme over a field such that the derived category $\per(X)$ of perfect complexes has a silting object (\eg the direct sum of a full exceptional sequence under suitable shifts), or equivalently, the bounded derived category $\Db(\coh(X))$ of coherent sheaves has an algebraic $t$-structure (note that these conditions imply that the Grothendieck group is free of finite rank).
\end{itemize}

For an ST-pair $(\C,\D)$ we establish an injective map $\Psi$ from the set $\silt(\C)$ of isomorphism classes of basic silting objects of $\C$ to the set $\tstr(\D)$ of bounded $t$-structures on $\D$. 
Moreover, we give characterisations of the map being bijective: it is bijective if and only if $\C$ is silting-discrete (\ie the number of silting objects of $\C$ is `locally finite', see Definition~\ref{def silting-discrete}) if and only if $\D$ is $t$-discrete (\ie the number of bounded $t$-structures on $\D$ is `locally finite', see Definition~\ref{defn:t-discreteness}). Furthermore, the map $\Psi$ is compatible with the natural partial orders and mutation operations on $\silt(\C)$ and $\tstr(\D)$.

\begin{theorem} \label{intro 1} 
Let $(\C,\D)$ be an ST-pair. 
\begin{itemize}
\item[(a)] {\rm (Theorem \ref{basic-map} and Proposition~\ref{st-mutation})} There is an order-preserving injection
\[
\Psi\colon \silt(\C) \to \tstr(\D),~~~~
M\mapsto (\tD_{M}^{\le 0}, \tD_{M}^{\ge 0})
\]
which takes silting mutation to semisimple HRS-tilt, where $\tD_{M}^{\le 0}$ (respectively, $\tD_{M}^{\ge 0}$) is the full subcategory of $\D$ consisting of objects $X$ such that
$\Hom(M,X[i])=0$ for all $i >0$ (respectively, $i<0$).
\item[(b)] {\rm (Theorem \ref{main})} The following conditions are equivalent. 
\begin{itemize}
\item[(i)] The map $\Psi$ is bijective.
\item[(ii)] $\C$ is silting-discrete.
\item[(iii)] $\D$ is t-discrete.
\item[(iv)] The heart of every bounded $t$-structure on $\D$ has a projective generator.
\end{itemize}
\end{itemize}
\end{theorem}

Under the conditions of Theorem \ref{intro 1}(b) all bounded $t$-structures are captured by silting theory, just as in the special case when $\C=\D$ is the bounded derived category of a Dynkin quiver \cite{KV88}. For the ST-pairs (1--4), the image of the map $\Psi$ is known to be the set of algebraic $t$-structures, see \cite[Theorem 6.1]{KoY14}, \cite[Theorem 1.2]{SY16}, \cite[Theorem 13.3]{KN2} and Corollary~\ref{cor:uniqueness-for-variety}, respectively.

\smallskip
Part (a) of the following theorem (and its proof) provides a certain canonical construction and justify the notion of ST pairs.
We also study uniqueness property for 
some important examples.

\begin{theorem}
\begin{itemize}
\item[(a)] {\rm (Theorem~\ref{thm:silting=>ST-pair})}
Let $\C$ be a Hom-finite Krull--Schmidt algebraic triangulated category with silting objects. Then there exist triangulated categories $\C'$ and $\D$ such that $(\C',\D)$ is an ST-pair and $\C'$ is triangle equivalent to $\C$.
\vspace{2pt}
\item[(b)] {\rm (Corollaries~\ref{cor:unique-ST-pair-for-the-homotopy-category}~and~\ref{cor:uniqueness-for-fd-algebra})} Let $\Lambda$ be a finite-dimensional algebra. If $\D$ is an algebraic triangulated category and $(\Kb(\proj\Lambda),\D)$ is an ST-pair, then $\D$ contains $\Kb(\proj\Lambda)$ and there is a triangle equivalence $\Db(\mod\Lambda)\to\D$ which is the identity on $\Kb(\proj\Lambda)$; if $(\C,\Db(\mod\Lambda))$ is an ST-pair, then $\C=\Kb(\proj\Lambda)$.
\vspace{2pt}
\item[(c)] {\rm (Corollaries~\ref{cor:ST-pair-for-variety-with-silting}~and~\ref{cor:uniqueness-for-variety})} Let $X$ be a projective scheme over a field. Then the pair $(\per(X),\Db(\coh(X)))$ is an ST-pair if and only if $\per(X)$ has a silting object if and only if $\Db(\coh(X))$ has an algebraic $t$-structure. Moreover, if $(\C,\Db(\coh(X)))$ is an ST-pair, then $\C=\per(X)$.
\end{itemize}
\end{theorem}

See Section~\ref{ss:completing-silting-to-ST-pair} for more uniqueness results, especially on categories involved in the ST-pairs (2) and (3).

\smallskip
One application of our main result is  the study of Bridgeland's stability spaces of triangulated categories \cite{Br1}. 
Basing on the work of Qiu and Woolf \cite{QW14}, we apply Theorem~\ref{intro 1} to give sufficient conditions on the contractibility of the stability space of a triangulated category $\D$. 

\begin{theorem}\label{main2}
\begin{itemize}
\item[(a)] {\rm (Theorem~\ref{thm:contractibility-vs-silting-discreteness})}
Let $(\C,\D)$ be an ST-pair. 
If $\C$ is silting-discrete, then the stability space of $\D$ is contractible.
\item[(b)] {\rm (Corollary \ref{cor:contractible-derived-category})} Let $\Lambda$ be a finite-dimensional algebra. 
If $\Kb(\proj\Lambda)$ is silting-discrete, then the stability space of $\Db(\mod\Lambda)$ is contractible.
\end{itemize}
\end{theorem}

In Example \ref{ex:silting-discrete-derived-categories} we will give some known examples where $\Kb(\proj\Lambda)$ is silting-discrete. By Theorem~\ref{main2}(b), Example  \ref{ex:silting-discrete-derived-categories}(2) affirms the first part of  \cite[Conjecture 5.8]{Q15}. During the preparation of this work, we were informed by Pauktztello, Saor\'in and Zvonareva that they independently obtained 
Theorem~\ref{main2}(b) in \cite{PSZ17}.

\smallskip
We also investigate silting-discreteness for a particularly  important 
class of triangulated categories.
We call an ST-pair $(\C,\D)$ a \emph{$(d+1)$-Calabi--Yau pair} (Definition \ref{def CYtriple}) if $\C\supseteq \D$ and there is a bifunctorial isomorphism $D\Hom(X,Y)\simeq \Hom(Y,X[d+1])$ for $X\in\D$ and $Y\in\C$. 
In this case, we give a characterisation of silting-discreteness for $\C$ using cluster-tilting theory for the \emph{cluster category} $\C/\D$ (\cite[Section 5.3]{IYa14}). 

\begin{theorem}[Theorem \ref{equivalent thm} and Corollary~\ref{cor:cluster-tilting-finite-vs-contractibility-cy-case}]
Let $d\geq 1$ be an integer and let $(\C,\D)$ be a $(d+1)$-Calabi--Yau pair. 
If the cluster category $\C/\D$ has only finitely many isomorphism classes of basic $d$-cluster-tilting objects, then $\C$ is silting-discrete and the stability space of $\D$ is contractible. 
\end{theorem}

We apply this result to two important classes of dg algebras: the complete Ginzburg dg algebras and $(d+1)$-derived preprojective algebras, which are examples of deformed Calabi--Yau completions in the sense of \cite{Ke2}. 

\begin{corollary}[{Lemma~\ref{lem:Hom-finiteness-and-cy-property-for-derived-preprojective-algebras}, Corollaries \ref{main3-1}~and~\ref{derived pp}}]
\label{cor:contractible-cy-category-1}
Let $d\geq 1$ be an integer and $\Gamma=\Gamma_{d+1}(Q)$ the derived $(d+1)$-preprojective algebra of a finite quiver $Q$.
Then the following statements hold.
\begin{itemize}
\item[(a)] $(\per(\Gamma),\Dfd(\Gamma))$ is a $(d+1)$-Calabi--Yau pair if and only if $\dim H^0(\Gamma)<\infty$.  
\item[(b)] $\per(\Gamma)$ is silting-discrete if and only if  $Q$ is Dynkin. 
\item[(c)] If $Q$ is Dynkin, then the stability space of $\Dfd(\Gamma)$ is contractible.
\end{itemize}
\end{corollary}

\begin{corollary}[{Corollaries \ref{main3-2}~and~\ref{derived pp}}]
\label{cor:contractible-cy-category-2}
Let $\Gamma=\widehat{\Gamma}(Q,W)$ be the complete Ginzburg dg algebra of a finite quiver $Q$ with a nondegenerate potential $W$.
Then the following statements hold.
\begin{itemize}
\item[(a)] $(\per(\Gamma),\Dfd(\Gamma))$ is a $3$-Calabi--Yau pair if and only if $\dim H^0(\Gamma)<\infty$.  
\item[(b)] $\per(\Gamma)$ is silting-discrete if and only if  $Q$ is related to a Dynkin quiver by a finite sequence of quiver mutations. 
\item[(c)] If  $Q$ is related to a Dynkin quiver by a finite sequence of quiver mutations, then the stability space of $\Dfd(\Gamma)$ is contractible.
\end{itemize}
\end{corollary}
Corollary~\ref{cor:contractible-cy-category-1}(c) and Corollary~\ref{cor:contractible-cy-category-2}(c) affirm \cite[Conjecture 1.3]{I17} and the second part of  \cite[Conjecture 5.8]{Q15}.

\smallskip

\subsection*{Notation and convention} 
Throughout this paper let $K$ be a field and  denote the $K$-dual by $D=\Hom_K(-,K)$. 
Let $\T$ be a $K$-linear additive category. We say that $\T$ is \emph{Hom-finite} if all morphism spaces of $\T$ are finite-dimensional over $K$. We say that $\T$ is \emph{idempotent complete} if any idempotent $e\colon M\to M$ in $\T$ arises from a direct sum decomposition $M\simeq \mathrm{im}(e)\oplus \mathrm{ker}(e)$. We say that $\T$ is \emph{Krull--Schmidt} if every object $M$ of $\T$ is isomorphic to $M_1\oplus\ldots \oplus M_n$ with $M_1,\ldots,M_n$ having local endomorphism algebras. For example, any Hom-finite $K$-linear abelian category is Krull--Schmidt (by \cite[Theorem 1]{Atiyah56}).
The object $M$ is said to be \emph{basic} if in the above decomposition $M_1\oplus\ldots\oplus M_n$ the objects $M_1,\ldots,M_n$ are pairwise non-isomorphic. In this case, we denote $|M|=n$. 

For an object $M$ of $\T$, we denote by $\add(M)$ the smallest full subcategory of $\T$ which contains $M$ and which is closed under taking finite direct sums and direct summands. For a full subcategory $\X$ of $\T$, define full subcategories $\X^\perp:=\X^{\perp_{\T}}$ and ${}^{\perp}\X:={}^{\perp_{\T}}\X$ of $\T$ as 
\begin{eqnarray*}
\X^{\perp_\T}\hspace{-5pt}&:=\hspace{-5pt}&\{Y\in\T\mid \Hom_\T(X,Y)=0~\forall X\in\X\}\\
{}^{\perp_\T}\X\hspace{-5pt}&:=\hspace{-5pt}&\{Y\in\T\mid \Hom_\T(Y,X)=0~\forall X\in\X\}.
\end{eqnarray*}
For $Y\in\T$, a \emph{right $\X$-approximation} of $Y$ is a morphism $f\colon Y\to X$ such that any morphism $Y\to X'$ with $X'\in\X$ factors through $f$. It is a \emph{minimal right $\X$-approximation} if in addition any morphism $g\colon X\to X$ with $gf=f$ is an isomorphism. We say that $\X$ is \emph{contravariantly finite} in $\T$ if right $\X$-approximation exists for any object of $\T$. If $\X$ is Hom-finite Krull--Schmidt and contravariantly finite in $\T$, then minimal right $\X$-approximations exist for any object of $\T$. Dually, we define  \emph{(minimal) left $\X$-approximations} and \emph{covariantly finite} subcategories. See for example \cite{AS80} and \cite[Section 1]{ManteseReiten04}. If $\T$ is Hom-finite and $M$ is an object of $\T$, then $\add(M)$ is both contravariantly finite and covariantly finite. 

Assume that $\T$ is triangulated with shift functor $[1]$. It is said to be \emph{algebraic} if it is triangle equivalent to the stable category of a Frobenius category (for example, the derived category of a dg algebra and its triangulated subcategories are algebraic triangulated categories). It is said to be \emph{of finite type} if for any $X,Y\in\T$ the space $\bigoplus_{p\in\mathbb{Z}}\Hom_\T(X,Y[p])$ is finite-dimensional. For a full subcategory $\X$ of $\T$, we denote by $\thick(\X)$ the \emph{thick subcategory} of $\T$ generated by $\X$, that is, the smallest triangulated subcategory of $\T$ which contains $\X$ and which is closed under taking direct summands. For two full subcategories $\X$ and $\Y$ of $\T$, define $\Hom_\T(\X,\Y)$ as the union of $\Hom_\T(X,Y)$ for all $X\in\X$ and $Y\in\Y$ and define $\X*\Y$ as the full subcategory of $\T$ consisting of objects $M$ which admits a triangle $X\to M\to Y\to X[1]$ with $X\in\X$ and $Y\in\Y$. We say that $\X$ is \emph{extension-closed} if $\X*\X\subseteq\X$. 
We refer to \cite{Neeman99,Happel88} for knowledge on triangulated categories.

\subsection*{Acknowledgement}
TA would like to express his deep gratitude to Osamu Iyama for many fruitful discussions. He thanks Liu Yu for helpful discussion. DY thanks Steffen Koenig and Lidia Angeleri H\"ugel for answering his questions. All authors gratefully thank the referee for his/her very helpful comments, especially the suggestions to study uniqueness of ST-pairs, to include ST-pairs from algebraic geometry and to compare discreteness in the sense of \cite{BPP3} with $t$-discreteness. These comments led to a great improvement of the paper.

\section{Preliminaries}
In this section we recall Morita's theorem, support $\tau$-tilting modules, derived categories of dg algebras and (co)homologically finite objects.

\subsection{Finite-dimensional algebras}
We refer to \cite{AuslanderReitenSmaloe95,AssemSimsonSkowronski,DrozdKirichenko} for knowledge on finite-dimensional algebras. 

\smallskip

Let $\Lambda$ be a finite-dimensional $K$-algebra. Then the category $\mod\Lambda$  of finite-dimensional (right) $\Lambda$-modules is Hom-finite and Krull--Schmidt. The algebra $\Lambda$ is said to be \emph{basic} if the free module $\Lambda$ is a basic object in $\mod\Lambda$. We denote by $\Db(\mod\Lambda)$ the bounded derived category of $\mod\Lambda$. We denote by $\proj\Lambda$ the category of finite-dimensional projective $\Lambda$-modules and by $\Kb(\proj\Lambda)$ the bounded homotopy category of $\proj\Lambda$. There is a natural embedding $\Kb(\proj\Lambda)\hookrightarrow\Db(\mod \Lambda)$, via which we view $\Kb(\proj\Lambda)$ as a thick subcategory of $\Db(\mod\Lambda)$. The following is a standard fact in the representation theory of finite-dimensional algebras. 

\begin{lemma}[{\cite[Lemma 2.1]{HocheneggerKalckPloog16}}]\label{lem:KS-vs-idempotent-complete}
Let $\T$ be a $K$-linear Hom-finite additive category. Then $\T$ is Krull--Schmidt if and only if $\T$ is idempotent complete.
\end{lemma} 


\smallskip

Let $\ca$ be a $K$-linear abelian category. A \emph{projective generator} of $\ca$ is a projective object $P$ such that to each object $M$ of $\ca$ there is an epimorphism $Q\to M$ with $Q\in\add(P)$. For a finite-dimensional $K$-algebra $\Lambda$, the free module $\Lambda$ is a projective generator of $\mod\Lambda$. Conversely, by Morita's theorem, if a $K$-linear Hom-finite abelian category $\ca$ has a projective generator $P$, then the functor $\Hom_{\ca}(P,-)\colon \ca\to\mod\End_{\ca}(P)$ is an equivalence.

\subsection{Support $\tau$-tilting modules} We follow \cite{AIR14,DIJ15}. 
Let $\Lambda$ be a basic finite-dimensional $K$-algebra. A basic finite-dimensional $\Lambda$-module $M$ is said to be \emph{$\tau$-tilting} if $|M|=|\Lambda|$ and $\Hom_\Lambda(M,\tau M)=0$, where $\tau$ is the Auslander--Reiten translation of $\mod\Lambda$. For example, a basic tilting module of projective dimension $\leq 1$ is a $\tau$-tilting module. The algebra $\Lambda$ is said to be \emph{$\tau$-tilting finite} if there are only finitely many isomorphism classes of basic $\tau$-tilting $\Lambda$-modules. For example, local finite-dimensional algebras and representation-finite finite-dimensional algebras are $\tau$-tilting finite.

A basic $\Lambda$-module $M$ is said to be \emph{support $\tau$-tilting} if it is a $\tau$-tilting module over the factor algebra $\Lambda/\Lambda e\Lambda$ for some idempotent $e\in\Lambda$. According to \cite[Proposition 3.9]{DIJ15}, $\Lambda$ is $\tau$-tilting finite if and only if the number of isomorphism classes of basic support $\tau$-tilting $\Lambda$-modules is finite. More equivalent conditions can be found in \cite{DIJ15}.

\subsection{Derived categories of dg algebras}

Let $A$ be a dg $K$-algebra. Let $\C_{\dg}(A)$ be the dg category of (right) dg $A$-modules, see \cite[Section 3.1]{Ke4}. For $M,N\in\C_{\dg}(A)$, the Hom-complex $\Hom_{\C_{\dg}(A)}(M,N)$ in $\C_{\dg}(A)$ is the complex of $K$-vector spaces with its degree $p$ component $\Hom^p_{\C_{\dg}(A)}(M,N)$ being the space of homogeneous $A$-linear maps of degree $p$ from $M$ to $N$ (here we consider $A$ as a graded algebra and $M,N$ as graded $A$-modules) and with its differential defined by $d(f)=d_N\circ f-(-1)^p f\circ d_M$ for $f\in\Hom^p_{\C_{\dg}(A)}(M,N)$. For example, $\Hom_{\C_{\dg}(A)}(A,M)=M$. A dg $A$-module $M$ is said to be \emph{$\KK$-projective} if $\Hom_{\C_{\dg}(A)}(M,N)$ is acyclic for any acyclic dg $A$-module $N$. For example, $A_A$ is $\KK$-projective. 

The dg category $\C_{\dg}(A)$ is a pretriangulated dg category in the sense of \cite[Section 4.3]{Ke4}. Let $\KK(A):=H^0\C_{\dg}(A)$ be the homotopy category of $\C_{\dg}(A)$, see \cite[Section 2.2]{Ke4}. Precisely, the objects of $\KK(A)$ are dg $A$-modules, and for two  dg $A$-modules $M$ and $N$ we have
\[
\Hom_{\KK(A)}(M,N):=H^0\Hom_{\C_{\dg}(A)}(M,N).
\]
Then $\KK(A)$ is a triangulated category with shift functor $[1]$ being the shift of dg modules. The \emph{derived category} $\D(A)$ of dg $A$-modules is defined as the triangle quotient of $\KK(A)$ by the full subcategory of $\KK(A)$ consisting of acyclic dg $A$-modules. For $p\in\mathbb{Z}$ and for dg $A$-modules $M$ and $N$ with $M$ being $\KK$-projective, we have
\begin{eqnarray*}\label{eq:Hom-space-in-derived-category}
\Hom_{\D(A)}(M,N[p])=\Hom_{\KK(A)}(M,N[p])=H^p\Hom_{\C_{\dg}(A)}(M,N).
\end{eqnarray*}
In particular, 
\begin{eqnarray}\label{eq:Hom-space-from-free-dg-module}
\Hom_{\D(A)}(A,M[p])=H^p(M)
\end{eqnarray} 
for $p\in\mathbb{Z}$. 
The \emph{perfect derived category} $\per(A)$ is the thick subcategory of $\D(A)$ generated by $A_A$, and the \emph{finite-dimensional derived category} $\Dfd(A)$ is the full subcategory of $\D(A)$ consisting of dg $A$-modules $M$ whose total cohomology is finite-dimensional over $K$. If $A$ is a finite-dimensional $K$-algebra, we can consider it as a dg algebra concentrated in degree $0$. In this case, $\D(A)$ is exactly the unbounded derived category $\D(\Mod A)$ of the category of all $A$-modules and there are canonical triangle equivalences $\Kb(\proj A)\to\per(A)$ and $\Db(\mod A)\to\Dfd(A)$.


The dg algebra $A$ is said to be \emph{homologically smooth} if  $A\in\per(A^{op}\ten A)$ and \emph{bimodule $n$-Calabi--Yau} if in addition there is an isomorphism $\RHom_{A^{op}\ten A}(A,A^{op}\ten A)\simeq A[-n]$ in $\D(A^{op}\ten A)$. 

\begin{lemma}[{\cite[Lemma 4.1]{Ke3} and its proof}]
\label{lem:homologically-smooth=>finite-projective-dimension}
If $A$ is homologically smooth, then $\Dfd(A)\subseteq\per(A)$. If $A$ is bimodule $n$-Calabi--Yau, then there is a bifunctorial isomorphism for $X\in\Dfd(A)$ and $Y\in\per(A)$
\[
D\Hom(X,Y)\stackrel{\simeq}{\longrightarrow}\Hom(Y,X[n]).
\]
\end{lemma}

\subsubsection{Derived categories of pseudo-compact dg algebras}
For bilaterally pseudocompact dg algebras (\eg complete Ginzburg dg algebras of quivers with potential, see Section~\ref{ss:ginzburg-dg-algebras}), there are pseudocompact versions $\per^\pc$ and $\Dfd^\pc$ of $\per$ and $\Dfd$, as defined and used in \cite[Appendix A]{KeY11}. In this subsection we briefly explain how they are defined and that under suitable conditions they are equivalent to $\per$ and $\Dfd$, respectively. This allows us to state the results which were in terms of $\per^{\pc}$ and $\Dfd^\pc$ now in terms of $\per$ and $\Dfd$. Lemma~\ref{lem:topologically homologically-smooth=>finite-projective-dimension} is the only result that will be used later.

\smallskip

We follow  \cite[Appendix A.11]{KeY11}.
Let $A$ be a bilaterally pseudocompact dg $K$-algebra. Let $\C_{\dg}^{\pc}(A)$ be the dg category of pseudocompact dg $A$-modules. For $M,N\in\C_{\dg}^{\pc}(A)$, the Hom-complex $\Hom_{\C_{\dg}^{\pc}(A)}(M,N)$ is the complex of $K$-vector spaces with its degree $p$ component $\Hom^p_{\C^{\pc}_{\dg}(A)}(M,N)$ being the space of continuous homogeneous $A$-linear maps of degree $p$ from $M$ to $N$ (here we consider $A$ as a pseudocompact graded algebra and $M,N$ as pseudocompact graded $A$-modules) and with its differential defined by $d(f)=d_N\circ f-(-1)^p f\circ d_M$ for $f\in\Hom^p_{\C^{\pc}_{\dg}(A)}(M,N)$. The dg category $\C^{\pc}_{\dg}(A)$ is pretriangulated. Similar to the beginning of this subsection, we can define $\D^{\pc}(A)$, $\per^{\pc}(A)$ and $\Dfd^{\pc}(A)$. The dg algebra $A$ is said to be \emph{topologically homologically smooth} if $A\in\per^{\pc}(A^{op}\widehat{\ten}A)$, where $A^{op}\widehat{\ten}A$ is the completion of $A^{op}\ten A$.

\begin{lemma}\label{lem:perfect-derived-category-of-pseudocompact-dg-algebra}
The forgetful functor $\D^{\pc}(A)\to\D(A)$ restricts to a triangle equivalence $\per^{\pc}(A)\to\per(A)$. If $A$ is topologically homologically smooth, then it restricts to a triangle equivalence  $\Dfd^{\pc}(A)\to\Dfd(A)$.
\end{lemma}

The second statement is contained in \cite[Proposition A.14(c)]{KeY11}. The first statement is implicitly claimed in \cite[Appendix A.11]{KeY11}. Here we give a quick explanation.  Let $\mathrm{pretr}_{\C_{\dg}^{\pc}(A)}(A)$ and $\mathrm{pretr}_{\C_{\dg}(A)}(A)$ be the pretriangulated subcategories of $\C_{\dg}^{\pc}(A)$ and $\C_{\dg}(A)$ generated by $A$, respectively. Then $\per^{\pc}(A)$ and $\per(A)$ are idempotent completions of the homotopy categories of $\mathrm{pretr}_{\C_{\dg}^{\pc}(A)}(A)$ and $\mathrm{pretr}_{\C_{\dg}(A)}(A)$, respectively.
Direct computation shows that $\Hom_{\C_{\dg}^{\pc}(A)}(A,A)=A=\Hom_{\C_{\dg}(A)}(A,A)$. This implies that the forgetful functor $\mathrm{pretr}_{\C_{\dg}^{\pc}(A)}(A)\to\mathrm{pretr}_{\C_{\dg}(A)}(A)$ is a dg equivalence and hence the induced triangle functor $\per(A)\to\per^{\pc}(A)$ is a triangle equivalence.

\smallskip

The dg algebra $A$ is said to be \emph{bimodule $n$-Calabi--Yau} if it is topologically homologically smooth and there is an isomorphism $\RHom_{\D^{\pc}(A^{op}\widehat{\ten}A)}(A,A^{op}\widehat{\ten}A)\simeq A[-n]$ in $\D^{\pc}(A^{op}\widehat{\ten}A)$. 
The following result is obtained by combining \cite[Proposition A.14(c) and Lemma A.16]{KeY11} with Lemma~\ref{lem:perfect-derived-category-of-pseudocompact-dg-algebra}.

\begin{lemma}
\label{lem:topologically homologically-smooth=>finite-projective-dimension}
If $A$ is topologically homologically smooth, then $\Dfd(A)\subseteq\per(A)$. If $A$ is bimodule $n$-Calabi--Yau, then there is a bifunctorial isomorphism for $X\in\Dfd(A)$ and $Y\in\per(A)$
\[
D\Hom(X,Y)\stackrel{\simeq}{\longrightarrow}\Hom(Y,X[n]).
\]
\end{lemma}

\subsection{Compact objects and (co)homologically finite objects}
Let $\T$ be a $K$-linear triangulated category. Denote by $\T^{\mathrm{c}}$ the full subcategory of $\T$ of compact objects, that is, objects $X$ such that $\Hom_\T(X,-)$ commutes with coproducts. For a full subcategory $\X$ of $\T$, define
\begin{eqnarray*}
\T^{\mathrm{chf}}_\X\hspace{-5pt}&:=\hspace{-5pt}&\{Y\in\T\mid \bigoplus\nolimits_{p\in\mathbb{Z}}\Hom_\T(X,Y[p]) \text{ is finite-dimensional for all }X\in\X\}\\
\T_{\mathrm{hf}}^\X\hspace{-5pt}&:=\hspace{-5pt}&\{Y\in\T\mid \bigoplus\nolimits_{p\in\mathbb{Z}}\Hom_\T(Y,X[p]) \text{ is finite-dimensional for all }X\in\X\}.
\end{eqnarray*}
If $\X=\T$, we will omit the subscript and superscript $\X$ and simply write $\T^{\chf}$ and $\T_{\hf}$.

\smallskip
Let $X$ be a projective scheme over $K$. Denote by $\coh(X)$ and $\Qcoh(X)$ the category of coherent sheaves and the category of quasi-coherent sheaves over $X$, respectively, and denote by $\per(X)$ the derived category of perfect complexes over $X$.

\begin{lemma}\label{lem:compact-and-homologically-finite-objects-for-projetive-variety}
Let $X$ be a projective scheme over $K$. 
Let $\T=\D(\Qcoh(X))$, $\C=\per(X)$ and $\D=\Db(\coh(X))$. Then $\C=\T^{\mathrm{c}}$, $\D=\T^{\chf}_{\C}$ and $\C=\T^{\D}_{\hf}=\D_{\hf}$.
\end{lemma}
\begin{proof}
The first equality is \cite[Theorem 3.1.1]{BondalVandenBergh03}, the second equality is \cite[Lemma 7.46]{Rouquier08} and the third equality is \cite[Lemma 7.49]{Rouquier08}.
\end{proof}

We have a similar result for dg algebras.
\begin{lemma}\label{lem:compact-and-homologically-finite-objects-for-dg-algebra}
Let $A$ be a dg $K$-algebra. Let $\T=\D(A)$, $\C=\per(A)$ and $\D=\Dfd(A)$. Then $\C=\T^\mathrm{c}$ and $\D=\T^{\chf}_{\C}$. If $A$ is homologically smooth or topologically homologically smooth, then $\D=\C^{\chf}$. If $H^p(A)=0$ for all $p>0$ and $H^*(A)$ is finite-dimensional,
then $\C=\T^{\D}_{\hf}=\D_{\hf}$.
\end{lemma}
\begin{proof}
The equality $\C=\T^\mathrm{c}$ follows from \cite[Theorem 5.3]{Ke1}. The equality $\D=\T^{\chf}_{\C}$ follows from the formula~\eqref{eq:Hom-space-from-free-dg-module}.  If $A$ is homologically smooth or topologically smooth, then $\C\supseteq\D$ by Lemma~\ref{lem:homologically-smooth=>finite-projective-dimension} or Lemma~\ref{lem:topologically homologically-smooth=>finite-projective-dimension}, and hence $\D=\T^{\chf}_{\C}=\C^{\chf}$. If $H^*(A)$ is finite-dimensional, then $\C\subseteq \D$. If moreover $H^p(A)=0$ for all $p>0$, then for objects in $\D$ minimal $\KK$-projective resolutions exist (one can either give a direct proof using \cite[Theorem 3.1 c)]{Ke1}, or use Proposition~\ref{prop:minimal-resolution}), and the rest of the proof is similar to that of \cite[Lemma 7.49]{Rouquier08}.
\end{proof}

As compact objects and (co)homologically finite objects are defined homologically, they are preserved under triangle equivalences.

\begin{lemma}\label{lem:restriction-of-der-equiv-to-per-and-dfd}
Let $A$ be a dg $K$-algebra. If $F:\T\to\D(A)$ is a triangle equivalence, then $F$ restricts to triangle equivalences $\T^{\mathrm{c}}\to\per(A)$ and $\T^{\mathrm{chf}}_{\T^{\mathrm{c}}}\to\Dfd(A)$.
\end{lemma}


\section{Silting objects and $t$-structures}
\label{s:silting-objects-and-t-structures}

In this section we recall the notions and basic properties of silting objects and $t$-structures.

\subsection{Silting objects}
In this subsection, we recall the definition of silting objects and silting mutations.
For details, we refer to \cite{AI12}. Let $\T$ be a $K$-linear triangulated category with shift functor $[1]$. 

\begin{definition} Let $M$ be an object of $\T$.
\begin{itemize}
\item[(a)] $M$ is said to be \emph{presilting} if $\Hom_{\T}(M,M[n])=0$ for all positive integers $n$.
\item[(b)] $M$ is said to be \emph{silting} if it is presilting and $\T=\thick(M)$. If $\T$ is Krull--Schmidt, then we denote by $\silt(\T)$ the set of isomorphism classes of basic silting objects of $\T$.
\item[(c)] $M$ is said to be \emph{tilting} if $\Hom_\T(M,M[n])=0$ for $n\neq 0$ and $\T=\thick(M)$.
\end{itemize}
\end{definition}

Note that a presilting object $M$ of $\T$ is a silting object of $\thick(M)$. We give a typical example of a silting object.
\begin{example}
Let $\Lambda$ be a $K$-algebra.
Then $\Lambda$ is a silting object of the bounded homotopy category $\Kb(\proj\Lambda)$. It is in fact a tilting object.
\end{example}

More generally, we have

\begin{lemma}\label{lem:non-positive-dg-algebra-and-silting}
Let $A$ be a dg $K$-algebra such that $H^{p}(A)=0$ for all $p>0$. Then $A$ is a silting object of $\per(A)$.
\end{lemma}
\begin{proof}
This follows from the formula \eqref{eq:Hom-space-from-free-dg-module} for $M=A$.
\end{proof}

For objects $M,N$ of $\T$, we write $M\ge N$ if $\Hom_{\T}(M,N[n])=0$ for all positive integers $n$.
Then the relation $\ge$ gives a partial order on $\silt(\T)$ by \cite[Theorem 2.11]{AI12}.

\begin{definition}\label{defn:n-term-silting}Let $M$ be a basic silting object of $\T$ and let $n$ be a positive integer. An object $N$ of $\T$ is called an 
\emph{$n$-term silting object with respect to $M$} if it belongs to the set
\[
\silt^n_M(\T):=\{ N\in \silt(\T) \mid M\ge N \ge M[n-1]\}. \notag
\]
\end{definition}

For a presilting object $M$ of $\T$ and integers $k\le l$, we define a full subcategory $M^{[k,l]}$ of $\T$ as follows:
\[
M^{[k,l]}:=\add (M[k])\ast \add(M[k+1])\ast \cdots\ast \add(M[l-1])\ast \add(M[l]). \notag
\]
Note that, for integers $k,l,m,$ and $n$ satisfying $k\leq l$, $m\leq n$ and $m\leq l$, we have
\[
M^{[k,l]}\ast M^{[m,n]} \subseteq M^{[\min\{ k,m\}, \max\{ l,n\}]}. 
\]

\begin{lemma}\label{AI223}
Let $M$ be a silting object of $\T$.
Then the following hold.
\begin{itemize}
\item[(a)] For all integers $k\le l$, we have $M^{[k,l]}=\add(M^{[k,l]})$, \ie it is closed under direct summands.
\item[(b)] For an object $X\in M^{[k,l]}$ and an integer $k\leq m < l$, the following hold.
\begin{itemize}
\item[(1)] If $\Hom_{\T}(M^{[k,m]},X)=0$, then $X$ is in $M^{[m+1,l]}$.
\item[(2)] If $\Hom_{\T}(X,M^{[m+1,l]})=0$, then $X$ is in $M^{[k,m]}$.
\end{itemize}
\item[(c)]  We have $\T =\bigcup_{n\ge 0} M^{[-n,n]}$.
\end{itemize}
\end{lemma}
\begin{proof} 
(a) follows from  \cite[Lemma 2.6]{IYa14}. (b) follows from  \cite[Lemma 2.5]{IYa14}. 
(c) follows from \cite[Proposition 2.23]{AI12} and \cite[Proposition 2.8]{IYa14}.
\end{proof}

\begin{lemma}\label{twolemma}
\label{lem:partial-order-and-interval}
Let $M,N$ be silting objects of $\T$. 
Then $M\geq N$ if and only if $N\in M^{[0,n]}$ for some $n\gg 0$.

More generally, for integers $k\le l$, we have $M[k]\ge N \ge M[l]$ if and only if $N\in M^{[k,l]}$.
\end{lemma}
\begin{proof} 
The `if' part of the first statement  follows from the fact that the subcategory of objects $X$ of $\T$ satisfying $\Hom_\T(M,X[n])=0$ for all $n>0$ is closed under extensions, direct summands and $[1]$. To show the `only if' part, assume that $M\geq N$.  By Lemma~\ref{AI223}(c) there exists $n\gg0$ such that $N\in M^{[-n,n]}$. Assume $M\geq N$. Then $\Hom(M[<\hspace{-3pt}0],N)=0$, in particular, $\Hom(M^{[-n,-1]},N)=0$. By Lemma~\ref{AI223}(b), $N\in M^{[0,n]}$.

Observe that $N\in M^{[k,l]}$ if and only if $N\in (M[k])^{[0,n]}$ and $M[l]\in N^{[0,n]}$ for some $n\gg0$. The second statement follows from the first one applied to the pairs of presilting objects $(M[k],N)$ and $(N,M[l])$.
\end{proof}

As consequences of Lemma~\ref{lem:partial-order-and-interval}, we obtain the following symmetry for  $M^{[k,l]}$ and description of $\silt^n_{M}(\T)$. 
\begin{lemma}\label{lemma32}
Let $M$ and $N$ be silting objects of $\T$ and let $k\le l$ be integers. Then $N\in M^{[k,l]}$ if and only if $M\in N^{[-l,-k]}$.
\end{lemma}
\begin{proof}
This follows from Lemma~\ref{lem:partial-order-and-interval} because $M[k]\geq N\geq M[l]$ if and only if $N[-l]\geq M\geq N[-k]$.
\end{proof}

\begin{proposition}
Let $M,N$ be silting objects of $\T$ and let $n$ be a positive integer. Then $N\in \silt^n_{M}(\T)$ if and only if $N\in M^{[0,n-1]}$.
\end{proposition}

Next we introduce silting mutation. Assume that $\T$ is Hom-finite and Krull--Schmidt. Let $M$ be a silting object of $\T$.
For a decomposition $M=X\oplus Y$, we take a triangle
\[
X\xto{f}Y'\xto{}X^{\ast}\xto{}X[1] \notag
\]
with $f$ being a minimal left $\add(Y)$-approximation of $X$.
Then $\mu_{X}^{\Le}(M):=X^{\ast}\oplus Y$ is called the \emph{left mutation} of $M$ with respect to $X$. 
Dually, we define the \emph{right mutation} $\mu_{X}^{\Ri}(M)$. In this paper \emph{mutation} will mean left or right mutation.  If $X$ is indecomposable, then we call the mutation an \emph{irreducible mutation}. 
By \cite[Theorem 2.31]{AI12}, a mutation of a silting object is again a silting object. In fact, if $N$ is a left mutation of $M$, then $N\in \silt^2_{M}(\T)$ by \cite[Proposition 2.33]{AI12}.
Moreover, by \cite[Theorem 2.35]{AI12}, the Hasse quiver of the partially ordered set $\silt(\Lambda)$ coincides with the mutation quiver.

\smallskip
We end this subsection with the relation between silting objects and co-$t$-structures.

\begin{definition}[{\cite[Definition 2.4]{Pauksztello08}}]\label{defn:co-t-structure}
A \emph{co-$t$-structure} on
$\T$  (or \emph{weight structure} in~\cite{Bo10}) is a pair $(\T_{\geq 0},\T_{\leq
0})$ of strict and full subcategories of $\T$ such that, putting $\T_{\geq n}=\T_{\geq 0}[-n]$ and $\T_{\leq n}=\T_{\leq 0}[-n]$ for $n\in\mathbb{Z}$, we have
\begin{itemize}
\item[(0)] both $\T_{\geq 0}$ and $\T_{\leq
0}$ are additive and closed under taking direct summands;
\item[(1)] $\T_{\geq 1}\subseteq\T_{\geq 0}$ and
$\T_{\leq 0}\subseteq\T_{\leq 1}$;
\item[(2)] $\Hom(\T_{\geq 1}, \T_{\leq 0})=0$;
\item[(3)] $\T=\T_{\geq 1}\ast\T_{\leq 0}$.
\end{itemize}
The intersection $\T_{\geq
0}~\cap~\T_{\leq 0}$ is called the \emph{co-heart} of the co-$t$-structure $(\T_{\geq 0},\T_{\leq 0})$. 
\end{definition}
The conditions (0) and (2) can be replaced by the following condition:
\begin{itemize}
\item[($2'$)] $\tT_{\geq 1}={}^\perp(\T_{\leq 0})$ and $\tT_{\leq 0}=(\tT_{\geq 1})^{\perp}$.
\end{itemize}
A co-$t$-structure $(\T^{\leq 0},\T^{\geq
0})$ is said to be \emph{bounded}~\cite{Bo10} if
\[
\bigcup\nolimits_{n\in\mathbb{Z}}\T_{\leq
n}=\T=\bigcup\nolimits_{n\in\mathbb{Z}}\T_{\geq n}.
\]

\begin{example}
Let $\Lambda$ be a $K$-algebra. Let $\KK_{\geq 0}$ (respectively, $\KK_{\leq 0}$) be the full subcategory of $\Kb(\proj\Lambda)$ consisting of objects isomorphic to complexes with trivial components in negative degrees (respectively, in positive degrees). Then $(\KK_{\geq 0},\KK_{\leq 0})$ is a bounded co-$t$-structure on $\Kb(\proj\Lambda)$ with co-heart $\add(\Lambda)$ which is equivalent to $\proj\Lambda$.
\end{example}

The following result gives a bijection between silting objects and co-$t$-structures.

\begin{lemma}[{\cite[Theorem 4.10 (a)]{MSSS13} and \cite[Proposition 2.8]{IYa14}}]\label{l:co-heart}
Let $M$ be an object of $\T$. Then $M$ is a silting object if and only if $\add(M)$ is the co-heart of a bounded co-$t$-structure on $\T$.
\end{lemma}

Let $A$ be a dg $K$-algebra satisfying
\begin{itemize}
\item[(P1)] $H^p(A)=0$ for all $p<0$,
\item[(P2)] $H^0(A)$ is a finite-dimensional semisimple $K$-algebra.
\end{itemize}
By \cite[Corollary 4.7]{KN1}, there is a dg $A$-module $S$ (unique up to isomorphism) such that the graded $H^*(A)$-module $H^*(S)$ is isomorphic to $H^0(A)$. The following result appeared in an earlier version of \cite{KN1}.

\begin{lemma}\label{lem:silting-for-positive-dg-algebra}
Keep the notation and assumptions as in the preceding paragraph.
Then $S$ is a silting object of $\Dfd(A)$.
\end{lemma}
\begin{proof}
Let $\D_{\geq 0}$ (respectively, $\D_{\leq 0}$) be the full subcategory of $\D(A)$ consisting of objects $X$ such that $H^p(X)=0$ for $p<0$ (respectively, $p>0$). 
By \cite[Corollary 4.1]{KN1}, $(\D_{\geq 0},\D_{\leq 0})$ is a co-$t$-structure on $\D(A)$. It naturally restricts to a co-$t$-structure on $\Dfd(A)$ such that $S$ belongs to the co-heart and $\Dfd(A)=\thick(S)$, namely, $S$ is a silting object of $\Dfd(A)$.
\end{proof}

\subsection{Silting-discreteness}
 In this subsection we recall the notion of silting-discrete triangulated categories. 
Let $\T$ be a $K$-linear Hom-finite Krull--Schmidt triangulated category with shift functor $[1]$. Assume that $\T$ has silting objects.

\begin{definition}\label{def silting-discrete}
The triangulated category $\T$ is said to be \emph{silting-discrete} if, for any basic silting object $M$, the set $\silt^n_M(\T)$ is finite for any positive integer $n$.
\end{definition}

By \cite[Proposition 3.8]{Ai13}, $\T$ is silting-discrete if and only if for any fixed basic silting object $M$ of $\T$ the set $\silt^n_M(\T)$ is finite for any positive integer $n$. Moreover, if $\T$ is silting-discrete, then we can obtain all basic silting objects in $\T$ from any fixed basic silting object by a finite sequence of mutations (see \cite[Corollary 3.9]{Ai13}).
We have the following criterion for silting-discreteness.

\begin{lemma}[{\cite[Theorem 2.4]{AM15}}]\label{AM}
The category $\T$ is silting-discrete if and only if the set $\silt_M^{2}(\T)$ is finite for any basic silting object $M$ of $\T$.
\end{lemma}

We collect some examples of silting-discrete triangulated categories.
\begin{example}\label{ex:silting-discrete-derived-categories}
The bounded homotopy category $\Kb(\proj\Lambda)$ is silting-discrete if $\Lambda$ is one of the following finite-dimensional $K$-algebras, where $K$ is assumed to be algebraically closed for (3), (4), (5), (6): 
\begin{itemize}
\item[(1)] local algebras (see \cite[Theorem 2.26]{AI12}), 
\item[(2)] representation-finite hereditary algebras (see \cite[Example 3.7]{Ai13}),
\item[(3)] derived-discrete algebras of finite global dimension (see \cite[Proposition 6.12]{BPP2}),
\item[(4)] representation-finite symmetric algebras (see \cite[Theorem 5.6]{Ai13}),
\item[(5)] Brauer graph algebras whose Brauer graph contains at most one cycle of odd length and no cycle of even length (see \cite[Theorem 6.7]{AAC15}),
\item[(6)] algebras of dihedral, semidihedral and quaternion type (see \cite[Corollary 6.9]{EJR16}).
\end{itemize}
\end{example}

\subsection{$t$-structures}\label{ss:$t$-structures}
In this subsection, we recall the notion of $t$-structures, which was introduced by Beilinson--Bernstein--Deligne \cite{BBD81}.
 Let $\T$ be a $K$-linear triangulated category with shift functor $[1]$.

\begin{definition}\label{BBDdef}
A \emph{$t$-structure} on $\T$ is a pair $(\tT^{\le 0},\tT^{\ge 0})$ of strict and full subcategories of $\T$ such that, putting $\tT^{\le n}=\tT^{\le 0}[-n]$ and $\tT^{\ge n}=\tT^{\ge 0}[-n]$ for any integer $n$, we have
\begin{itemize}
\item[(1)] $\tT^{\le 1}\supseteq\tT^{\le 0}$ and $\tT^{\ge 0}\supseteq\tT^{\ge 1}$,
\item[(2)] $\Hom_{\T}(\tT^{\le 0},\tT^{\ge 1})=0$,
\item[(3)] $\T=\tT^{\le 0}\ast \tT^{\ge 1}$. 
\end{itemize}
The category $\T^0:=\T^{\leq 0}\cap\T^{\geq 0}$ is called the \emph{heart} of the $t$-structure $(\T^{\leq 0},\T^{\geq 0})$.
\end{definition}

It is easy to see that for every integer $n$, the pair $(\tT^{\le n},\tT^{\ge n})$ is also a $t$-structure and the category $\tT^{n}:=\tT^{\le n}\cap \tT^{\ge n}$ is the heart. Note that $\tT^{\le 0}$ and $\tT^{\ge 0}$ are additive subcategories which are closed under extensions and direct summands.
The condition (2) in Definition \ref{BBDdef} can be replaced by the following condition:
\begin{itemize}
\item[($2'$)] $\tT^{\le 0}={}^{\perp}(\tT^{\ge 1})$ and $\tT^{\ge 1}=(\tT^{\le 0})^{\perp}$.
\end{itemize}
By the condition (3) of Definition \ref{BBDdef}, for $Z\in\T$ there is a triangle $X\to Z\to Y\to X[1]$ with $X\in\T^{\leq 0}$ and $Y\in\T^{\leq 1}$. This triangle is unique up to a unique isomorphism, so the correspondences $Z\mapsto X$ and $Z\mapsto Y$ extend to functors 
\[
\sigma^{\le 0}\colon  \T \to \tT^{\le 0}\ \textnormal{and}\ \sigma^{\ge 1}\colon  \T \to \tT^{\ge 1},\ \textnormal{respectively}, \notag
\]
called the \emph{truncation functors}. 
The functor $\sigma^{\le 0}$ is right adjoint to the inclusion $\tT^{\le 0}\to \T$
and the functor $\sigma^{\ge 1}$ is left adjoint to the inclusion $\tT^{\ge 1}\to \T$.

In the following proposition, we collect some basic results on $t$-structures.
\begin{proposition}[{\cite[Th\'eor\`eme 1.3.6]{BBD81}}]\label{BBD}
Let $(\tT^{\le 0},\tT^{\ge 0})$ be a $t$-structure on $\T$.
Then the following statements hold.
\begin{itemize}
\item[(a)] The heart $\tT^{0}$ is an abelian category.
\item[(b)] The \emph{cohomology functor} $\sigma^{0}:=\sigma^{\ge 0}\sigma^{\le 0}=\sigma^{\le 0}\sigma^{\ge 0}\colon  \T \to \tT^{0}$ is a cohomological functor, that is, it takes triangles in $\T$ to long exact sequences in $\T^0$.
\item[(c)] For objects $X,Y,Z$ of $\tT^{0}$, the following are equivalent.
\begin{itemize}
\item[(i)] $0\to X\xto{f} Y\xto{g} Z\to 0$ is a short exact sequence in $\tT^{0}$.
\item[(ii)] $X\xto{f} Y \xto{g} Z\xto{h}X[1]$ is a triangle in $\T$ for some morphism $h\colon Z\to X[1]$.
\end{itemize}
\end{itemize}
\end{proposition}

\begin{example}[{\cite[Exemple 1.3.2(i)]{BBD81}}]\label{ex:standard-t-structure}
Let $\ca$ be a $K$-linear abelian category and $\Db(\ca)$ the bounded derived category of $\ca$. Put
\begin{eqnarray*}
\D_{\std}^{\leq 0}\hspace{-5pt}&:=\hspace{-5pt}&\{X\in\Db(\ca)\mid H^n(X)=0~\forall n>0\},\\
\D_{\std}^{\geq 0}\hspace{-5pt}&:=\hspace{-5pt}&\{X\in\Db(\ca)\mid H^n(X)=0~\forall n<0\}.
\end{eqnarray*}
Then $(\D_{\std}^{\leq 0},\D_{\std}^{\geq 0})$ is a $t$-structure on $\Db(\ca)$, called the \emph{standard $t$-structure}.
The associated truncation functors are the standard truncations of complexes and the associated cohomology functor is the functor $H^0\colon \Db(\ca)\to\ca$ of taking the zeroth cohomology, which restricts to an equivalence between the heart $\D_{\std}^0$ and $\ca$.
\end{example}

\begin{example}\label{ex:t-structure-for-non-positive-dg-algebra}
Let $A$ be a dg $K$-algebra with $H^p(A)=0$ for $p>0$. Let $\D^{\leq 0}$ (respectively, $\D^{\geq 0}$) be the full subcategory of $\D(A)$  consisting of the dg $A$-modules $M$ with $H^p(M)=0$ for $p>0$ (respectively, for $p<0$). Then $(\D^{\leq 0},\D^{\geq 0})$ is a $t$-structure on $\D(A)$. The associated truncation functors are standard truncations and $H^0\colon \D(A)\to \Mod H^0(A)$ restricts to an equivalence between the heart and $\Mod H^0(A)$, see \cite[Section 2.1]{Am09}. 

Let $\Dfd^-(A)$ be the full subcategory of $\D(A)$ consisting of dg $A$-modules $M$ with $H^p(M)=0$ for $p>\hspace{-5pt}>0$. Then $\Dfd^-(A)$ is stable under standard truncations, so $(\D^{\leq 0},\D^{\geq 0})$ induces a $t$-structure $(\Dfd^{-,\leq 0},\Dfd^{-,\geq 0})$  on $\Dfd^-(A)$ by \cite[3.1.19]{BBD81}, where $\Dfd^{-,\leq 0}=\D^{\leq 0}\cap \Dfd^-(A)$ and $\Dfd^{-,\geq 0}=\D^{\geq 0}\cap \Dfd^-(A)$.
\end{example}

A $t$-structure $(\tT^{\le 0}, \tT^{\ge 0})$ is said to be \emph{bounded} if 
\[
\displaystyle \T=\bigcup\nolimits_{n\in \ZZ} \tT^{\le n}=\bigcup\nolimits_{n\in \ZZ}\tT^{\ge n}.
\]
For example, the standard $t$-structure in Example~\ref{ex:standard-t-structure} is a bounded $t$-structure, while the two $t$-structures in Example~\ref{ex:t-structure-for-non-positive-dg-algebra} are not bounded. We denote by $\tstr(\T)$ the set of bounded $t$-structures on $\T$.
For two $t$-structures $(\tT^{\le 0}, \tT^{\ge 0})$ and $(\tT'^{\le 0}, \tT'^{\ge 0})$, we write $(\tT^{\le 0}, \tT^{\ge 0})\ge (\tT'^{\le 0}, \tT'^{\ge 0})$ if $\tT^{\le 0} \supseteq \tT'^{\le 0}$, or equivalently $\tT^{\ge 0}\subseteq \tT'^{\ge 0}$. Then $\ge$ gives a partial order on $\tstr(\T)$. 
The following results on bounded $t$-structures are well-known. 

\begin{lemma}\label{KY33}\label{lem:bounded-$t$-structure}
Let $(\tT^{\le 0},\tT^{\ge 0})$ be a $t$-structure on $\T$ with heart $\tT^{0}$.
\begin{itemize}
\item[(a)]
$(\tT^{\le 0},\tT^{\ge 0})$ is bounded if and only if $\T=\thick(\tT^{0})$.
\item[(b)] {\rm (\cite[(1.3.13.1)]{BBD81})}
If $(\tT^{\le 0},\tT^{\ge 0})$ is bounded, there is the following decompositions
\begin{eqnarray*}
\T\hspace{-5pt}&=\hspace{-5pt}&\bigcup\nolimits_{n\ge 0}\tT^{-n}\ast\tT^{-n+1}\ast\cdots\ast\tT^{n-1}\ast\tT^{n},\\
\tT^{\le 0}\hspace{-5pt}&=\hspace{-5pt}&\bigcup\nolimits_{n\ge 0}\tT^{-n}\ast\tT^{-n+1}\ast\cdots\ast\tT^{0},\\
\tT^{\geq 0}\hspace{-5pt}&=\hspace{-5pt}&\bigcup\nolimits_{n\ge 0}\tT^{0}\ast\tT^{1}\ast\cdots\ast\tT^{n}.
\end{eqnarray*}
\item[(c)] 
Let $(\T'^{\leq 0},\T'^{\geq 0})$ be a $t$-structure on $\T$ such that $\T^{\leq 0}\supseteq \T'^{\leq 0}\supseteq \T^{\leq -n}$ for some positive integer $n$. Then $(\T'^{\leq 0},\T'^{\geq 0})$ is bounded if and only if $(\T^{\leq 0},\T^{\geq 0})$ is bounded.
\end{itemize}
\end{lemma}

A bounded $t$-structure $(\tT^{\le 0}, \tT^{\ge 0})$ is said to be \emph{algebraic} if the heart $\T^0$ is a length category  (\ie objects in $\T^0$ admit finite filtrations) with finitely many isomorphism classes of simple objects. For example, for $\ca=\mod\Lambda$, where $\Lambda$ is a finite-dimensional $K$-algebra, the standard $t$-structure on $\Db(\ca)$ in Example~\ref{ex:standard-t-structure} is an algebraic $t$-structure. 

\begin{example}\label{ex:algebraic-t-structure}
A bounded $t$-structure is algebraic provided that its heart has finitely many isomorphism classes of indecomposable objects. For example, all bounded $t$-structures on $\Db(\mod\Lambda)$ are algebraic for the following finite-dimensional $K$-algebra $\Lambda$:
\begin{itemize}
\item[(1)] $\Lambda$ is representation-finite hereditary,
\item[(2)] $K$ is algebraically closed and $\Lambda$ is a finite-dimensional derived-discrete $K$-algebra of finite global dimension (\cite[Proposition 7.1]{BPP1}).
\end{itemize}
\end{example}

\begin{lemma}\label{lem:intermediate-t-structures-wrt-algebraic-t-structure}
Let $(\T^{\leq 0},\T^{\geq 0})$ and $(\T'^{\leq 0},\T'^{\geq 0})$ be $t$-structures on $\T$. Assume that $(\T^{\leq 0},\T^{\geq 0})$ is algebraic. Then $(\T'^{\leq 0},\T'^{\geq 0})$ is bounded if and only if there are integers $m>n$ such that $\T^{\leq m}\supseteq \T'^{\leq 0}\supseteq \T^{\leq n}$.
\end{lemma}
\begin{proof} 
Let $S$ be the direct sum of a complete set of pairwise non-isomorphic simple objects in $\T^0$. Then by Lemma~\ref{KY33}(b), there exist integers $m>n$ such that $S\in \T'^{\geq -m}\cap \T'^{\leq -n}$. Hence $\T^{0}\subseteq \T'^{\geq -m}\cap \T'^{\leq -n}$ because objects of $\T^{0}$ are iterated extensions of direct summands of $S$.
By Lemma \ref{KY33}(b) again, $\T^{\le 0} \subseteq \T'^{\le -n}$ and $\T^{\ge 0}\subseteq \T'^{\ge -m}$ hold, and hence $\T^{\le m}\supseteq \T'^{\le 0}\supseteq \T^{\le n}$. 
\end{proof}

When $\T$ is Hom-finite and Krull--Schmidt, we are particularly interested in the class of bounded $t$-structures $(\T^{\leq 0},\T^{\geq 0})$ whose heart $\T^0$ admits a projective generator $P$. By Morita's theorem, in this case the functor $\Hom_{\T^0}(P,-)\colon \T^0\to\mod\End_{\T^0}(P)$ is an equivalence, in particular,  $(\T^{\leq 0},\T^{\geq 0})$ is algebraic. 

\subsection{$t$-discreteness}

In this subsection we introduce the notion of $t$-discrete triangulated categories. 
Let $\T$ be a $K$-linear triangulated category with shift functor $[1]$.

\begin{definition}\label{defn:t-discreteness}
The triangulated category $\T$ is said to be \emph{t-discrete} if for any bounded $t$-structure $(\T^{\leq 0},\T^{\geq 0})$ and for any positive integer $n$, the number of $t$-structures $(\T'^{\leq 0},\T'^{\geq 0})$ satisfying $\T^{\leq 0}\supseteq \T'^{\leq 0}\supseteq\T^{\leq -n}$ is finite.
\end{definition}

\begin{example}
Let $\Lambda$ be a representation-finite finite-dimensional hereditary $K$-algebra. Then $\Db(\mod\Lambda)$ is $t$-discrete. Indeed, let  $(\T^{\leq 0},\T^{\geq 0})$ be any fixed bounded $t$-structure on $\Db(\mod\Lambda)$ and let $n$ be any positive integer. Then associating a $t$-structure to its heart defines an injection from the set of  $t$-structures $(\T'^{\leq 0},\T'^{\geq 0})$ satisfying $\T^{\leq 0}\supseteq \T'^{\leq 0}\supseteq\T^{\leq -n}$ to the set of strict and full Krull--Schmidt additive subcategories of $\T^{\leq 0}\cap \T^{\geq -n}$, which is a finite set simply because the category $\T^{\leq 0}\cap \T^{\geq -n}$ has only finitely many isomorphism classes of indecomposable objects.
\end{example}

The following lemma shows that $t$-discreteness can be verified by fixing an algebraic $t$-structure $(\T^{\leq 0},\T^{\geq 0})$ and checking the required finiteness condition.

\begin{lemma}\label{lem:criterion-for-t-discreteness}
Let $(\T^{\leq 0},\T^{\geq 0})$ be an algebraic $t$-structure on $\T$. Then $\T$ is t-discrete if and only if for any positive integer $n$, the number of $t$-structures $(\T'^{\leq 0},\T'^{\geq 0})$ satisfying $\T^{\leq 0}\supseteq \T'^{\leq 0}\supseteq\T^{\leq -n}$ is finite.
\end{lemma}
\begin{proof} Let $(\T''^{\leq 0},\T''^{\geq 0})$ be any bounded $t$-structure on $\T$. By Lemma~\ref{lem:intermediate-t-structures-wrt-algebraic-t-structure}, there are integers $k>l$ such that
 $\T^{\le k}\supseteq \T''^{\le 0}\supseteq \T^{\le l}$. As a consequence, any $t$-structure $(\T'^{\leq 0},\T'^{\geq 0})$ satisfying $\T''^{\leq 0}\supseteq \T'^{\leq 0}\supseteq\T''^{\leq -n}$ must satisfy $\T^{\leq k}\supseteq\T'^{\leq 0}\supseteq\T^{-n+l}$. The set of such $t$-structures, by assumption, is finite.
\end{proof}

\subsection{Torsion classes}
Let $\A$ be a $K$-linear abelian category.
A full subcategory $\X$ of $\A$ is called a \emph{torsion class} if, for each $Z\in \A$, there exists an exact sequence 
$0\to X\to Z\to Y\to 0$
with $X\in \X$ and $Y\in \X^{\perp}$. It is said to be \emph{finitely generated} if there exists an object $X\in \A$ such that $\X=\Fac(X)$, the full subcategory of $\A$ consisting of objects $M$ such that there is an epimorphism $X'\to M$ with $X'\in\add(X)$. For example, if $T$ is a support $\tau$-tilting module over a finite-dimensional $K$-algebra $\Lambda$, then $\Fac(T)$ is a finitely generated torsion class of $\mod\Lambda$ (\cite{AIR14}); for another example, torsion sheaves form a torsion class of $\coh(\mathbb{P}^1)$, which is not finitely generated. We denote by $\tors(\A)$ the set of torsion classes of $\A$ and by $\ftors(\A)$ the set of finitely generated torsion classes of $\A$. 

\subsubsection{The relation between torsion classes and intermediate $t$-structures} Let $\T$ be a $K$-linear triangulated category with shift functor $[1]$. 
Let $(\T^{\le 0},\T^{\ge 0})$ be a bounded $t$-structure on $\T$. For a torsion class $\X$ of the heart $\T^{0}$, we define two full subcategories of $\T$ as follows:
\[
\mu_{\X}^{\Le}\tT^{\le 0}:=\tT^{\le -1}\ast \X,\ \ \mu_{\X}^{\Le}\T^{\ge 0}:=\X^{\perp_{\T^{0}}}[1]\ast \T^{\ge 0}.\notag
\]
By \cite[Chapter 1, Proposition 2.1]{HRS96}, the pair $(\mu_{\X}^{\Le}\tT^{\le 0}, \mu_{\X}^{\Le}\tT^{\ge 0})$ is also a bounded $t$-structure on $\T$ and its heart is
$
\mu_{\X}^{\Le}\tT^{\le 0}\cap\mu_{\X}^{\Le}\tT^{\ge 0}=\X^{\perp_{\tT^{0}}}[1]\ast \X.$
We call this $t$-structure the \emph{HRS-tilt} of $(\tT^{\le 0}, \tT^{\ge 0})$ with respect to $\X$. Clearly, we have
$\tT^{\le 0}\supseteq \mu_{\X}^{\Le}\tT^{\le 0}\supseteq \tT^{\le -1}. $
Conversely, let $(\tT'^{\le 0},\tT'^{\ge 0})$ be a $t$-structure satisfying $\tT^{\le 0}\supseteq \tT'^{\le 0}\supseteq \tT^{\le -1}$. Note that such a $t$-structure is always bounded by Lemma \ref{lem:bounded-$t$-structure}(c). Then the subcategory $\tT'^{0}\cap \tT^{0}$ is a torsion class of $\tT^{0}$ by \cite[Theorem 3.1]{BR07}.

\begin{proposition}[{\cite[Proposition 2.1]{W10}}]\label{prop:woolf's-proposition}
Fix a bounded $t$-structure $(\tT^{\le 0},\tT^{\ge 0})$ and let $\tstr^2(\T)$ denote the set of $t$-structures $(\tT'^{\le 0},\tT'^{\ge 0})$ satisfying $\tT^{\le 0}\supseteq \tT'^{\le 0}\supseteq \tT^{\le -1}$.
Then there are mutually inverse bijections
\[
\xymatrix@C=3pc{\tstr^2(\T)\ar@/^1mm/[r]^{\ \Phi}&\tors(\T^{0})\ar@/^1mm/[l]^{\ \ \Phi'}}, \notag
\]
where $\Phi\colon  (\T'^{\le 0}, \T'^{\ge 0})\mapsto \T'^{0}\cap \T^{0}$ and $\Phi'\colon  \X\mapsto (\mu_{\X}^{\Le}\tT^{\le 0}, \mu_{\X}^{\Le}\tT^{\ge 0})$.
\end{proposition}

\subsubsection{The relation between torsion classes and 2-term silting objects}
Let $\T$ be a $K$-linear Hom-finite Krull--Schmidt triangulated category with shift functor $[1]$. 

\begin{proposition}\label{twosilt-fgtor}
Let $M$ be a basic silting object of $\T$ and $E=\End_\T(M)$. 
\begin{itemize}
\item[(a)]
There is a bijection 
\[
\silt^2_M(\T) \longrightarrow \ftors(\mod E),\  N\mapsto \Fac(\Hom_\T(M,N)). \]
\item[(b)] The following are equivalent.
\begin{itemize}
\item[(i)] The set $\silt^2_M(\T)$ is finite.
\item[(ii)] All torsion classes of $\mod E$ are finitely generated.
\item[(iii)] The set $\tors(\mod E)$ is finite.
\item[(iv)] $E$ is $\tau$-tilting finite.
\end{itemize}
\end{itemize}
\end{proposition}
\begin{proof}
By \cite[Proposition 4.6]{AS80}, a torsion class in $\mod E$  is finitely generated if and only if it is both contravariantly finite and covariantly finite.
(a) follows from {\cite[Theorem 0.5]{AIR14} and \cite[Theorems 0.2 and 0.3]{IJY14}}.
(b)  follows from (a) and \cite[Corollary 2.9 and Theorem 3.8]{DIJ15}.
\end{proof}


\section{ST-pairs: examples and basic properties}
\label{s:ST-pairs-examples}

In this section we introduce the notion of ST-pair of triangulated categories, give motivating examples and study basic properties. 

\medskip
Throughout this section, let $\T$ be a $K$-linear idempotent complete triangulated category (not necessary Hom-finite) with shift functor $[1]$. 

\subsection{Subcategories orthogonal to a presilting object}
\label{ss:categories-associated-to-presilting-object}

For a presilting object $M$ of $\T$, a thick subcategory $\S$ of $\T$ and an integer $n$, we define full subcategories of $\S$ as follows:
\begin{eqnarray*}
\S_{M}^{\le n}\hspace{-5pt}&:=\hspace{-5pt}&\{ X\in \S \mid \Hom_{\T}(M,X[>\hspace{-3pt}n])=0 \},\\
\S_{M}^{\ge n}\hspace{-5pt}&:=\hspace{-5pt}&\{ X\in \S \mid \Hom_{\T}(M,X[<\hspace{-3pt}n])=0 \}, \\
\S_{M}^{n}\hspace{-5pt}&:=\hspace{-5pt}&\S_{M}^{\le n}\cap \S_{M}^{\ge n}. 
\end{eqnarray*}
Here, $\Hom_{\T}(M,X[>\hspace{-3pt}n])=0$ means $\Hom_{\T}(M,X[p])=0$ for all integers $p>n$, and $\Hom_{\T}(M,X[<\hspace{-3pt}n])=0$ is defined similarly. It immediately follows from definition that 
\begin{itemize}
\item[(1)] $\S_M^{\leq n}=\T_M^{\leq n}\cap \S$, $\S_M^{\geq n}=\T_M^{\geq n}\cap\S$ and $\S_M^n=\T_M^n\cap\S$,
\item[(2)] $\S_{M}^{\le n}=\S_{M}^{\le 0}[-n]=S_{M[-n]}^{\leq 0}$ and $\S_{M}^{\ge n}=\S_{M}^{\ge 0}[-n]=\S_{M[-n]}^{\geq 0}$,
\item[(3)] $\S_M^{\leq n}\subseteq \S_M^{\leq n+1}$ and $\S_M^{\geq n}\supseteq\S_M^{\geq n+1}$,
\item[(4)] both $\S_{M}^{\le n}$ and $\S_{M}^{\ge n}$ are extension-closed.

\end{itemize}

The following lemma describes the relationship between partial orders of silting objects and $t$-structures. 
\begin{lemma}\label{lem:partial-orders-of-presilting-and-t-structure}
Let $M,N$ be presilting objects of $\T$ satisfying $\thick(M)=\thick(N)$. 
Then $M\geq N$ if and only if $\T_M^{\leq 0}\supseteq \T_N^{\leq 0}$.

More generally, for integers $k\le l$, we have $M[k]\ge N \ge M[l]$ if and only if $\tT_{M}^{\le -k}\supseteq \tT_{N}^{\le 0}\supseteq \tT_{M}^{\le -l}$.
\end{lemma}
\begin{proof} 
Assume $\T_M^{\leq 0}\supseteq \T_N^{\leq 0}$. Because $N$ is presilting, it belongs to $\T_N^{\leq 0}$, and hence to $\T_M^{\leq 0}$. This implies that $M\geq N$.
Conversely, assume $M\geq N$. By Lemma~\ref{lem:partial-order-and-interval}, there exists $n\in\mathbb{N}$ such that $N\in M^{[0,n]}$. For any $X\in\T_N^{\leq 0}$, we have $\Hom(N^{[-n,0]},X[>\hspace{-3pt}0])=0$. By Lemma~\ref{lemma32}, $M\in N^{[-n,0]}$ and hence $\Hom(M,X[>\hspace{-3pt}0])=0$, \ie $X\in\T_M^{\leq 0}$.

Recall from the proof of Lemma~\ref{lem:partial-order-and-interval} that $N\in M^{[k,l]}$ if and only if $N\in (M[k])^{[0,n]}$ and $M[l]\in N^{[0,n]}$ for some $n\gg0$ and recall that $\T_{M[k]}^{\leq 0}=\T_M^{\leq -k}$. The second statement follows from the first statement applied to the pairs of presilting objects $(M[k],N)$ and $(N,M[l])$.
\end{proof}

\begin{lemma}\label{triangle1}
\label{lem:co-t-structure-for-presilting}
Let $M$ be a presilting object of $\T$ such that $\add(M)$ is contravariantly finite in $\T$.  For any integer $l>0$, the following equality holds: 
\[
\tT_{M}^{\le 0}=M^{[0,l-1]}\ast \tT_{M}^{\le -l}.\notag
\]
\end{lemma}
\begin{proof}
Since $\T^{\leq 0}_M$ is extension-closed and contains both $M^{[0,l-1]}$ and $\T^{\leq -l}_{M}$, we have $\tT_{M}^{\le 0}\supseteq M^{[0,l-1]}\ast \tT_{M}^{\le -l}$.
Fix any $X\in \tT_{M}^{\le 0}$. Let $f\colon M'\to X$ be a right $\add(M)$-approximation and form a triangle $M'\xto{f}X\to X'\to M'[1]$.
Applying $\Hom_{\T}(M,-)$ to this triangle, we obtain an exact sequence
\[
\Hom_{\T}(M,M'[i])\xto{\Hom(M,f[i])} \Hom_{\T}(M,X[i])\to \Hom_{\T}(M,X'[i])\to \Hom_{\T}(M,M'[i+1]).
\]
Thus $\Hom_{\T}(M,X'[>\hspace{-3pt}-1])=0$, \ie $X'\in \tT_{M}^{\le -1}$. So $X\in M'\ast X' \subseteq\add(M)\ast \tT_{M}^{\le -1}$, and hence $\tT_{M}^{\le 0}\subseteq \add(M)\ast \tT_{M}^{\le -1}$. Inductively, we have $\tT_{M}^{\le 0}\subseteq M^{[0,l-1]}\ast \tT_{M}^{\le -l}$.
\end{proof}

\subsection{ST-pairs}
\label{ss:ST-pair}

\begin{definition}\label{WST}\label{defn:ST-pair}
Let $\C$ and $\D$ be thick subcategories of $\T$.
The pair $(\C,\D)$ is called an \emph{ST-pair} inside $\T$ if there exists a silting object $M$ of $\C$ such that
\begin{itemize}
\item[(ST1)]  $\Hom_{\T}(M,T)$ is finite-dimensional for any object $T$ of $\T$, 
\item[(ST2)] $(\tT_{M}^{\le 0}, \tT_{M}^{\ge 0})$ is a $t$-structure on $\T$,
\item[(ST3)] $\T=\bigcup_{n\in\mathbb{Z}}\T_M^{\leq n}$ and $\D=\bigcup_{n\in\mathbb{Z}}\T_M^{\geq n}$.
\end{itemize}
When there is a need to emphasise the silting object $M$, we call the triple $(\C,\D, M)$ an \emph{ST-triple}. We will omit the ambient category $\T$ when it is not relevant.
\end{definition}

A prototypical example is the ST-pair $(\Kb(\proj\Lambda),\Db(\mod\Lambda))$, where $\Lambda$ is a finite-dimensional algebra (Lemma~\ref{lem:ST-pair-for-fd-algebra}). More motivating examples will be given in Sections~\ref{ss:ST-pair-for-fd-algebra} and \ref{ss:ST-pair-for-smooth-non-positive-dg-algebra}.   We remark that all the following four possibilities occur: (1) $\C=\D$ (Sections~\ref{ss:ST-pair-for-fd-algebra} and~\ref{ss:ST-pair-for-smooth-non-positive-dg-algebra}), (2) $\C\subsetneq\D$ (Section~\ref{ss:ST-pair-for-fd-algebra}), (3) $\C\supsetneq\D$ (Section~\ref{ss:ST-pair-for-smooth-non-positive-dg-algebra}), (4) $\C$ and $\D$ are not comparable (Section~\ref{ss:completing-silting-to-ST-pair}).

We give two remarks on the conditions (ST1) and (ST3).

\begin{remark}\label{rem111}\label{rem:ST1}
\begin{itemize}
\item[(a)] If $\T$ is Hom-finite, then (ST1) is always satisfied for any silting object $M$.
\item[(b)] The objects $X$ of $\T$ such that $\Hom_\T(X,T)$ is finite-dimensional for any $T\in\T$  form a thick subcategory of $\T$. As a consequence, if one silting object $M$ of $\C$ satisfies the condition (ST1), so does any silting object of $\C$. 
\item[(c)] (ST1) holds for $M$ if and only if $\Hom_\T(M,M[p])$ is finite-dimensional for all $p\in\mathbb{Z}$ and $\add(M)$ is contravariantly finite in $\T$. 
\item[(d)] By (ST1), $\C$ is Hom-finite. Therefore $\C$ is Krull--Schmidt since it is idempotent complete (Lemma~\ref{lem:KS-vs-idempotent-complete}). Similarly, $\D$ is Hom-finite and Krull--Schmidt by (ST1') in Proposition~\ref{prop:ST=>TS}. 
\end{itemize}
\end{remark}

\begin{remark}\label{rem:ST2}
Under (ST2), the following conditions are equivalent:
\begin{itemize}
\item[(i)] the condition (ST3),
\item[(ii)] $\tT_{M}^{\ge 0}\subseteq \D$ and $\D=\thick(\tT_{M}^{0})$,
\item[(iii)] $\tT_{M}^{\ge 0}\subseteq \D$ and $(\D_M^{\leq 0},\D_M^{\geq 0})$ is a bounded $t$-structure on $\D$.
\end{itemize}
\smallskip
{\it Proof: }(i)$\Rightarrow$(iii):  $(\D_M^{\leq 0},\D_M^{\geq 0})$ is the induced $t$-structure in the sense of \cite[1.3.19]{BBD81}. The verification of boundedness is straightforward.

(iii)$\Rightarrow$(i): The condition $\T_M^{\geq 0}\subseteq \D$ implies that $\T_M^{\geq n}\subseteq \D$ for all $n\in\mathbb{Z}$, so $\T_M^{\geq n}=\D_M^{\geq n}$. Therefore $\D=\bigcup_{n\in\mathbb{Z}}\D_M^{\geq n}=\bigcup_{n\in\mathbb{Z}}\T_M^{\geq n}$. To show the other equality, take $X\in\T$ and form a triangle $X'\to X\to X''\to X[1]$ with $X'\in \T_M^{\leq 0}$ and $X''\in\T_M^{\geq 1}$. As $X''\in\D$, there is a positive integer $n$ such that $X''\in\D_M^{\leq n}$. So $X\in \T_M^{\leq 0}*\D_M^{\leq n}\subseteq\T_M^{\leq n}$. Therefore $\T=\bigcup_{n\in\mathbb{Z}}\T_M^{\leq n}$.

(ii)$\Leftrightarrow$(iii): Assume $\T_M^{\geq 0}\subseteq \D$. Then 
\[
\T_M^0=\T_M^{\leq 0}\cap\D\cap\T_M^{\geq 0}=\D_M^{\leq 0}\cap\D_M^{\geq 0}=\D_M^0.
\]
The equivalence (ii)$\Leftrightarrow$(iii) follows by Lemma~\ref{lem:bounded-$t$-structure}(a).
\end{remark}

In the rest of this subsection, we fix an ST-triple $(\C,\D,M)$ inside $\T$. 

\begin{proposition}\label{prop:heart-of-silting-t-structure}
\begin{itemize}
\item[(a)] Let $\sigma_{M}^{0}$ be the  cohomology functor associated with the $t$-structure $(\T_M^{\leq 0},\T_M^{\geq 0})$. Then the object $\sigma_{M}^{0}(M)$ is a projective generator of the heart $\tT^{0}_{M}$, and $\Hom_\T(M,-)$ restricts to an equivalence from $\tT_{M}^{0}$ to $\mod \End_{\T}(M)$.  
\item[(b)] $(\tD_{M}^{\le 0},\tD_{M}^{\ge 0})$ is a bounded $t$-structure on $\D$ and $\tD_{M}^{0}=\tT_{M}^{0}$. 
\item[(c)] $(\T_M^{\leq 0},\T_M^{\geq 0})$ is a non-degenerated $t$-structure on $\T$, \ie $\bigcap_{n\in\mathbb{Z}}\T_M^{\leq n}=0=\bigcap_{n\in\mathbb{Z}}\T_M^{\geq n}$. 
\end{itemize}
\end{proposition}
\begin{proof} (b) was shown in Remark~\ref{rem:ST2}. By (ST2), $M$ is a silting object in the sense of \cite[Definition 4.1]{PV15}. 
Thus (c) follows from \cite[Proposition 4.3]{PV15}.
Let us prove (a).  By \cite[Proposition 4.3]{PV15}, $\sigma^0_M(M)$ is projective in $\T_M^0$. Let $X\in\T_M^0$. By Lemma~\ref{lem:co-t-structure-for-presilting}, there is a triangle $M'\stackrel{f}{\to} X\to Y \to M'[1]$ with $M'\in\add(M)$ and $Y\in\T_M^{\leq -1}$. Then $\sigma^0_M(f)\colon \sigma^0_M(M')\to X$ is an epimorphism, and it follows that $\sigma^0_M(M)$ is a projective generator of $\T_M^0$. By Morita's theorem, the functor $\Hom_{\tT_{M}^{0}}(\sigma_{M}^{0}(M),-)\colon  \tT_{M}^{0}\to \mod\End_{\T}(\sigma_{M}^{0}(M))$ is an equivalence of abelian categories. By Lemma~\ref{lem:truncation-preserves-morphisms} below we have an isomorphism $\End_{\T}(\sigma_{M}^{0}(M))\simeq \End_{\T}(M)$ of algebras and an isomorphism $\Hom_{\T_{M}^{0}}(\sigma_{M}^{0}(M),-)\simeq \Hom_{\T}(M,-)|_{\tT_{M}^{0}}$ of functors. Therefore the functor $\Hom_\T(M,-)|_{\T_M^0}\colon \T_M^0\to\mod\End_\T(M)$ is an equivalence.
\end{proof}

\begin{lemma}\label{lem:truncation-preserves-morphisms}
Let $X\in\T_M^{\leq 0}$ and let $g\colon X\to\sigma^0_M(X)$ be the canonical morphism. Then for $M'\in\add(M)$ and $Y\in\T_M^0$ the maps
$\Hom_\T(M',g)\colon\Hom_\T(M',X)\to\Hom_\T(M',\sigma^0_M(X))$  and $\Hom_\T(g,Y)\colon\Hom_\T(\sigma^0_M(X),Y)\to\Hom_\T(X,Y)$ are functorial isomorphisms.
\end{lemma} 
\begin{proof}
Applying $\Hom_\T(M',-)$ to the triangle $\sigma^{\leq -1}_M(X)\to X\stackrel{g}{\to} \sigma^0_M(X)\to(\sigma^{\leq -1}_M(X))[1]$, we obtain an exact sequence
\[{{
\Hom_\T(M',\sigma^{\leq -1}_M(X))\to \Hom_\T(M',X){\rightarrow} \Hom_\T(M',\sigma^0_M(X))\to \Hom_\T(M',(\sigma^{\leq -1}_M(X))[1]).
}}\]
The two outer terms vanish because both $\sigma^{\leq -1}_M(X)$ and $(\sigma^{\leq -1}_M(X))[1]$ belong to $\T_M^{\leq -1}$. This implies that $\Hom(M',g)$ is an isomorphism. That $\Hom(g,Y)$ is an isomorphism can be shown similarly.
\end{proof}

It turns out that the indecomposable direct summands of $M$ and simple objects of $\D_M^0$ satisfy a Hom-duality, as satisfied by indecomposable projective modules and simple modules over a finite-dimensional algebra.

\begin{corollary}\label{cor:Hom-duality-between-silting-and-smc}
Assume that $M$ is basic and fix a decomposition $M=M_1\oplus\ldots\oplus M_n$ with $M_1,\ldots,M_n$ indecomposable. Let $S_i$ be the simple top of $\sigma^0_M(M_i)$. Then $\Hom_\T(M_i,S_j[p])=0$ unless $i=j$ and $p=0$, in which case it is isomorphic to $\End(S_i)$.
\end{corollary}
\begin{proof}
For $p\neq 0$, the space $\Hom_\T(M_i,S_j[p])$ vanishes because $S_j$ belongs to $\T_M^0$. By Lemma~\ref{lem:truncation-preserves-morphisms}, the space $\Hom_\T(M_i,S_j)$ is isomorphic to $\Hom_\T(\sigma^0_M(M_i),S_j)$, which vanishes for $i\neq j$ and is $\End(S_i)$ if $i=j$.
\end{proof}

The following stronger version of Lemma~\ref{lem:co-t-structure-for-presilting} shows that `minimal $M$-resolution' exists in $\T$. Let $S$ be the direct sum of a complete set of pairwise non-isomorphic simple objects of $\T_M^0$.

\begin{proposition}\label{prop:minimal-resolution}
Let $l>0$ and $X\in\T_M^{\leq 0}$. Then there exist $M^0,M^{-1},\ldots,M^{-l+1}\in\add(M)$ and $Y\in\T_M^{\leq -l}$ such that $X\in M^0\ast M^{-1}[1]\ast\cdots\ast M^{-l+1}[l-1]\ast Y$ and 
\[
\Hom_\T(X,S[p])\simeq\begin{cases} \Hom_\T(M^{-p},S) & \text{if } 0\leq p\leq l-1,\\
\Hom_\T(Y,S[p]) & \text{if } p\geq l.\end{cases}
\]
\end{proposition}
\begin{proof} We prove the statement for $l=1$. The general case is obtained by induction on $l$. 
Take a minimal right $\add(M)$-approximation $f\colon M^0\to X$ and form a triangle $M^0\stackrel{f}{\to} X\to Y\to M^0[1]$. Then $Y$ belongs to $\T_M^{\leq -1}$ as shown in the proof of Lemma~\ref{lem:co-t-structure-for-presilting}. It remains to prove that $\Hom_\T(f,S)\colon\Hom_\T(X,S)\to\Hom_\T(M^0,S)$ is an isomorphism. By the definition of minimal approximations, the homomorphism $F(f)\colon F(M'){\to}F(X)$ is a projective cover in $\mod\End_\T(M)$, where $F=\Hom_\T(M,-)$.
So by Lemma~\ref{lem:truncation-preserves-morphisms}, the homomorphism $F(\sigma^0_M(f))\colon F(\sigma^0_M(M')){\to}F(\sigma^0_M(X))$ is a projective cover in $\mod\End_\T(M)$. Recall from Proposition~\ref{prop:heart-of-silting-t-structure}(a) that $F|_{\T^0_M}\colon\T^0_M\to \mod\End_\T(M)$ is an equivalence. So $F(S)$ is a semisimple $\End_\T(M)$-module and $\Hom(F(\sigma^0_M(f)),F(S))\colon\Hom(F(\sigma^0_M(X)),F(S))\to\Hom(F(\sigma^0_M(M')),F(S))$ is an isomorphism. Therefore the bottom map in the following commutative diagram is an isomorphism
\[
\xymatrix@C=4.5pc{
\Hom_\T(X,S)\ar[r]^{\Hom(f,S)} &\Hom_\T(M',S)\\
\Hom_\T(\sigma^0_M(X),S)\ar[r]^{\Hom(\sigma^0_M(f),S)}\ar[u]&\Hom_\T(\sigma^0_M(M'),S).\ar[u]
}
\]
The two vertical maps are isomorphisms by Lemma~\ref{lem:truncation-preserves-morphisms}. So $\Hom(f,S)$ is also an isomorphism.
\end{proof}

\subsection{ST-pairs for finite-dimensional algebras and projective schemes}\label{ss:ST-pair-for-fd-algebra}

The ST-pair in the next lemma is a prototypical example.

\begin{lemma}\label{lem:ST-pair-for-fd-algebra}
Let $\Lambda$ be a finite-dimensional $K$-algebra. Then $(\Kb(\proj\Lambda),\Db(\mod\Lambda),\Lambda)$ is an ST-triple inside $\Db(\mod\Lambda)$.
\end{lemma}
\begin{proof}
Let $\T:=\Db(\mod \Lambda)$. 
Then $\Lambda$ is a presilting object of $\T$, $\Kb(\proj\Lambda)=\thick(\Lambda)$, and 
$(\T_\Lambda^{\leq 0},\T_\Lambda^{\geq 0})$ is the standard $t$-structure on $\T$ (Example~\ref{ex:standard-t-structure}) because $\Hom_{\T}(\Lambda,X[p])=H^p(X)$ for $X\in\T$ and $p\in\mathbb{Z}$. (ST1), (ST2) and the condition (iii) in Remark~\ref{rem:ST2} are satisfied, so $(\Kb(\proj\Lambda),\Db(\mod\Lambda),\Lambda)$ is an ST-triple inside $\Db(\mod\Lambda)$.
\end{proof}

\begin{corollary}\label{cor:ST-pair-variety-with-tilting}
Let $X$ be a projective scheme over $K$. 
If $\per(X)$ has a tilting object, then $(\per(X),\Db(\coh(X)))$ is an ST-pair inside $\Db(\coh(X))$.
\end{corollary}
\begin{proof}
Let $T$ be a tilting object in $\per(X)$ and let $\Lambda:=\End(T)$. Then by tilting theory there is a triangle equivalence $F\colon\D(\Mod\Lambda)\to\D(\Qcoh(X))$, which, by Lemma~\ref{lem:restriction-of-der-equiv-to-per-and-dfd}, restricts to triangle equivalences $\Kb(\proj\Lambda)\to \per(X)$ and $\Db(\mod\Lambda)\to\Db(\coh(X))$. The desired result then follows from Lemma~\ref{lem:ST-pair-for-fd-algebra}.
\end{proof}

For a finite-dimensional $K$-algebra $\Lambda$ it is known that $\Kb(\proj\Lambda)=\Db(\mod\Lambda)$ if and only if $\Lambda$ is of finite global dimension. In fact we have

\begin{proposition}
For a finite-dimensional $K$-algebra $\Lambda$, the following are equivalent:
\begin{itemize}
\item[(i)] $\Lambda$ is of finite global dimension,
\item[(ii)] $\Kb(\proj\Lambda)$ has an algebraic $t$-structure,
\item[(iii)] there is a triangulated category $\C$ such that $(\C,\Kb(\proj\Lambda))$ is an ST-pair,
\item[(iv)] $\Db(\mod\Lambda)$ has a silting object,
\item[(v)] there is a triangulated category $\D$ such that $(\Db(\mod\Lambda),\D)$ is an ST-pair.
\end{itemize}
\end{proposition}
\begin{proof}
(i)$\Rightarrow$(iii)$\Rightarrow$(ii) and (i)$\Rightarrow$(v)$\Rightarrow$(iv) are clear. (iv)$\Rightarrow$(i) is \cite[Example 2.5(a)]{AI12}.

(ii)$\Rightarrow$(i): Let $S$ be the direct sum of a complete set of pairwise non-isomorphic simple objects of the heart of the given algebraic $t$-structure on $\Kb(\proj\Lambda)$. Then $\Kb(\proj\Lambda)=\thick(S)$, in other words, $S$ is a compact generator of $\D(\Mod\Lambda)$. Let $\tilde{E}=\RHom_\Lambda(S,S)$. Then by \cite[Lemma 6.1]{Ke1} there is a triangle equivalence $\RHom_\Lambda(S,-)\colon\D(\Mod\Lambda)\to\D(\tilde{E})$, which restricts to triangle equivalences $\Db(\mod\Lambda)\to\Dfd(\tilde{E})$ and $\Kb(\proj\Lambda)\to\per(\tilde{E})$. Moreover, $\tilde{E}$ satisfies the conditions (P1) and (P2) in Lemma~\ref{lem:silting-for-positive-dg-algebra} , so $\Dfd(\tilde{E})$ has a silting object, and hence $\Db(\mod\Lambda)$ has a silting object. By (iv)$\Rightarrow$(i), we obtain that $\Lambda$ has finite global dimension.
\end{proof}

\smallskip
As finite-dimensional non-positive dg algebras appear naturally as the dg endomorphism algebras of silting objects in $\Kb(\proj\Lambda)$ (see for example \cite[Section 5.1]{KoY14}),  we include the following generalisation of Lemma~\ref{lem:ST-pair-for-fd-algebra}. 
We omit its proof as it is a special case of Proposition~\ref{prop:ST-pair-for-non-positive-dg-algebra}.

\begin{lemma}\label{lem:ST-pair-for-fd-non-positive-dg-algebra}
Let $A$ be a dg $K$-algebra satisfying
\begin{itemize}
\item[(FN1)] $H^p(A)=0$ for all $p>0$,
\item[(FN2)] $H^*(A)$ is finite-dimensional.
\end{itemize}Then $(\per(A),\Dfd(A),A)$ is an ST-triple inside $\Dfd(A)$.
\end{lemma}

\begin{corollary}\label{cor:ST-pair-for-variety-with-silting}
Let $X$ be a projective scheme over $K$.
If $\per(X)$ has a silting object, then $(\per(X),\Db(\coh(X)))$ is an ST-pair inside $\Db(\coh(X))$.
\end{corollary}
\begin{proof}
Let $T$ be a silting object of $\per(X)$ and let $A:=\RHom(T,T)$. Then the cohomology of $A$ is concentrated in non-positive degrees, and by \cite[Theorem 4.3]{Ke1} there is a triangle equivalence $\RHom(T,-)\colon\D(\Qcoh(X))\to \D(A)$, which, by Lemma~\ref{lem:restriction-of-der-equiv-to-per-and-dfd}, restricts to triangle equivalences $\per(X)\to \per(A)$ and $\Db(\coh(X))\to \Dfd(A)$. Since $\per(X)$ is of finite type (Lemma~\ref{lem:compact-and-homologically-finite-objects-for-projetive-variety}), the total cohomology of $A$ is finite-dimensional. The desired result then follows from Lemma~\ref{lem:ST-pair-for-fd-non-positive-dg-algebra}.
\end{proof}

\subsection{ST-pairs for smooth non-positive dg algebras}\label{ss:ST-pair-for-smooth-non-positive-dg-algebra}

Let $\Gamma$ be a dg $K$-algebra satisfying the following conditions:
\begin{itemize}
\item[(SN1)] $H^{p}(\Gamma)=0$ for each integer $p>0$,
\item[(SN2)] $H^{0}(\Gamma)$ is finite-dimensional,
\item[(SN3)] $\per(\Gamma)\supseteq\Dfd(\Gamma)$.
\end{itemize}
The condition (SN3) is satisfied if $\Gamma$ is topologically homologically smooth (Lemma~\ref{lem:topologically homologically-smooth=>finite-projective-dimension}) or homologically smooth (Lemma~\ref{lem:homologically-smooth=>finite-projective-dimension}). Important examples of $\Gamma$ include complete Ginzburg dg algebras of Jacobi-finite quivers with potential and derived preprojective algebras of Dynkin quivers, which will be discussed in Section \ref{section calabi-yau}. If $\Gamma$ is homologically smooth, then the ingredients of the proof of the following lemma is already contained in \cite{Am09}.

\begin{lemma}\label{lem:standard-ST-pair-for-smooth-non-positive-dg-algebra}
$(\per(\Gamma),\Dfd(\Gamma),\Gamma)$ is an ST-triple inside $\per(\Gamma)$.
\end{lemma}
\begin{proof}
Let $\T:=\per(\Gamma)$, which is Hom-finite and Krull--Schmidt by \cite[Proposition 2.5]{KaY16}. By Lemma~\ref{lem:non-positive-dg-algebra-and-silting}, $\Gamma$ is a silting object of $\T$. It follows from the formula \eqref{eq:Hom-space-from-free-dg-module} that
\begin{eqnarray*}
\T_{\Gamma}^{\leq 0}\hspace{-5pt}&=\hspace{-5pt}&\{X\in\per(\Gamma)\mid H^p(X)=0~\forall p>0\},\\
\T_{\Gamma}^{\geq 0}\hspace{-5pt}&=\hspace{-5pt}&\{X\in\per(\Gamma)\mid H^p(X)=0~\forall p<0\},\\
\T_{\Gamma}^{0}\hspace{-5pt}&=\hspace{-5pt}&\{X\in\per(\Gamma)\mid H^p(X)=0~\forall p\neq 0\}.
\end{eqnarray*}
By \cite[Propositions 2.5 and 2.1(c)]{KaY16}, $(\tT_{\Gamma}^{\le 0}, \tT_{\Gamma}^{\ge 0})$ is a $t$-structure on $\T$ and $\Dfd(\Gamma)=\thick(\tT_{\Gamma}^{0})$. For any $X\in\T$, we have $H^p(X)=\Hom(\Gamma,X[p])=0$ for $p\gg 0$. It follows that $\tT_{\Gamma}^{\ge 0}\subseteq \Dfd(\Gamma)$.
So $(\per(\Gamma),\Dfd(\Gamma),\Gamma)$ is an ST-triple inside $\per(\Gamma)$.
\end{proof}

It is clear that under the conditions (SN1--3), the equality $\per(\Gamma)=\Dfd(\Gamma)$ holds if and only if $H^*(\Gamma)$ is finite-dimensional.

\subsection{Left-right symmetry}

This subsection is about the left-right symmetry of the categories and structures involved in an ST-pair.

\smallskip
Let $(\C,\D,M)$ be an ST-triple inside $\T$. By Proposition~\ref{prop:heart-of-silting-t-structure}, $\D_M^0$ is a length abelian category. Let $S$ be the direct sum of a complete set of pairwise non-isomorphic simple objects of $\D_M^0$. Then $\D_M^{\leq 0}$ (respectively, $\D_M^{\geq 0}$) is the smallest full subcategory of $\T$ which contains $S$ and is closed under extensions, direct summands and $[1]$ (respectively, $[-1]$), by Lemma~\ref{BBD}(b).
Put
\[
\T_{M,\geq 0}:=\bigcup\nolimits_{n\geq 0} M^{[-n,0]},~~\T_{M,\leq 0}:=\T_M^{\leq 0}.
\]
The following description of these categories are dual to the definition of $\T_M^{\leq 0}$ and $\T_M^{\geq 0}$.

\begin{lemma} \label{lem:dual-description-of-co-t-str}
We have
\begin{eqnarray*}
\T_{M,\leq 0}\hspace{-5pt}&=\hspace{-5pt}&\{X\in\T\mid \Hom(X[>\hspace{-3pt}0],S)=0\},\\
\T_{M,\geq 0}\hspace{-5pt}&=\hspace{-5pt}&\{X\in\T\mid\Hom(X[<\hspace{-3pt}0],S)=0\}.
\end{eqnarray*}
\end{lemma}
\begin{proof}
The first equality follows from the fact $\T_M^{\leq 0}={}^{\perp_\T}\D_M^{\geq 1}$. The inclusion `$\subseteq$' in the second equality is clear because the category on the right hand contains $M$ and is closed under extensions, direct summands and $[-1]$. To show `$\supseteq$', take $X\in\T$ satisfying $\Hom(X[<\hspace{-3pt}0],S)=0$. By (ST3), there exists $n\geq 0$ such that $X\in \T_M^{\leq n}$. 
By Proposition~\ref{prop:minimal-resolution}, there exist $M'\in M^{[-n,0]}$ and $Y\in\T_M^{\leq -l}$ such that $X\in M'\ast Y$ and $\Hom(Y,S[p])\simeq \Hom(X,S[p])$. Therefore $\Hom(Y,S[p])=0$ for all $p\in\mathbb{Z}$, so $Y\in\bigcap_{n\in\mathbb{Z}}{}^{\perp_\T}\D_M^{\geq n}=\bigcap_{n\in\mathbb{Z}}\T^{\leq n}_M=0$, where the last equality is by Proposition~\ref{prop:heart-of-silting-t-structure}(c). Therefore $X\in M^{[-n,0]}\subseteq \T_{M,\geq 0}$.
\end{proof}

In view of Lemma~\ref{lem:dual-description-of-co-t-str}, the following properties are dual to (ST1--3).

\begin{proposition} \label{prop:ST=>TS}
We have
\begin{itemize}
\item[(ST1')] $\Hom_\T(T,S)$ is finite-dimensional for any object $T$ of $\T$,
\item[(ST2')] $(\T_{M,\geq 0},\T_{M,\leq 0})$ is a co-$t$-structure on $\T$,
\item[(ST3')] $\T=\bigcup_{n\in\mathbb{Z}}\T_{M,\leq n}$ and $\C=\bigcup_{n\in\mathbb{Z}}\T_{M,\geq n}$.
\end{itemize}
\end{proposition}
\begin{proof}
(ST1') Let $T\in\T$. By Proposition~\ref{prop:minimal-resolution}, there exists $M^0\in\add(M)$ such that $\Hom_\T(T,S)$ is isomorphic to $\Hom_\T(M^0,S)$, which is finite-dimensional.

(ST2') We verify the four conditions in Definition~\ref{defn:co-t-structure}.
(0) and (1) follow from Lemma~\ref{lem:dual-description-of-co-t-str}.
(2) follows from the definition of the two categories. (3) holds because 
\begin{eqnarray*}
\T\hspace{-5pt}&=\hspace{-5pt}&\bigcup\nolimits_{n\geq 0}\T_M^{\leq n}=\bigcup\nolimits_{n\geq 0}(M^{[-n,-1]}\ast \T_M^{\leq 0})\\
\hspace{-5pt}&=\hspace{-5pt}&(\bigcup\nolimits_{n\geq 0}M^{[-n,-1]})\ast \T_M^{\leq 0}=\T_{M,\geq 1}\ast\T_{M,\leq 0},
\end{eqnarray*}
where the first equality follows from (ST3), and the second equality follows from Lemma~\ref{lem:co-t-structure-for-presilting}.

(ST3') This follows from (ST3) and Lemma~\ref{AI223}(c).
\end{proof}

\begin{remark}
The co-heart of the co-$t$-structure $(\T_{M,\geq 0},\T_{M,\leq 0})$ is $\add(M)$. This follows from Lemma~\ref{AI223}(b) because the co-heart consists of objects $X$ which belongs to $X\in M^{[-n,0]}$ for some positive integer $n$ and which satisfies  $\Hom_\T(M,X[<\hspace{-3pt}0])=0$. Let $\C_{M,\geq 0}:=\T_{M,\geq 0}$ and $\C_{M,\leq 0}:=\T_{M,\leq 0}\cap\C$. Then $\C_{M,\leq 0}=\bigcup_{n\geq 0}M^{[0,n]}$ and $(\C_{M,\geq 0},\C_{M,\leq 0})$ is a bounded co-$t$-structure on $\C$ with co-heart $\add(M)$. The $t$-structure $(\D_M^{\leq 0},\D_M^{\geq 0})$  is right orthogonal to the co-t-structure $(\C_{M,\geq 0},\C_{M,\leq 0})$ in the sense of Bondarko \cite[Definition 2.5.1]{Bondarko10a}.
\end{remark}

The categories $\C$ and $\D$ determine each other in the following way.

\begin{theorem}\label{thm:S-and-T-determine-each-other}
We have $\D=\T_{\C}^{\chf}$ and $\C=\T_{\hf}^{\D}$.
\end{theorem}
\begin{proof}
We have
\begin{eqnarray*}
\T_{\C}^{\chf}\hspace{-5pt}&=\hspace{-5pt}&\{Y\in\T\mid\bigoplus_{p\in\mathbb{Z}}\Hom_\T(X,Y[p]) \text{ is finite-dimensional }\forall X\in\C\}\\
\hspace{-5pt}&=\hspace{-5pt}&\{Y\in\T\mid\bigoplus_{p\in\mathbb{Z}}\Hom_\T(M,Y[p]) \text{ is finite-dimensional }\}\\
\hspace{-5pt}&=\hspace{-5pt}&\{Y\in\T\mid\exists~m,n\in\mathbb{Z} \text{ s.t. }\Hom_\T(M,Y[p])=0\text{ for }p>m\text{ and }p<n\}\\
\hspace{-5pt}&=\hspace{-5pt}&\bigcup_{m,n\in\mathbb{Z}}(\T_M^{\geq -m}\cap\T_M^{\leq -n})\\
\hspace{-5pt}&=\hspace{-5pt}&\bigcup_{m,n\in\mathbb{Z}}(\D_M^{\geq -m}\cap\D_M^{\leq -n})\\
\hspace{-5pt}&=\hspace{-5pt}&\D,
\end{eqnarray*}
and
\begin{eqnarray*}
\T^{\D}_{\hf}\hspace{-5pt}&=\hspace{-5pt}&\{X\in\T\mid\bigoplus_{p\in\mathbb{Z}}\Hom_\T(X,Y[p]) \text{ is finite-dimensional }\forall Y\in\D\}\\
\hspace{-5pt}&=\hspace{-5pt}&\{X\in\T\mid\bigoplus_{p\in\mathbb{Z}}\Hom_\T(X,S[p]) \text{ is finite-dimensional }\}\\
\hspace{-5pt}&=\hspace{-5pt}&\{X\in\T\mid\exists~m,n\in\mathbb{Z} \text{ s.t. }\Hom_\T(X,S[p])=0\text{ for }p>m\text{ and }p<n\}\\
\hspace{-5pt}&=\hspace{-5pt}&\bigcup_{m,n\in\mathbb{Z}}(\T_{M,\geq -m}\cap\T_{M,\leq -n})\\
\hspace{-5pt}&=\hspace{-5pt}&\bigcup_{m,n\in\mathbb{Z}}(\C_{M,\geq -m}\cap\C_{M,\leq -n})\\
\hspace{-5pt}&=\hspace{-5pt}&\C. \qedhere
\end{eqnarray*}
\end{proof}

Immediately we obtain the following corollary.

\begin{corollary}\label{cor:ST-pair-for-category-of-finite-type}
Let $(\C,\D)$ be an ST-pair. Then $\C$ is of finite type if and only if $\C\subseteq \D$; $\D$ is of finite type if and only if $\C\supseteq\D$.
\end{corollary}

\begin{remark}
The Euler form $K_0(\C)\times K_0(\D)\to\mathbb{Z}$, $([X],[Y])\mapsto \sum_{p\in\mathbb{Z}}\Hom(X,Y[p])$, is a non-degenerate bilinear form. In fact, $K_0(\C)$ is free with basis the classes of indecomposable direct summands of $M$, and $K_0(\D)$ is free with basis the classes of simple objects of $\D_M^0$. By Corollary~\ref{cor:Hom-duality-between-silting-and-smc}, these two bases are dual to each other up to the dimensions of the endomorphism algebras of the simple objects.
\end{remark}


\section{From silting objects to $t$-structures}
\label{s:order-preserving-map}

In \cite{KV88} Keller and Vossieck constructed a map from the set of silting objects to the set of bounded $t$-structures on derived categories of Dynkin quivers. This work was later generalised in \cite{KN2,KoY14,SY16} for derived categories of non-positive dg algebras. 
In this section, we make a further generalisation in the framework of an ST-pair $(\C,\D)$  and establishing an order-preserving injective map from the set of silting objects of $\C$ to the set of  bounded $t$-structures on $\D$. In the situations of \cite{KV88,KN2,KoY14,SY16}, there are natural ST-pairs and for these ST-pairs our map coincides with the maps in \emph{loc.cit.}.


\subsection{$t$-structures induced by a presilting object}
In this subsection we prove the following result, which generalises \cite[Theorem 4.4]{IYa14} and plays a fundamental role in the construction of the map from silting objects to $t$-structures. We remark that $M$ is assumed to be silting in the original statement of \cite[Theorem 4.4]{IYa14}. Let $\T$ be a $K$-linear triangulated category with shift functor $[1]$. 

\begin{lemma}\label{IYa-thm44}
Let $M,N$ be presilting objects of $\T$ satisfying $\thick(M)=\thick(N)$. Assume that both $\add(M)$ and $\add(N)$ are contravariantly finite in $\T$. Then $(\tT_{M}^{\le 0},\tT_{M}^{\ge 0})$ is a $t$-structure on $\T$ if and only if $(\tT_{N}^{\le 0}, \tT_{N}^{\ge 0})$ is a $t$-structure on $\T$.
\end{lemma}

\begin{proof}
We only show that, if $(\tT_{M}^{\le 0},\tT_{M}^{\ge 0})$ is a $t$-structure on $\T$, so is $(\tT_{N}^{\le 0},\tT_{N}^{\ge 0})$. It is enough to prove that (1) $\Hom_{\T}(\tT_{N}^{\le 0}, \tT_{N}^{\ge 1})=0$ and (2) $\T=\tT_{N}^{\le 0}\ast \tT_{N}^{\ge 1}$.

Since $N\in \thick(M)$, it follows from Lemma \ref{AI223}(c) that there exist non-negative integers $m,n$ such that $N\in M^{[-m,n]}$ (equivalently, $M\in N^{[-n,m]}$ by Lemma \ref{lemma32}). Without loss of generality, we may assume $m=0$ and $n>0$. By Lemma \ref{lem:partial-orders-of-presilting-and-t-structure}, we have $\tT_{M}^{\le 0}\supseteq \tT_{N}^{\le 0}\supseteq \tT_{M}^{\le -n}$. Then
\[
N^{[0,n-1]}\ast \tT_{M}^{\le -n} \subseteq \tT_{N}^{\le 0}=N^{[0,n-1]}\ast \tT_{N}^{\le -n}\subseteq N^{[0,n-1]}\ast \tT_{M}^{\le -n}, 
\]
where the equality is obtain by applying Lemma \ref{triangle1} to the presilting object $N$. So we obtain the equality
\begin{eqnarray}\label{decomp-subcat}
\tT_{N}^{\le 0}=N^{[0,n-1]}\ast \tT_{M}^{\le -n}.
\end{eqnarray}

(1) By the equality (\ref{decomp-subcat}), it suffices to show the equalities $\Hom_{\T}(N^{[0,n-1]}, \tT_{N}^{\ge 1})=0$ and $\Hom_{\T}(\tT_{M}^{\le -n}, \tT_{N}^{\ge 1})=0$. The first equality is clear by the definition of $\tT_{N}^{\ge 1}$. To show the second equality, take $X\in \tT_{N}^{\ge 1}$; then $\Hom_{\T}(M,X[<\hspace{-3pt}-n+1])\subseteq\Hom_{\T}(N^{[0,n]},X[<\hspace{-3pt}1])=0$, showing that $\T_N^{\geq 1}\subseteq \T_M^{\geq -n+1}$.

(2) Take $\X:=N^{[0,2n]}$ and $\Y:=\X^{\perp_{{\T}}}$. Then $\T=\X\ast \Y$ by \cite[Lemma 4.6]{IYa14}.
We claim that $\Y\subseteq \tT_{N}^{\le 0}\ast \tT_{N}^{\ge 1}$.
Then
\[
\T=\X\ast \Y \subseteq \tT_{N}^{\le 0}\ast (\tT_{N}^{\le 0}\ast \tT_{N}^{\ge 1})= \tT_{N}^{\le 0}\ast \tT_{N}^{\ge 1} \subseteq \T,
\]
implying the equality $\T=\tT_{N}^{\le 0}\ast \tT_{N}^{\ge 1}$.

Now we show the inclusion $\Y\subseteq \tT_{N}^{\le 0}\ast \tT_{N}^{\ge 1}$.
Let $Y\in \Y$. 
Since $(\tT_{M}^{\le -2n-1}, \tT_{M}^{\ge -2n-1})$ is a $t$-structure, there is a triangle 
$Y' \to Y \to Y'' \to Y'[1]$ 
with $Y'\in \tT_{M}^{\le -2n-1}$ and $Y''\in \tT_{M}^{\ge -2n}$. Since $Y'\in \T_M^{\leq -2n-1}\subseteq\tT_{N}^{\le 0}$, it remains to show $Y''\in \tT_{N}^{\ge 1}$. By the equality (\ref{decomp-subcat}) and  Lemma \ref{triangle1}, we have 
\[
\tT_{N}^{\le 0}=N^{[0,n-1]}\ast M^{[n, 2n]}\ast \tT_{M}^{\le -2n-1}.
\]
Because $N\in M^{[0,n]}$ and $M\in N^{[-n,0]}$, there is an inclusion $N^{[0,n-1]}\ast M^{[n,2n]}\subseteq M^{[0,2n]}\cap N^{[0,2n]}$.
This implies the vanishing of $\Hom_\T(N^{[0,n-1]}\ast M^{[n,2n]},Y)$ and $\Hom_\T(N^{[0,n-1]}\ast M^{[n,2n]},Y'[1])$. It follows that $\Hom_\T(N^{[0,n-1]}\ast M^{[n,2n]},Y'')=0$. Since $\Hom_\T(\T_M^{\leq -2n-1},Y'')$ vanishes, so does $\Hom_\T(\T_N^{\leq 0},Y'')$. In particular, $Y''\in \T_N^{\geq 1}$.
\end{proof}

\begin{proposition}\label{ST-property} 
Let $(\C,\D,M)$ be an ST-triple inside $\T$ and let $N$ be a silting object of $\C$. Then $(\C,\D,N)$ is an ST-triple.
\end{proposition}
\begin{proof}
We need to verify the three conditions (ST1--3) for $N$.

(ST1) This follows from (ST1) for $M$, see Remark~\ref{rem:ST1}(b).

(ST2) This follows from Lemma~\ref{IYa-thm44} and (ST2) for $M$.

(ST3) Since $N\in \C=\thick(M)$, it follows from Lemma~\ref{AI223}(c) that there exists an integer $n\ge 0$ such that $N\in M^{[-n,n]}$. By Lemmas~\ref{lem:partial-order-and-interval} and~\ref{lem:partial-orders-of-presilting-and-t-structure}, we have $\tT_{M}^{\le n}\supseteq \tT_{N}^{\le 0}\supseteq \tT_{M}^{\le -n}$, and hence $\T_M^{\ge n}\subseteq \tT_{N}^{\ge 0}\subseteq \tT_{M}^{\ge -n}$. Therefore
\[
\T=\bigcup\nolimits_{l\in\mathbb{Z}} \T_M^{\leq l}=\bigcup\nolimits_{l\in\mathbb{Z}}\T_N^{\leq l},~~\text{and}~~
\D=\bigcup\nolimits_{l\in\mathbb{Z}} \T_M^{\geq l}=\bigcup\nolimits_{l\in\mathbb{Z}}\T_N^{\geq l}.\qedhere
\]
\end{proof}

\subsection{Order-preserving map}\label{ss:order-preserving-map}
Let $(\C,\D)$ be an ST-pair. 
In view of Proposition~\ref{prop:heart-of-silting-t-structure}(b), we  introduce a new class of $t$-structures as follows. 

\begin{definition}
A bounded $t$-structure $(\tD^{\le 0}, \tD^{\ge 0})$ on $\D$ is said to be \emph{$\C$-silting} if there exists a silting object $M$ of $\C$ such that $\tD^{\le 0}=\tD_{M}^{\le 0}$ (or equivalently $\tD^{\ge 0}=\tD_{M}^{\ge 0}$). 
\end{definition}

Then $\C$-silting $t$-structures have the following remarkable property by Proposition~\ref{prop:heart-of-silting-t-structure}(a).

\begin{corollary}\label{C-silt is alg} 
Every $\C$-silting $t$-structure on $\D$ is algebraic. 
\end{corollary}

Proposition~\ref{ST-property} allows us to establish an order-preserving map from the set $\silt(\C)$ of silting objects of $\C$ to the set $\tstr(\D)$ of bounded $t$-structures on $\D$.

\begin{theorem}\label{basic-map}
Let $(\C,\D)$ be an ST-pair.
There is an order-preserving injection
\[
\Psi=\Psi_{(\C,\D)}\colon \silt(\C) \to \tstr(\D),\ \ M\mapsto (\tD_{M}^{\le 0}, \tD_{M}^{\ge 0}). 
\]
Consequently, it induces a bijection from $\silt(\C)$ to the set of $\C$-silting $t$-structures on $\D$.
\end{theorem}

\begin{proof}
By Propositions~\ref{ST-property} and~\ref{prop:heart-of-silting-t-structure}(b), the map $\Psi$ is well-defined.
Moreover, it is an order-preserving injective map because the following holds:
\[
M\ge N \Longleftrightarrow \tT_{M}^{\le 0} \supseteq \tT_{N}^{\le 0} \Longleftrightarrow \tD_{M}^{\ge 0} =\tT_{M}^{\ge 0} \subseteq \tT_{N}^{\ge 0}=\tD_{N}^{\ge 0} \Longleftrightarrow \tD_{M}^{\le 0} \supseteq \tD_{N}^{\le 0}. 
\]
The proof is complete.
\end{proof}

In view of Lemma~\ref{lem:partial-orders-of-presilting-and-t-structure}, the following corollary is immediate.

\begin{corollary} \label{cor:restriction-of-the-map-to-n-terms}
Let $(\C,\D)$ be an ST-pair,  $M$ a basic silting object of $\C$ and $n$ a positive integer. The map  $\Psi_{(\C,\D)}$ restricts to an injection
\[
\silt^n_{M}(\C)\longrightarrow\tstr^n_{M}(\D),
\]
where $\tstr^n_{M}(\D)$ denotes the set of $t$-structures $(\D'^{\leq 0},\D'^{\geq 0})$ on $\D$ satisfying $\D_M^{\leq 0}\supseteq\D'^{\leq 0}\supseteq\D_M^{\leq -n+1}$. 
\end{corollary}

\begin{remark}\label{rem:ST-pair-interpretation-of-previous-results}
For the following pairs $(\C,\D)$:
\begin{itemize}
\item[(1)] $(\Db(\mod k\Delta),\Db(\mod k\Delta))$, where $\Delta$ is a Dynkin quiver,
\item[(2)] $(\Kb(\proj\Lambda),\Db(\mod\Lambda))$, where $\Lambda$ is a finite-dimensional $K$-algebra,
\item[(3)] $(\per(\Gamma),\Dfd(\Gamma))$, where $\Gamma$ is a dg $K$-algebra satisfying the conditions (SN1--3) in Section~\ref{ss:ST-pair-for-smooth-non-positive-dg-algebra},
\item[(4)] $(\per(A),\Dfd(A))$, where $A$ is a dg $K$-algebra satisfying (FN1--2) in Lemma~\ref{lem:ST-pair-for-fd-non-positive-dg-algebra},
\end{itemize}
a map $\Psi'_{(\C,\D)}\colon\silt(\C)\to\tstr(\D)$ was established in \cite{KV88,KoY14,KN2,SY16}, respectively. We have seen in Section~\ref{s:ST-pairs-examples} that these pairs are ST-pairs. In fact, $\Psi_{(\C,\D)}=\Psi'_{(\C,\D)}$. For cases (2), (3) and (4), the definitions of $\Psi$ and $\Psi'$ are exactly the same; for case (1), observe that if $\C=\D$ then $\D_M^{\leq 0}=\C_{M,\leq 0}$ is the smallest full subcategory of $\D$ which contains $M$ and which is closed under extensions, direct summands and $[1]$, by Lemma~\ref{AI223}.  Moreover, it is shown in {\it loc. cit.} (\cite[Theorem 5]{KV88}, \cite[Theorem 6.1]{KoY14}, \cite[Theorem 13.3]{KN2}, \cite[Theorem 1.2]{SY16}) that $\Psi$ is bijective in case (1) and the image of $\Psi$ is the set of algebraic $t$-structures on $\D$ in cases (2), (3) and (4) (for (3) and (4), the field $K$ is assumed to be algebraically closed).
In particular, in all these cases the converse of Corollary~\ref{C-silt is alg} holds: every algebraic $t$-structure on $\D$ is $\C$-silting. 
\end{remark}

\section{ST-pairs: uniqueness and existence}

In this section we discuss the uniqueness of ST-pairs, especially for the ST-pairs in Lemmas~\ref{lem:ST-pair-for-fd-algebra},~\ref{lem:ST-pair-for-fd-non-positive-dg-algebra}~and~\ref{lem:standard-ST-pair-for-smooth-non-positive-dg-algebra} and Corollary~\ref{cor:ST-pair-for-variety-with-silting}, and show that if $\C$ is a Hom-finite Krull--Schmidt algebraic triangulated category with silting objects, then there exists an ST-pair $(\C,\D)$.

\subsection{Uniqueness}

In this subsection we discuss the uniqueness of ST-pairs. Let $\T$ and $\T'$ be $K$-linear triangulated categories with shift functor $[1]$. 

\smallskip
We first show that a triangulated category can have at most one ST-pair.
\begin{theorem}\label{thm:uniqueness-of-ST-pair-inside}
Let $(\C,\D,M)$ and $(\C',\D',N)$ be ST-triples inside $\T$. Then $\C=\C'$ and $\D=\D'$.
\end{theorem}
\begin{proof}
By (ST3), there exists $n\in\mathbb{Z}$ such that $N\in\T_M^{\leq n}$, which implies that $\T_N^{\geq 0}\supseteq \T_M^{\geq n}$. Similarly, there exists $m\in\mathbb{Z}$ such that $\T_M^{\geq 0}\supseteq \T_N^{\geq m}$. Then $\T_N^{\geq m}\subseteq\T_M^{\geq 0}\subseteq\T_N^{\geq -n}$. It follows that 
\[
\D=\bigcup\nolimits_{n\in\mathbb{Z}}\T_M^{\geq n}=\bigcup\nolimits_{m\in\mathbb{Z}}\T_N^{\geq m}=\D',
\] 
and $\T_N^{\leq m}\supseteq\T_M^{\leq 0}\supseteq\T_N^{\leq -n}$, which implies that $\T_{N,\geq m}\subseteq\T_{M,\geq 0}\subseteq\T_{N,\geq -n}$. Therefore
\[
\C=\bigcup\nolimits_{n\in\mathbb{Z}}\T_{M,\geq n}=\bigcup\nolimits_{m\in\mathbb{Z}}\T_{N,\geq m}=\C'.\qedhere
\]
\end{proof}

In general, the category $\T$ in which an ST-pair sits is not unique. For example, if $\Lambda$ is a finite-dimensional $K$-algebra, then $(\Kb(\proj\Lambda),\Db(\mod\Lambda))$ is an ST-pair also inside the right bounded derived category $\D^-(\mod\Lambda)$. In fact, if $(\C,\D,M)$ be an ST-triple inside $\T$, then $(\C,\D,M)$ is an ST-triple inside any thick subcategory of $\T$ containing both $\C$ and $\D$. 

\begin{definition}\label{defn:equivalence-of-ST-pairs}
Let $(\C,\D)$ be an ST-pair inside $\T$ and $(\C',\D')$ be an ST-pair inside $\T'$. We say that $(\C,\D)$ is \emph{equivalent} to $(\C',\D')$ if there is a triangle equivalence $\thick_\T(\C\cup\D)\to\thick_{\T'}(\C'\cup\D')$. 
\end{definition}

For example, the ST-pair $(\Kb(\proj\Lambda),\Db(\mod\Lambda))$ inside $\Db(\mod\Lambda)$ is equivalent to the ST-pair $(\Kb(\proj\Lambda),\Db(\mod\Lambda))$ inside $\D^-(\mod\Lambda)$. 
Note that by Theorem~\ref{thm:uniqueness-of-ST-pair-inside}, an equivalence in Definition~\ref{defn:equivalence-of-ST-pairs} automatically restricts to triangle equivalences $\C\to\C'$ and $\D\to\D'$.

We expect that if $\C\simeq \C'$ or $\D\simeq \D'$ then the ST-pairs $(\C,\D)$ and $(\C',\D')$ are equivalent. We can prove this under certain assumptions.

\begin{proposition}\label{prop:tilting=>unique-ST-pair}
Let $(\C,\D)$ be an ST-triple. Assume that $\D$ is algebraic and that $\C$ has a tilting object $M$. Then $(\C,\D)$ is equivalent to the ST-pair $(\Kb(\proj\Lambda),\Db(\mod\Lambda))$, where $\Lambda=\End_\C(M)$.
\end{proposition}
\begin{proof}
Since $\C$ has a tilting object, it is of finite type, and hence by Corollary~\ref{cor:ST-pair-for-category-of-finite-type} we have $\C\subseteq \D$. By Lemma~\ref{IYa-thm44}, $(\C,\D,M)$ is an ST-triple. By Proposition~\ref{prop:heart-of-silting-t-structure}, the heart $\D_M^0$ is equivalent to $\mod\Lambda$. Fix an equivalence $F\colon\mod\Lambda\to\D_M^0$ which takes $\Lambda$ to $M$ and extend it to a triangle functor 
\[
R\colon\Db(\mod\Lambda)\longrightarrow\D
\]
by \cite[Proposition 3.1.10]{BBD81}. We will show that $R$ is a triangle equivalence. By Beilinson's criterion \cite[Lemma 4.2]{Ke1}, it is enough to show that $\Hom(S_\Lambda,S_\Lambda[p])\stackrel{R}{\to}\Hom(S,S[p])$ is an isomorphism for all $p\in\mathbb{Z}$, where $S_\Lambda$ is the direct sum of a complete set of pairwise non-isomorphic simple $\Lambda$-modules and $S=F(S_\Lambda)$. First, it is clear that $\Hom(\Lambda,S_\Lambda[p])\stackrel{R}{\to}\Hom(M,S[p])$ is an isomorphism for all $p\in\mathbb{Z}$. This implies that $\Hom(X,S_\Lambda[p])\stackrel{R}{\to}\Hom(F(X),S[p])$ is an isomorphism for all $X\in \Kb(\proj\Lambda)$ and all $p\in\mathbb{Z}$. Now fix any $p\in\mathbb{Z}$. There is a triangle $X\to S_\Lambda\to Y\to X[1]$ in $\Db(\mod\Lambda)$ with $X\in\Lambda^{[0,p+1]}$ and $Y\in\D_{\std}^{\leq -p-2}$. So $R(X)\in M^{[0,p+1]}$ and $R(Y)\in\D_M^{\leq -p-2}$ and there is a triangle $R(X)\to S\to R(Y)\to R(X)[1]$ in $\D$. Applying $\Hom_{\Db(\mod\Lambda)}(-,S_\Lambda)$ and $\Hom_\D(-,S)$ to these two triangles, respectively, we obtain a commutative diagram
\[
\xymatrix@C=1pc{
\Hom(Y,S_\Lambda[p])\ar[r]\ar[d] & \Hom(S_\Lambda,S_\Lambda[p])\ar[r]\ar[d]^R & \Hom(X,S_\Lambda[p])\ar[r] \ar[d]^R_\simeq & \Hom(Y,S_\Lambda[p+1])\ar[d]\\
\Hom(R(Y),S[p])\ar[r] & \Hom(S,S[p]) \ar[r] & \Hom(R(X),S[p])\ar[r] & \Hom(R(Y),S[p+1]).
}
\]
The four outer terms vanish, so we obtain the desired isomorphism.
\end{proof}

In fact, we have proved the following result.

\begin{corollary}\label{cor:unique-ST-pair-for-the-homotopy-category}
Let $\Lambda$ be a finite-dimensional $K$-algebra. Let $(\Kb(\proj\Lambda),\D)$ be an ST-pair with $\D$ a $K$-linear algebraic triangulated category. Then $\Kb(\proj\Lambda)\subseteq\D$ and there is a triangle equivalence $\Db(\mod\Lambda)\to \D$ which restricts to the identity on $\Kb(\proj\Lambda)$.
\end{corollary}

\subsubsection{The case $\C\supseteq\D$}

\begin{lemma}\label{lem:uniqueness-case-CcontainsD}
Let $(\C,\D)$ be an ST-pair inside $\T$ with $\C\supseteq\D$ and let $(\C,\D')$ be an ST-pair inside $\T'$. Then $\D'=\D$.
\end{lemma}
\begin{proof}
Let $M$ be a silting object of $\C$. By Theorem~\ref{thm:S-and-T-determine-each-other}, we have $\D=\C^{\chf}\subseteq (\T')_\C^{\chf}=\D'$. It follows that $\D_M^0\subseteq (\D')_M^0$. By Lemma~\ref{IYa-thm44}, both $(\C,\D,M)$ and $(\C,\D',M)$ are ST-triples. By Proposition~\ref{prop:heart-of-silting-t-structure}(a), both $\Hom(M,-):(\D')_M^0\to\mod\End(M)$ and its restriction to $\D_M^0$ are equivalences. Therefore the inclusion $\D_M^0\subseteq (\D')_M^0$ is an equivalence. So $\D'=\thick((\D')_M^0)=\thick(\D_M^0)=\D$.
\end{proof}

\begin{proposition}\label{prop:uniqueness-case:C-is-homologically-smooth}
Let $\C$ be a $K$-linear triangulated category triangle equivalent to $\per(\Gamma)$ for some dg $K$-algebra $\Gamma$ satisfying the conditions (SN1--3) in Section~\ref{ss:ST-pair-for-smooth-non-positive-dg-algebra}.
\begin{itemize}
\item[(a)] $(\C,\C^{\chf})$ is an ST-pair and it is equivalent to $(\per(\Gamma),\Dfd(\Gamma))$.
\item[(b)] If $(\C,\D)$ is an ST-pair inside $\T$, then $\D=\C^{\chf}$. In particular, the image of $\Psi_{(\C,\D)}$ is the set of algebraic $t$-structures on $\D$, provided that $K$ is algebraically closed.
\end{itemize}
\end{proposition}
\begin{proof}
(a) By Lemma~\ref{lem:standard-ST-pair-for-smooth-non-positive-dg-algebra}, $(\per(\Gamma),\Dfd(\Gamma))$ is an ST-pair. Since $\Dfd(\Gamma)=\per(\Gamma)^{\chf}$ by Lemma~\ref{lem:compact-and-homologically-finite-objects-for-dg-algebra}, we deduce that $(\C,\C^{\chf})$ is an ST-pair and is equivalent to $(\per(\Gamma),\Dfd(\Gamma))$.

(b) The first statement follows from (a) and Lemma~\ref{lem:uniqueness-case-CcontainsD}. The second statement follows from (a) and Remark~\ref{rem:ST-pair-interpretation-of-previous-results}.
\end{proof}

\subsubsection{The case $\C\subseteq\D$}

\begin{lemma}\label{lem:uniqueness-case-DcontainsC}
Let $(\C,\D)$ be an ST-pair inside $\T$ with $\C\subseteq \D$ such that the image of $\Psi_{(\C,\D)}$ is the set of algebraic $t$-structures on $\D$. Let $(\C',\D)$ be an ST-pair inside $\T'$. Then $\C'=\C$.
\end{lemma}
\begin{proof}
By Theorem~\ref{thm:S-and-T-determine-each-other}, we have $\C=\D_{\hf}\subseteq (\T')_{\hf}^{\D}=\C'$. Let $M$ be a silting object of $\C'$. Then by Lemma~\ref{IYa-thm44}, $(\C',\D,M)$ is an ST-triple, and hence by Proposition~\ref{prop:heart-of-silting-t-structure}(a), $(\D_M^{\leq 0},\D_M^{\geq 0})$ is an algebraic $t$-structure on $\D$. By assumption, there is a silting object $N$ of $\C$ such that $\D_N^{\leq 0}=\D_M^{\leq 0}$ and $\D_N^{\geq 0}=\D_M^{\geq 0}$. Let $S$ be the direct sum of a complete set of pairwise non-isomorphic simple objects of $\D_N^0=\D_M^0$. Then $\Hom(N,S[p])=0$ for $p\neq 0$, so $N$ belongs to $\T_{M,\geq 0}\cap\T_{M,\leq 0}=\add(M)$ by Lemma~\ref{lem:dual-description-of-co-t-str}. Assume that $M$ and $N$ are basic, then by Proposition~\ref{prop:heart-of-silting-t-structure}(a), we have $|M|=|S|=|N|$. So $M\simeq N$ and consequently, $\C'=\thick(M)=\thick(N)=\C$.
\end{proof}

\begin{proposition}\label{prop:uniqueness-case:D-is-homologically-smooth}
Let $\D$ be triangle equivalent to $\Dfd(A)$ for some dg $K$-algebra $A$ satisfying (FN1) and (FN2) in Lemma~\ref{lem:ST-pair-for-fd-non-positive-dg-algebra}.
\begin{itemize}
\item[(a)] $(\D_{\hf},\D)$ is an ST-pair and it is equivalent to $(\per(A),\Dfd(A))$.
\item[(b)] If $(\C,\D)$ is an ST-pair inside $\T$, then $\C=\D_{\hf}$. In particular, the image of $\Psi_{(\C,\D)}$ is the set of algebraic $t$-structures on $\D$, provided that $K$ is algebraically closed.
\end{itemize}
\end{proposition}
\begin{proof} 
(a) By Lemma~\ref{lem:ST-pair-for-fd-non-positive-dg-algebra}, $(\per(A),\Dfd(A))$ is an ST-pair. Since $\per(A)=\Dfd(A)_{\hf}$ by Lemma~\ref{lem:compact-and-homologically-finite-objects-for-dg-algebra}, we deduce that $(\D_{\hf},\D)$ is an ST-pair and is equivalent to $(\per(A),\Dfd(A))$.

(b) The first statement follows from (a) and Lemma~\ref{lem:uniqueness-case-DcontainsC}. The second statement follows from (a) and Remark~\ref{rem:ST-pair-interpretation-of-previous-results}.
\end{proof}

We are particularly interested in the following two cases.

\begin{corollary}\label{cor:uniqueness-for-fd-algebra}
Let $\Lambda$ be a finite-dimensional $K$-algebra.
\begin{itemize}
\item[(a)]
If $(\C,\Db(\mod\Lambda))$ is an ST-pair, then $\C=\Kb(\proj\Lambda)$.
\item[(b)]
If $(\C,\D)$ is an ST-pair inside $\T$ such that $\D$ is triangle equivalent to $\Db(\mod\Lambda)$, then $\C=\D_{\hf}$ and $(\C,\D)$ is equivalent to $(\Kb(\proj\Lambda),\Db(\mod\Lambda))$. In particular, the image of $\Psi_{(\C,\D)}$ is the set of algebraic $t$-structures on $\D$.
\end{itemize}
\end{corollary}
\begin{proof}
(a) and the first statement of (b) follows from Proposition~\ref{prop:uniqueness-case:D-is-homologically-smooth}(b).The second statement of (b) follows from the first statement and 
Remark~\ref{rem:ST-pair-interpretation-of-previous-results}.
\end{proof}

\begin{corollary}\label{cor:uniqueness-for-variety}
Let $X$ be a projective scheme over $K$.
\begin{itemize}
\item[(a)] Assume that $\Db(\coh(X))$ has an algebraic $t$-structure, then $(\per(X),\Db(\coh(X)))$ is an ST-pair. Moreover, the image of $\Psi_{(\per(X),\Db(\coh(X)))}$ is the set of algebraic $t$-structures on $\Db(\coh(X))$, provided that $K$ is algebraically closed.
\item[(b)] If $(\C,\Db(\coh(X)))$ is an ST-pair, then $\C=\per(X)$.
\item[(c)] If $(\C,\D)$ is an ST-pair inside $\T$ such that $\D$ is triangle equivalent to $\Db(\coh(X))$, then $\C=\D_{\hf}$ and $(\C,\D)$ is equivalent to $(\per(X),\Db(\coh(X)))$.
\end{itemize}
\end{corollary}
\begin{proof}
(a) By \cite[Theorem 6.3]{Lunts10}, there is an object $E$  such that $\Db(\coh(X))=\thick(E)$, $B=\RHom(E,E)$ is homologically smooth, and $\RHom(E,-)\colon~\Db(\coh(X))\to\per(B)$ is a triangle equivalence . Let $L$ be the direct sum of a complete set of pairwise non-isomorphic simple objects of the heart of a fixed algebraic $t$-structure on $\Db(\coh(X))$, and let $\Gamma=\RHom(L,L)$. Then there is a triangle equivalence $\RHom(L,-)\colon\Db(\coh(X))\to\per(\Gamma)$. Moreover, $-\otimes^{\mathbf{L}}_\Gamma\RHom(E,L)\colon\D(\Gamma)\to\D(B)$ is a triangle equivalence, and hence by \cite[Lemma 3.9]{Lunts10}, $\Gamma$ is homologically smooth too. By Lemma~\ref{lem:homologically-smooth=>finite-projective-dimension}, $\per(\Gamma)\supseteq\Dfd(\Gamma)$. It is clear that $\Gamma$ satisfies the conditions (P1) and (P2) in Lemma~\ref{lem:silting-for-positive-dg-algebra} and it follows that $\Dfd(\Gamma)$ has a silting object. As $\Dfd(\Gamma)=\per(\Gamma)^{\chf}$ by Lemma~\ref{lem:compact-and-homologically-finite-objects-for-dg-algebra}, it follows that $\Db(\coh(X))^{\chf}$ has a silting object. By \cite[Proposition 7.47]{Rouquier08}, there is a fully faithful triangle functor $\mathbb{S}:\per(X)\to\Db(\coh(X))$ together with a bifunctorial isomorphism $D\Hom(\cf,\cg)\stackrel{\simeq}{\longrightarrow}\Hom(\cg,\mathbb{S}(\cf))$ for $\cf\in\per(X)$ and $\cg\in\Db(\coh(X))$. By Lemma~\ref{lem:compact-and-homologically-finite-objects-for-projetive-variety}, $\mathbb{S}$ induces a triangle equivalence $\per(X)\to\Db(\coh(X))^{\chf}$. It follows that $\per(X)$ has a silting object. Now by Corollary~\ref{cor:ST-pair-for-variety-with-silting}, $(\per(X),\Db(\coh(X)))$ is an ST-pair. By the proof of Corollary~\ref{cor:ST-pair-for-variety-with-silting}, there is a dg algebra $A$ satisfying (FN1) and (FN2) such that $(\per(X),\Db(\coh(X)))$ is equivalent to $(\per(A),\Dfd(A))$. The second statement of (a) follows from Proposition~\ref{prop:uniqueness-case:D-is-homologically-smooth}(b).

(b) follows from Proposition~\ref{prop:uniqueness-case:D-is-homologically-smooth}(b).

(c) follows from Proposition~\ref{prop:uniqueness-case:D-is-homologically-smooth} because by the proof of (a) there is a dg algebra $A$ satisfying (FN1) and (FN2) such that $\D$ is equivalent to $\Dfd(A)$ and $(\per(X),\Db(\coh(X)))$ is equivalent to $(\per(A),\Dfd(A))$.
\end{proof}

\subsection{Completing a triangulated category with silting objects}
\label{ss:completing-silting-to-ST-pair}

The aim of this subsection is to prove the following theorem.

\begin{theorem}\label{thm:silting=>ST-pair}
Let $\C$ be a $K$-linear Hom-finite Krull--Schmidt algebraic triangulated category with silting objects. Then there exist $K$-linear triangulated categories $\T$, $\C'$ and $\D$ such that $(\C',\D)$ is an ST-pair inside $\T$ and $\C'$ is triangle equivalent to $\C$.
\end{theorem}
\begin{proof}
Let $M$ be a silting object of $\C$. Then by \cite[Theorem 3.8 b)]{Ke4}, there exists a dg $K$-algebra $A$  together with a triangle equivalence from $\per(A)$ to $\C$ taking $A$ to $M$. Therefore
\[
H^p(A)=\Hom_{\per(A)}(A,A[p])\simeq \Hom_\C(M,M[p])
\]
vanishes for $p>0$ and is finite-dimensional for any $p\in\mathbb{Z}$. By Proposition~\ref{prop:ST-pair-for-non-positive-dg-algebra} below, the categories $\T=\Dfd^-(A)$, $\C'=\per(A)$ and $\D=\Dfd(A)$ satisfy the desired properties.
\end{proof}

The following result generalises Lemmas~\ref{lem:ST-pair-for-fd-algebra}~and~\ref{lem:standard-ST-pair-for-smooth-non-positive-dg-algebra}. We point out that in general $\per(A)$ and $\Dfd(A)$ are not comparable. For example, let $A=K[x,y]/(y^2)$ with $x$ and $y$ in degree $-1$ and with trivial differential. Then $A\notin\Dfd(A)$ and $A/(x,y)\notin\per(A)$.

\begin{proposition}\label{prop:ST-pair-for-non-positive-dg-algebra}
Let $A$ be a dg $K$-algebra satisfying
\begin{itemize}
\item[(N1)] $H^p(A)=0$ for any $p>0$,
\item[(N2)] $H^p(A)$ is finite-dimensional for any $p\in\mathbb{Z}$.
\end{itemize}
Then the following statements hold.
\begin{itemize}
\item[(a)] $\per(A)$ and $\Dfd(A)$ are Hom-finite.
\item[(b)] $(\per(A),\Dfd(A),A)$ is an ST-triple inside $\Dfd^-(A)$ (defined in Example~\ref{ex:t-structure-for-non-positive-dg-algebra}).
\end{itemize}
\end{proposition}
\begin{proof} 
(a) is a consequence of (b) by Remark~\ref{rem:ST1}(d).

(b) By Lemma~\ref{lem:non-positive-dg-algebra-and-silting}, $A$ is a silting object of $\per(A)$. Moreover $A$ belongs to $\Dfd^-(A)$. We will verify the three conditions (ST1--3) for the silting object $A$.

(ST1): For $N\in\Dfd^-(A)$, the space $\Hom(A,N)=H^{0}(N)$ is finite-dimensional.

(ST2):  Recall that in Example~\ref{ex:t-structure-for-non-positive-dg-algebra} we defined a $t$-structure $(\Dfd^{-,\leq 0},\Dfd^{-,\geq 0})$ on $\Dfd^-(A)$. Thanks to the formula~\eqref{eq:Hom-space-from-free-dg-module}, we have $\Dfd^{-,\leq 0}=(\Dfd^-)_A^{\leq 0}$ and $\Dfd^{-,\geq 0}=(\Dfd^-)_A^{\geq 0}$. So $((\Dfd^-)_A^{\leq 0},(\Dfd^-)_A^{\geq 0})$ is a $t$-structure on $\Dfd^-(A)$.

(ST3): This is clear due to the equalities $\Dfd^{-,\leq 0}=(\Dfd^-)_A^{\leq 0}$ and $\Dfd^{-,\geq 0}=(\Dfd^-)_A^{\geq 0}$.
\end{proof}

The following result shows that the ST-pair in Lemma~\ref{lem:standard-ST-pair-for-smooth-non-positive-dg-algebra} is a `normal form' for ST-pairs $(\C,\D)$ with $\C\supseteq\D$.

\begin{corollary}
Let $(\C,\D)$ be an ST-pair such that $\C\supseteq\D$ and $\C$ is an algebraic triangulated category. Then there is dg $K$-algebra $\Gamma$ satisfying (SN1--3) in Section~\ref{ss:ST-pair-for-smooth-non-positive-dg-algebra} such that $(\C,\D)$ is equivalent to $(\per(\Gamma),\Dfd(\Gamma))$.
\end{corollary}
\begin{proof}
As shown in the proof of Theorem~\ref{thm:silting=>ST-pair}, there is a dg $K$-algebra $\Gamma$ satisfying (N1) and (N2) (which imply (SN1) and (SN2)) such that $\C$ is triangle equivalent to $\per(\Gamma)$. Therefore $(\per(\Gamma),\per(\Gamma)^{\chf})$ is an ST-pair equivalent to $(\C,\D)$. But by Proposition~\ref{prop:ST-pair-for-non-positive-dg-algebra}, $(\per(\Gamma),\Dfd(\Gamma))$ is also an ST-pair. So by Lemma~\ref{lem:uniqueness-case-CcontainsD} we have $\Dfd(\Gamma)=\per(\Gamma)^{\chf}$, in particular, (SN3) is satisfied. The proof is complete.
\end{proof}

The rest of this subsection is devoted to a discussion on the uniqueness of the ST-pair constructed in the proof of Theorem~\ref{thm:silting=>ST-pair}. If $\C$ has a tilting object (respectively, the dg algebra $A$ is homologically smooth), then by Proposition~\ref{prop:tilting=>unique-ST-pair} (respectively, Proposition~\ref{prop:uniqueness-case:C-is-homologically-smooth}) such an ST-pair must be unique up to equivalence. In general we do not know if this is the case. 
The dg algebra $A$ is the `dg endomorphism algebra' of the silting object $M$ in a certain dg category. As $\C$ is algebraic, there is a dg $K$-category $\ca$ whose homotopy category $H^0\ca$ is triangle equivalent to $\C$ (a standard way to construct such a dg category is given in the proof of \cite[Theorem 4.3]{Ke1}). Such a dg category is called a \emph{dg enhancement} of $\C$. The dg algebra $A$ depends both on the silting object $M$ and on the dg enhancement. If $\C=\Kb(\proj\Lambda)$ for a finite-dimensional $K$-algebra $\Lambda$, $\ca$ is the dg category of bounded complexes of finitely generated projective $\Lambda$-modules and $M=\Lambda$, then $A=\Lambda$ and the resulting ST-pair is exactly $(\Kb(\proj\Lambda),\Db(\mod\Lambda))$.

(1) Let $\ca$ be a dg enhancement of $\C$. Let $M$ be a silting object of $\C$ and $\tilde{M}$ a lift of $M$ in $\ca$. Then by construction $A=\End_\ca(\tilde{M})$. If $N$ is another silting object of $\C$ and $B=\End_\cb(\tilde{N})$, then $-{\ten}^{\mathbf{L}}_A\Hom_\ca(\tilde{N},\tilde{M})\colon \D(A)\to\D(B)$ is a triangle equivalence. So it restricts to triangle equivalences $\Dfd^-(A)\to\Dfd^-(B)$, $\Dfd(A)\to\Dfd(B)$ and $\per(A)\to\per(B)$. As a consequence, the ST-pair $(\per(A),\Dfd(A))$ inside $\Dfd^-(A)$ is equivalent to the ST-pair $(\per(B),\Dfd(B))$ inside $\Dfd^-(B)$.

(2) We say that $\C$ \emph{has a unique dg enhancement} \cite[Definition 2.2]{LuntsOrlov10} if for any two dg enhancements $\ca$ and $\cb$ there is a quasi-functor $\phi:\ca\to\cb$  such that $H^0\phi\colon H^0\ca\to\H^0\cb$ is a triangle equivalence.

\begin{lemma}\label{lem:unique-dg-enhancement=>unique-ST-pair-silting-case}
Assume that $\C$ has a unique dg enhancement. Then the ST-pair constructed in the proof of Theorem~\ref{thm:silting=>ST-pair} is unique up to equivalence.
\end{lemma}
\begin{proof}
Let $\ca$ and $\cb$ be two dg enhancements of $\C$. Let $M\in\C$ be a silting object and let $\tilde{M}$ and $\hat{M}$ be lifts of $M$ in $\ca$ and $\cb$, respectively. Let $A=\End_\ca(\tilde{M})$ and $B=\End_\cb(\hat{M})$. By assumption, there is a quasi-functor $\phi:\ca\to\cb$  such that $H^0\phi\colon H^0\ca\to\H^0\cb$ is a triangle equivalence. As a consequence, $\phi(\tilde{M},\hat{M})\colon A\to B$ is a quasi-functor inducing a triangle equivalence $\D(A)\to\D(B)$. It follows that the ST-pair $(\per(A),\Dfd(A))$ inside $\Dfd^-(A)$ is equivalent to the ST-pair $(\per(B),\Dfd(B))$ inside $\Dfd^-(B)$. The proof is complete, due to the above discussion in (1).
\end{proof}

Let $A$ and $B$ be as in the proof of Lemma~\ref{lem:unique-dg-enhancement=>unique-ST-pair-silting-case}. Even if $\C$ does not have a unique dg enhancement, there is a triangle equivalence $\per(A)\to\per(B)$ as both categories are triangle equivalent to $\C$. But in general we do not know whether it is possible to extend such an equivalence to a triangle equivalence $\Dfd^-(A)\to\Dfd^-(B)$.



\section{Silting-discreteness and $t$-discreteness}

For an ST-pair $(\C,\D)$, we constructed in Section~\ref{s:order-preserving-map} an injective map $\Psi$ from the set $\silt(\C)$ of isomorphism classes of basic silting objects to the set $\tstr(\D)$ of bounded $t$-structures on $\D$. 
In this section, we study when the map is bijective by giving characterisations in terms of silting theory of $\C$ as well as the theory of $t$-structures of $\D$. 
The following theorem is the main result of this paper. 

\begin{theorem}\label{main}
\label{thm:bijectivity-vs-discreteness}
The following statements are equivalent for an ST-pair $(\C,\D)$.
\begin{itemize}
\item[(i)] The map $\Psi\colon \silt(\C) \to \tstr(\D), M\mapsto (\tD_{M}^{\le 0}, \tD_{M}^{\ge 0})$, is bijective. Namely, all bounded $t$-structures on $\D$ are $\C$-silting.
\item[(ii)] $\C$ is silting-discrete.
\item[(iii)] $\D$ is t-discrete.
\item[(iv)] The heart of every bounded $t$-structure on $\D$ has a projective generator. 
\item[(v)] The heart of any bounded $t$-structure has finitely many torsion classes, provided that it has  a projective generator. 
\end{itemize}
\end{theorem}


We will make some preparations in Sections~\ref{sub4.2}, ~\ref{ss:silting-mutations-vs-semisimple-tilts} and~\ref{ss:mutation-vs-partial-order}, and give the proof of Theorem~\ref{main} in Section~\ref{ss:proof-of-the-main-theorem}.
In Section~\ref{ss:discrete-triangulated-categories} we will apply Theorem~\ref{thm:bijectivity-vs-discreteness} to show that discrete triangulated categories in the sense of \cite{BPP3} are $t$-discrete, provided that it is part of an ST-pair.

\begin{remark}
\begin{itemize}
\item[(a)]
In cases (1) and (2) of Example~\ref{ex:algebraic-t-structure} the map $\Psi$ was known to be bijective. In case (1) it is \cite[Theorem 5]{KV88},  and in case (2) it is discussed in \cite[Section 1.6]{BPP2}.
\item[(b)] 
The equivalence between (i) and (iv) follows from \cite[Theorem 6.1]{KoY14} for the ST-pair in Lemma~\ref{lem:ST-pair-for-fd-algebra}, from \cite[Theorem 13.3]{KN2} for the ST-pair in Lemma~\ref{lem:standard-ST-pair-for-smooth-non-positive-dg-algebra} ($K=\overline{K}$) and from \cite[Theorem 1.2]{SY16} for the ST-pair in Lemma~\ref{lem:ST-pair-for-fd-non-positive-dg-algebra} ($K=\overline{K}$).
\end{itemize}
\end{remark}

\begin{corollary}\label{cor:projective-variety-not-t-discrete}
Let $X$ be a projective scheme over $K$ such that $\Db(\coh(X))$ has an algebraic $t$-structure. The following conditions are equivalent:
\begin{itemize}
\item[(i)] $\Db(\coh(X))$ is $t$-discrete,
\item[(ii)] $\per(X)$ is silting-discrete,
\item[(iii)] $\dim X=0$.
\end{itemize}
\end{corollary}
\begin{proof}
By Corollary~\ref{cor:uniqueness-for-variety}(a), the pair $(\per(X),\Db(\coh(X)))$ is an ST-pair, so the equivalence (i)$\Leftrightarrow$(ii) follows from Theorem~\ref{thm:bijectivity-vs-discreteness}.

(i)$\Rightarrow$(iii): If $\dim X>0$, then $\coh(X)$ is not a length category. So the condition (iv) in Theorem~\ref{thm:bijectivity-vs-discreteness} is not satisfied, in view of Example~\ref{ex:standard-t-structure}. Therefore $\Db(\coh(X))$ is not $t$-discrete.

(iii)$\Rightarrow$(ii): If $\dim X=0$, then $\coh(X)$ is equivalent to $\mod\Lambda$, where $\Lambda$ is the product of finitely many local commutative finite-dimensional $K$-algebras. 
 So $\per(X)\simeq \Kb(\proj\Lambda)$ is silting-discrete by Example~\ref{ex:silting-discrete-derived-categories}(1).
\end{proof}

\subsection{Two-term silting objects, intermediate $t$-structures and torsion classes}\label{sub4.2}
Let $(\C,\D)$ be an ST-pair.
In this subsection we study further the relationship between two-term silting objects, $t$-structures and torsion classes. 
This will be needed in the proof of the implication (iv)$\Rightarrow$(ii) in Theorem~\ref{main}.
\newcommand{\pg}{\mathrm{g}}
\smallskip
For a silting object $M$ of $\C$, we let 
\begin{eqnarray*}
\tstr^2_{M}(\D)\hspace{-5pt}&:=\hspace{-5pt}&\{ (\tD^{\le 0}, \tD^{\ge 0})\in \tstr(\D) \mid \tD_{M}^{\le 0}\supseteq \tD^{\le 0}\supseteq \tD_{M}^{\le -1} \},\notag\\
\tstr^{2,\pg}_{M}(\D)\hspace{-5pt}&:=\hspace{-5pt}&\{ (\tD^{\le 0}, \tD^{\ge 0})\in \tstr^2_{M}(\D) \mid \textnormal{$\tD^{0}$ has a projective generator}\}.\notag
\end{eqnarray*}

\begin{proposition}\label{restmap}
\label{prop:intermediate-t-structure-vs-torsion-pairs-silting-case}
\begin{itemize}
\item[(a)]There are mutually inverse bijections
\[
\xymatrix@C=3pc{\tstr^2_{M}(\D)\ar@/^1mm/[r]^{\ \Phi}&\tors(\tD_{M}^{0})\ar@/^1mm/[l]^{\ \ \Phi'}},
\]
where $\Phi\colon  (\tD^{\le 0}, \tD^{\ge 0})\mapsto \tD^{0}\cap \tD_{M}^{0}$ and $\Phi'\colon  \X\mapsto (\mu_{\X}^{\Le}\tD_{M}^{\le 0}, \mu_{\X}^{\Le}\tD_{M}^{\ge 0})$.
\item[(b)]The restriction of the maps $\Phi$ and $\Phi'$ gives mutually inverse bijections
\[
\xymatrix@C=3pc{\tstr_{M}^{2,\pg}(\D)\ar@/^1mm/[r]^{\ \phi}&\ftors(\tD_{M}^{0})\ar@/^1mm/[l]^{\ \ \phi'}}. 
\]
\end{itemize}
\end{proposition}

\begin{proof}
(a) This is obtained by applying Proposition~\ref{prop:woolf's-proposition} to the $t$-structure $(\D_M^{\leq 0},\D_M^{\geq 0})$.

(b) In view of (a), it is enough to show that both $\phi$ is well-defined and surjective.
Let $(\tD^{\le 0},\tD^{\ge 0})$ be in $\tstr_{M}^{2,\pg}(\D)$ and $P$ a projective generator in $\tD^{0}$.
By (a), the subcategory $\X:=\tD^{0}\cap \tD_{M}^{0}$ is a torsion class of $\tD_{M}^{0}$ such that $\tD^{\le 0}=\mu_{\X}^{\Le}\tD_{M}^{\le 0}=\tD_{M}^{\le -1}\ast \X$. Note that $\sigma_{M}^{0}(P)\in \X$. We claim that $\X=\Fac(\sigma_{M}^{0}(P))$, which shows that $\phi$ is well-defined. Because $\X$ is closed under taking factor objects, it contains $\Fac(\sigma_{M}^{0}(P))$. Conversely, let $X$ be a nonzero object in $\X$.  Since $P$ is a projective generator of $\tD^{0}$, there is an exact sequence $0\to X'\to P'\xto{f} X \to 0$ in $\tD^{0}$ with $P'\in \add(P)$ and $f\neq 0$. By Proposition \ref{BBD}(c), we obtain a triangle $X'\to P' \xto{f} X\to X'[1]$ in $\D$.
Applying $\sigma_{M}^{0}$ to this triangle, we obtain an exact sequence
$\sigma_{M}^{0}(P')\xto{\sigma_{M}^{0}(f)}\sigma_{M}^{0}(X)\to\sigma_{M}^{0}(X'[1])=0$, implying that $X\simeq \sigma_{M}^{0}(X)$ belongs to $\Fac(\sigma_{M}^{0}(P))$.

Next we show that $\phi$ is surjective. By Propositions~\ref{twosilt-fgtor} and~\ref{prop:heart-of-silting-t-structure}, there is a bijection
\[
\theta\colon \silt^2_{M}(\C) \longrightarrow \ftors(\tD_{M}^{0}),\  N\mapsto \Fac(\sigma^{0}_{M}(N)).
\]
Namely, for a finitely generated torsion class $\X$ of $\tD_{M}^{0}$, there is a unique silting object $N\in\silt^2_{M}(\C)$ such that $\X=\Fac(\sigma_{M}^{0}(N))$. By Corollary~\ref{cor:restriction-of-the-map-to-n-terms} and Proposition~\ref{prop:heart-of-silting-t-structure}(a), the $t$-structure $(\tD^{\le 0}_{N},\tD_{N}^{\ge 0})$ is in $\tstr^{2,\pg}_{M}(\D)$ and $\sigma_{N}^{0}(N)$ is a projective generator of the heart $\tD_{N}^{0}$. We will show that $\Phi(\tD^{\le 0}_{N},\tD_{N}^{\ge 0})=\X$. By the above argument, $\Phi(\tD^{\le 0}_{N},\tD_{N}^{\ge 0})=\Fac(\sigma_M^0(\sigma_N^0(N)))$. Therefore it remains to show that $\sigma_M^0(N)\simeq\sigma_M^0(\sigma_N^0(N))$. Applying $\sigma_M^0$ to the canonical triangle $\sigma_N^{\leq -1}(N)\to N\to \sigma_N^0(N)\to \sigma_N^{\leq -1}(N)[1]$ we obtain an exact sequence
\[
\sigma_M^0(\sigma_N^{\leq -1}(N))\to \sigma_M^0(N)\to \sigma_M^0(\sigma_N^0(N))\to \sigma_M^0(\sigma_N^{\leq -1}(N)[1]).
\]
The two outer terms are trivial because both $\sigma_N^{\leq -1}(N)$ and $\sigma_N^{\leq -1}(N)[1]$ belong to $\D_N^{\leq -1}\subseteq\D_M^{\leq -1}$. Therefore the two middle terms are isomorphic and the proof is complete.
\end{proof}

Let $\psi\colon  \silt^2_{M}(\C) \to \tstr_{M}^{2,\pg}(\D)$ be the restriction of $\Psi$ to $\silt^2_M(\C)$. In the proof of the surjectivity of $\phi$, we have shown that the following diagram is commutative
\[
\xymatrix{
\silt^2_M(\C)\ar[dr]^\theta\ar[d]^\psi & \\
\tstr^{2,\pg}_M(\D)\ar[r]^\phi & \ftors(\D).
}
\]
As a consequence of Proposition~\ref{restmap}, we obtain the following result.

\begin{theorem}\label{2term bij}
The map $\psi\colon  \silt^2_{M}(\C) \to \tstr_{M}^{2,\pg}(\D)$ is bijective.
\end{theorem}

Moreover, Proposition~\ref{restmap} has the following corollary.
 
\begin{corollary}\label{finiteness}
For a basic silting object $M$ of $\C$, the following are equivalent.
\begin{itemize}
\item[(i)] The set $\silt^2_{M}(\C)$ is finite.
\item[(ii)] The heart of each bounded $t$-structure in $\tstr^2_{M}(\D)$ has a projective generator.
\item[(iii)] All torsion classes of $\tD^{0}_{M}$ are finitely generated.
\item[(iv)] The set $\tors(\D_M^0)$ is finite.
\item[(v)] $\End_\T(M)$ is $\tau$-tilting finite.
\end{itemize}
\end{corollary}
\begin{proof}
By Proposition \ref{restmap}, (ii)$\Leftrightarrow$(iii) holds. 
By Proposition~\ref{prop:heart-of-silting-t-structure}(a), $\Hom_\T(M,-)$ restricts to an equivalence from $\tD_{M}^{0}$ to $\mod \End_{\T}(M)$. Therefore (i)$\Leftrightarrow$(iii)$\Leftrightarrow$(iv)$\Leftrightarrow$(v) follows from  Proposition~\ref{twosilt-fgtor}.
\end{proof}

\subsection{Silting mutations and semisimple tilts}\label{ss:silting-mutations-vs-semisimple-tilts} 
Let $(\C,\D)$ be an ST-pair. 
In this subsection, we study the relation between silting mutations of silting objects of $\C$ and semisimple tilts of $\C$-silting $t$-structures on $\D$ under the map $\Psi$ in Theorem~\ref{basic-map}. 

\smallskip
Let $(\tD^{\le 0},\tD^{\ge 0})$ be an algebraic $t$-structure on $\D$.
For a semisimple object $S$ of the heart $\tD^{0}$, we define a full subcategory $\X_{S}$ of $\D^0$ as
\[
\X_{S}:=\{ X\in \tD^{0} \mid \Hom_{\T}(X,S)=0 \}. 
\]
Then $\X_{S}$ is a torsion class of $\tD^{0}$.
The HRS-tilt $(\mu_{\X_{S}}^{\Le}\tD^{\le 0}, \mu_{\X_{S}}^{\Le}\tD^{\ge 0})$ is called the \emph{semisimple tilt} of $(\tD^{\le 0}, \tD^{\ge 0})$ with respect to $S$. It is called a \emph{simple tilt} if $S$ is simple (see \cite{Br2}). Let $\mu_{S}^{\Le}\tD^{\le 0}:=\mu_{\X_{S}}^{\Le}\tD^{\le 0}$.

Let $M=M_{1}\oplus M_{2}\oplus\ldots\oplus M_{n}$ be a basic silting object of $\C$, where each $M_{i}$ is indecomposable. Let $P_i:=\sigma_M^0(M_i)$ and let $S_i$ be the simple top of $P_i$. It follows from Proposition~\ref{prop:heart-of-silting-t-structure}(a) that $P_1,\ldots, P_n$ is a complete set of pairwise non-isomorphic indecomposable projective objects of $D_M^0$ and $S_1,\ldots,S_n$ is a complete set of pairwise non-isomorphic simple objects of $D_M^0$. The following result shows that the map $\Psi$ takes silting mutation to semisimple tilts.

\begin{proposition}\label{st-mutation}
Let $M$ be a basic silting object of $\C$ and let $I$ be a subset of $\{1, 2, \ldots, n\}$. Put $M_{I}=\bigoplus_{i\in I}M_{i}$ and $S_{I}=\bigoplus_{i\in I}S_{i}$. Then
\[
\tD_{\mu_{M_{I}}^{\Le}(M)}^{\le 0} = \mu_{S_{I}}^{\Le}\tD_{M}^{\le 0}.
\]
Namely, a silting object $N$ of $\C$ is the left mutation of $M$ with respect to $M_{I}$ if and only if the $t$-structure $(\tD_{N}^{\le 0},\tD_{N}^{\ge 0})$ is the semisimple tilt of $(\tD_{M}^{\le 0},\tD_{M}^{\ge 0})$ with respect to $S_{I}$.
\end{proposition}
\begin{proof}
Let $N$ be the left mutation of $M$ with respect to $M_{I}$.
By Proposition~\ref{prop:intermediate-t-structure-vs-torsion-pairs-silting-case}(b) and its proof, we have $\D_N^{\leq 0}=\D_M^{\leq -1}*\X(N)$, where
\[
\X(N)=\D_N^0\cap\D_M^0=\{X\in\D_M^0\mid\Hom_\T(N,X[1])=0\}=\Fac(\sigma_M^0(N)).
\]
It is enough to show that $\X_{S_I}=\X(N)$.

`$\supseteq$': It is enough to show $\sigma_M^0(N)\in\X_{S_I}$ because $\X_{S_I}$ is closed under factor objects. By the definition of left mutation, there is a triangle 
\begin{eqnarray}\label{mut-seq}
M_{I}\xto{f}M'\to M_{I}^{\ast}\to M_{I}[1],
\end{eqnarray}
where $f$ is a minimal left $\add(M/M_{I})$-approximation, such that $N=M_{I}^{\ast}\oplus(M/M_{I})$.
Applying $\Hom_\T(-,S_I)$ to this triangle, we obtain an exact sequence $\Hom_\T(M_I[1],S_I)\to \Hom_\T(M_I^*,S_I)\to \Hom_T(M',S_I)$. Both outer terms vanish by Corollary~\ref{cor:Hom-duality-between-silting-and-smc}. Therefore $\Hom_\T(M_I^*,S_I)=0$, and hence $\Hom_\T(N,S_I)=0$. By Lemma~\ref{lem:truncation-preserves-morphisms} we have $\Hom_\T(\sigma^0_M(N),S_I)=0$.

`$\subseteq$':  Let $X\in\X_{S_I}$, \ie $\Hom(X,S_I)=0$, and let $P_X\to X$ be a projective cover of $X$ in $\D_M^0$. Then $P_X$ belongs to $\add(P/P_I)$, where $P=\sigma^0_M(M)$ and $P_I=\sigma^0_M(M_I)=\bigoplus_{i\in I}P_i$. So any morphism $P_I\to X$ factors through $P_X$, and further through $P':=\sigma_M^0(M')$, as by Lemma~\ref{lem:truncation-preserves-morphisms} the morphism $\sigma^0_M(f)\colon P_I\to P'$ is a left $\add(P/P_I)$-approximation of $P_I$. Therefore the map $\Hom(\sigma_M^0(f),X)$ is surjective. By Lemma~\ref{lem:truncation-preserves-morphisms} again, the map $\Hom(f,X)$ is surjective. Now applying $\Hom(-,X)$ to the triangle \eqref{mut-seq}, we obtain that $\Hom(M_I^*,X[1])=0$. Consequently, $\Hom(N,X[1])=0$.
\end{proof}

\subsection{The covering relation of bounded $t$-structures}
\label{ss:mutation-vs-partial-order}
Let $(\C,\D)$ be an ST-pair.
In this subsection we prove Proposition~\ref{keythm} below, which will be needed in the proof of the implication (ii)$\Rightarrow$(i) of Theorem~\ref{main}.

\begin{proposition}\label{keythm}
\label{prop:mutations-vs-partial-order}
Let $M$ be a basic silting object of $\C$ and $(\tD^{\le 0}, \tD^{\ge 0})$ a bounded $t$-structure on $\D$.
If $\tD_{M}^{\le 0}\supsetneq \tD^{\le 0}$, then there exists a left mutation $N$ of $M$ such that 
\[
\tD_{M}^{\le 0}\supsetneq\tD_{N}^{\le 0}\supseteq \tD^{\le 0}.\notag
\]
\end{proposition}

\begin{proof}
First, \emph{$\D_M^{\leq 0}\cap\D^{\geq 1}$ is non-zero}. Otherwise, for any $X\in\D_M^{\leq 0}$, consider the canonical triangle $\sigma^{\leq 0}X\to X\to \sigma^{\geq 1}X\to \sigma^{\leq 0}(X)[1]$ associated to the $t$-structure $(\D^{\leq 0},\D^{\geq 0})$. Because $\sigma^{\leq 0}(X)[1]\in \D^{\leq -1}\subseteq\D_M^{\leq -1}$, we obtain that $\sigma^{\geq 1}(X)\in\D_M^{\leq 0}\cap \D^{\geq -1}$. By assumption, $\sigma^{\geq 1}(X)=0$, and therefore $X\simeq \sigma^{\leq 0}(X)\in \D^{\leq 0}$, a contradiction.

Secondly, \emph{$\D_M^{0}\cap\D^{\geq 1}$ is non-zero}. Let $X$ be a non-zero object of $\D_M^{\leq 0}\cap\D^{\geq 1}$. Let $l$ be the maximal integer satisfying $\sigma_M^{0}(X[-l])\neq 0$, and consider the triangle $\sigma_M^{0}(X[-l])\to X[-l]\to \sigma^{\geq 1}_M(X[-l])\to \sigma_M^{0}(X[-l])[1]$. Then both $X[-l]$ and $\sigma^{\geq 1}_M(X[-l])$ are contained in $\D^{\geq 1}$, so does $\sigma_M^{0}(X[-l])$. Consequently, $\sigma_M^{0}(X[-l])$ is a non-zero object of $\D_M^{0}\cap \D^{\geq 1}$.

Thirdly, \emph{$\D_M^0\cap\D^{\geq 1}$ is closed under subobjects in $\D_M^0$}. Let $X\in \D_M^0\cap\D^{\geq 1}$ and let $X'$ be a subobject of $X$ in $\D_M^0$. Then there is a short exact sequence $0\to X'\to X\to X''\to 0$ in $\D_M^0$. It yields a triangle $X''[-1]\to X'\to X\to X''$ in $\D$. Since $X''\in \D_M^0$, we have $X''[-1]\in \D_M^{\geq 1}\subseteq \D^{\geq 1}$. It follows that $X'$ belongs to $\D^{\geq 1}$ and hence to $\D_M^0\cap\D^{\geq 1}$.

Finally, by the preceding two steps, there is a simple object $S$ of $\D_M^0$ belonging to $\D^{\geq 1}$. Let $X\in\D^{\leq 0}$. Then $\Hom(X,S)=0$, which, by Lemma~\ref{lem:truncation-preserves-morphisms}, implies $\Hom(\sigma^0_M(X),S)=0$, \ie $\sigma^0_M(X)\in \X_S$, so $X\in\sigma_M^{\leq -1}(X)*\sigma_M^0(X)\subseteq \D_M^{\leq -1}\ast \X_S=\mu_S^\Le\D_M^{\leq 0}$. Therefore $\D^{\leq 0}\subseteq \mu_S^\Le\D_M^{\leq 0}$. Let $N$ be the left mutation of $M$ with respect to the indecomposable direct summand of $M$ corresponding to $S$, then $\D_N^{\leq 0}=\mu_S^\Le\D_M^{\leq 0}$ by Proposition~\ref{st-mutation}. The inclusion $\D_M^{\leq 0}\supseteq\D_N^{\leq 0}$ is strict because $M\gneq N$. The proof is complete.
\end{proof}

\subsection{Proof of Theorem~\ref{main}}
\label{ss:proof-of-the-main-theorem}
This subsection is devoted to the proof of Theorem~\ref{main}.

\begin{proof}
(i)$\Rightarrow$(iv): This holds true because each $\C$-silting $t$-structure on $\D$ has a projective generator by Proposition \ref{prop:heart-of-silting-t-structure}(a).

(iv)$\Rightarrow$(ii): Let $M$ be an arbitrary silting object of $\C$. By the condition (iv), the heart of each bounded $t$-structure in $\tstr^2_{M}(\D)$ has a projective generator. It follows from Corollary \ref{finiteness}(ii)$\Rightarrow$(i) that the set $\silt^2_{M}(\C)$ is finite. We obtain (ii) by applying Lemma \ref{AM}. 

(ii)$\Rightarrow$(i): Assume that $(\tD^{\le 0}, \tD^{\ge 0})$ is a bounded $t$-structure on $\D$ which is not $\C$-silting. Let $M$ be a silting object of $\C$. Then by Lemma~\ref{lem:intermediate-t-structures-wrt-algebraic-t-structure}, 
 there exist integers $l>k$ such that
$\tD_{M}^{\le -k}\supseteq \tD^{\le 0}\supseteq \tD_{M}^{\le -l}$.  Both inclusions are strict because $(\tD^{\le 0}, \tD^{\ge 0})$ is not $\C$-silting.
Up to shift we may assume that $k=0$.
By Proposition \ref{keythm}, there exists a basic silting object $N_{1}$ such that $\tD_{M}^{\le 0}\supsetneq \tD_{N_{1}}^{\le 0}\supsetneq \tD^{\le 0}$.
Applying Proposition \ref{keythm} repeatedly, we obtain an infinite sequence of subcategories
\[
\tD_{M}^{\le 0} \supsetneq \tD_{N_{1}}^{\le 0}\supsetneq \tD_{N_{2}}^{\le 0} \supsetneq \cdots (\supsetneq \tD^{\le 0}\supsetneq \tD_{M}^{\le -l}),
\]
and hence an infinite sequence of silting objects
\[
M\gneq N_{1}\gneq N_{2} \gneq\cdots (\gneq M[l]). 
\]
This contradicts the fact that $\C$ is silting-discrete.

(iii)$\Rightarrow$(v): Let $\H$ be the heart of a bounded $t$-structure on $\D$. By Proposition~\ref{prop:woolf's-proposition}, if $\D$ is $t$-discrete, then $\tors(\D)$ is finite. In particular, (v) is satisfied.

(v)$\Rightarrow$(ii): Let $M$ be any silting object of $\C$. Then $\tors(\D_M^0)$ is finite, and hence $\silt^2_M(\C)$ is finite by Proposition~\ref{finiteness}(iv)$\Rightarrow$(i). We obtain (ii) by Lemma~\ref{AM}.

(i)$\Rightarrow$(iii): By (i), any bounded $t$-structure on $\D$ is of the form $(\D_M^{\leq 0},\D_M^{\geq 0})$ for some basic silting object $M$ of $\C$, and moreover, by Corollary~\ref{cor:restriction-of-the-map-to-n-terms}, the map $\Psi$ restricts to a bijection $\silt^n_{M}(\C)\rightarrow\tstr^n_{M}(\D)$ for any positive integer $n$.
If $\C$ is silting-discrete, then $\silt^n_{M}(\C)$ is finite, and hence $\tstr^n_{M}(\D)$ is finite. Therefore $\D$ is $t$-discrete.
\end{proof}


\subsection{Discrete triangulated categories are $t$-discrete}
\label{ss:discrete-triangulated-categories}

Assume that $K$ is algebraically closed and let $(\C,\D)$ be an ST-pair. 

Let $\H$ be the heart of a fixed bounded $t$-structure on $\D$, $\sigma^0$ the associated cohomology functor and $\sigma^i=\sigma^0\circ[i]$. For an object $X$ of $\H$, let $[X]$ denote the class of $X$ in the Grothendieck group $K_0(\H)$. For an object $X$ of $\D$, define a function $v_X\colon\mathbb{Z}\to K_0(\H)$ by $v_X(i)=[\sigma^i(X)]$ for $i\in\mathbb{Z}$. Following \cite{BPP3}, we say that $\D$ is \emph{$\H$-discrete} if for any function $v\colon\mathbb{Z}\to K_0(\H)$ the set of isomorphism classes of indecomposable objects $X$ of $\D$ satisfying $v_X=v$ is finite.

\begin{theorem}\label{thm:discreteness=>t-discreteness}
Keep the above notation and assumptions. Assume that $\H$ is a length category with finitely many isomorphism classes of simple objects and that $\D$ is $\H$-discrete. Then  $\D$ is $t$-discrete. Moreover, $\D$ is $\H'$-discrete for the heart $\H'$ of any bounded $t$-structure on $\D$.
\end{theorem}
\begin{proof}
Let $K_0^+(\H)$ be the subset of $K_0(\H)$ consisting of $[X]$ with $X\in\H$. For two functions $v,w\colon\mathbb{Z}\to K_0(\H)$, define $v\geq w$ if $v(i)-w(i)\in K_0^+(\H)$ for any $i\in\mathbb{Z}$. Let $S_1,\ldots,S_n$ be a complete set of pairwise non-isomorphic simple objects of $\H$. Then $K_0(\H)$ is a free abelian group with basis $[S_1],\ldots,[S_n]$, and $K_0^+(\H)$ consists of those elements with non-negative entries with respect to this basis. It follows that for any two functions $v,w\colon\mathbb{Z}\to K_0(\H)$ the interval $[w,v]:=\{u\colon\mathbb{Z}\to K_0(\H)|w\leq u\leq v\}$ is finite.

We will show that the condition (v) of Theorem~\ref{thm:bijectivity-vs-discreteness} is satisfied. Then the $t$-discreteness of $\D$ follows. Let $\H'$ be the heart of a bounded $t$-structure and assume that $\H'$ has a projective generator. Then $\H'$ is equivalent to $\mod\Lambda'$ for some finite-dimensional $K$-algebra $\Lambda'$. Let $X$ be an object of $\H'$ and let $0\to X'\to X\to X''\to 0$ be a short exact sequence in $\H'$. Then there is a triangle $X'\to X\to X''\to X'[1]$ in $\D$ by Proposition~\ref{BBD}(c). Since $\sigma^0$ is a cohomological functor, it follows that $v_X\leq v_{X'}+v_{X''}$. Therefore by induction we have $0\leq v_X\leq v_{X_1}+\ldots+v_{X_l}$, where $X_1,\ldots,X_l$ are simple factors of $X$. As discussed above, the interval $[0,v_{X_1}+\ldots+v_{X_l}]$ is finite. So due to the $\H$-discreteness of $\D$, the number of isomorphism classes of indecomposable objects $X$ with simple factors $X_1,\ldots,X_l$ is finite. Consequently, the number of isomorphism classes of indecomposable $\Lambda'$-modules with a fixed dimension vector is finite. So by the validity of the second Brauer--Thrall conjecture (see for example~\cite[Section IV.5]{AssemSimsonSkowronski}), $\Lambda'$ is representation-finite. In particular, $\tors(\H')=\tors(\mod\Lambda')$ is finite.

\smallskip
Now let $\H'$ be the heart of any bounded $t$-structure on $\D$. Then $\H'$ has a projective generator by Theorem~\ref{thm:bijectivity-vs-discreteness}. By the above discussion, $\H'$ has only finitely many isomorphism classes of indecomposable objects. The $\H'$-discreteness of $\D$ then follows by \cite[Theorem 2.5]{BPP3}.
\end{proof}

By applying Theorem~\ref{thm:discreteness=>t-discreteness} to the ST-pair $(\Kb(\proj\Lambda),\Db(\mod\Lambda))$ we obtain the following corollary. Recall that a finite-dimensional $K$-algebra $\Lambda$ is \emph{derived-discrete} if $\Db(\mod\Lambda)$ is $\mod\Lambda$-discrete (\cite[Section 2]{BPP3}).

\begin{corollary}[{\cite[Proposition 3.2]{BPP3}}] 
Let $\Lambda$ be a finite-dimensional derived-discrete algebra and let $\H$ be the heart of  any bounded $t$-structure on $\Db(\mod\Lambda)$. Then $\Db(\mod\Lambda)$ is $\H$-discrete.
\end{corollary}

\section{Silting-discrete Calabi--Yau categories}\label{section calabi-yau}

In this section, we study the silting-discreteness for relative Calabi--Yau triangulated categories. More precisely, we give a characterisation of silting-discreteness in terms of cluster-tilting theory. Moreover, we show that the perfect derived category of a derived preprojective algebra of a quiver is silting-discrete if and only if the quiver is Dynkin. 

\medskip
First we recall the notion of Calabi--Yau pairs. Fix an integer $d\ge 1$. 

\begin{definition}\label{def CYtriple}
Let $(\C,\D)$ be an ST-pair inside $\C$. 
We call $(\C,\D)$ a \emph{$(d+1)$-Calabi--Yau pair} if there exists a bifunctorial isomorphism for $X\in\D$ and $Y\in\C$:
\[
D\Hom_{\C}(X,Y)\simeq \Hom_{\C}(Y,X[d+1]). 
\]
\end{definition}

If $M$ is a silting object of $\C$, then $(\C,\D,M)$ is a \emph{$(d+1)$-Calabi--Yau triple} in the sense of \cite[Section 5.1]{IYa14}. 
Note that, for silting objects $M$ and $N$, $(\C,\D,M)$ is a $(d+1)$-Calabi--Yau triple if and only if $(\C,\D,N)$ is a $(d+1)$-Calabi--Yau triple.

\begin{example}\label{ex:CY-triple-from-CY-algebra}
Let $\Gamma$ be a bimodule $(d+1)$-Calabi--Yau dg algebra such that $H^p(\Gamma)=0$ for $p>0$ and $H^0(\Gamma)$ is finite-dimensional. Then by Lemmas~\ref{lem:homologically-smooth=>finite-projective-dimension},~\ref{lem:topologically homologically-smooth=>finite-projective-dimension} and~\ref{lem:standard-ST-pair-for-smooth-non-positive-dg-algebra}, $(\per(\Gamma),\Dfd(\Gamma),\Gamma)$ is a $(d+1)$-Calabi--Yau triple.
\end{example}

Let  $(\C,\D)$ be a $(d+1)$-Calabi--Yau pair. 
Consider the triangle quotient 
\[
\U:=\C/\D, 
\]
which is called the \emph{cluster category} (\cite[Section 5.3]{IYa14}). Let $\pi\colon  \C\to \U$ be the canonical projection functor.
We call $T\in\U$ a \emph{$d$-cluster tilting object} 
if 
\begin{eqnarray*}
\add(T)
\hspace{-3pt}&=\hspace{-3pt}& \{X\in \U\ |\ \Hom_{\U}(X,T[i])=0\ \textnormal{for}\ 1\leq i\leq d-1\}\notag\\
\hspace{-3pt}&=\hspace{-3pt}&\{X\in \U\ |\ \Hom_{\U}(T,X[i])=0\ \textnormal{for}\ 1\leq i\leq d-1\}.
\end{eqnarray*}
We denote by $\ctilt{d}(\U)$ the set of isomorphism classes of basic $d$-cluster tilting objects of $\U$.
Then we have the following result.

\begin{theorem}[{\cite[Theorem 5.8 and Corollary 5.12]{IYa14}}]\label{IY cluster} For a $(d+1)$-Calabi--Yau triple $(\C,\D,M)$, the following  statements hold.
\begin{itemize}
\item[(a)] The category $\U$ is a $d$-Calabi--Yau triangulated category.
\item[(b)] The functor $\pi\colon \C\to \U$ induces an equivalence $M^{[0,d-1]}\to \U$.
\item[(c)] The restriction of $\pi$ induces an injection
\[
\phi\colon \silt^d_{M}(\C)\to \ctilt{d}(\U),
\]
which is a bijection if $d=1$ or $d=2$.
\end{itemize}
\end{theorem}

Then we give a criterion for $\C$ being silting-discrete in terms of the cluster category $\U$.

\begin{theorem}\label{equivalent thm}
For a $(d+1)$-Calabi--Yau pair $(\C,\D)$, the following statements hold.
\begin{itemize}
\item[(a)] Assume $d\geq 2$. If $\ctilt{d}(\U)$ is a finite set, then $\C$ is silting-discrete. The converse holds true if $d=2$.
\item[(b)]
Assume that $d=1$ or $2$ and let $M$ be any basic silting object of $\C$. The following are equivalent. 
\begin{itemize}
\item[(i)] $\C$ is silting-discrete. 
\item[(ii)] $\silt^2_{M}(\C)$ is a finite set.
\item[(iii)] $\End_\C(M)$ is $\tau$-tilting finite.
\end{itemize}
\end{itemize}
\end{theorem}

\begin{proof}
(a) Assume that $\ctilt{d}\U$ is finite. Let $M$ be any basic silting object of $\C$. By Theorem \ref{IY cluster}(c) applied to the $(d+1)$-Calabi--Yau triple $(\C,\D,M)$, we obtain that $\silt^d_{M}(\C)$ is finite, which implies that $\silt^2_{M}(\C)$ is finite. Therefore, by Lemma \ref{AM},  $\C$ is silting-discrete.

If $d=2$,  then according to Theorem \ref{IY cluster}(c) there is a bijection between $\silt^2_{M}(\C)$ and $\ctilt{2}(\U)$. If $\C$ is silting-discrete, then $\silt^2_{M}(\C)$ is finite. Hence $\ctilt{2}(\U)$ is finite.

(b) 
By Proposition~\ref{twosilt-fgtor}(b), (ii) and (iii) are equivalent. Next we show that (i) and (ii) are equivalent.

Case $d=2$: By Theorem \ref{IY cluster}(c), there is a bijection between $\silt^2_{M}(\C)$ and $\ctilt{2}(\U)$. It follows that $\silt^2_{M}(\C)$ is finite if and only if $\ctilt{2}(\U)$ is finite. The equivalence of (i) and (ii) then follows by (a).

Case $d=1$:  Let $N$ be any basic silting object  of $\C$. 
Then $\pi(N)$ and $\pi(M)$ are basic 1-cluster-tilting objects by Theorem~\ref{IY cluster}(c), which is unique up to isomorphism.  
Moreover, since $\End_\C(N)\simeq\End_\U(\pi(N))$ and $\End_\C(M)\simeq\End_\U(\pi(M))$ by Theorem~\ref{IY cluster}(b), it follows that $\End_\C(N)\simeq\End_\C(M)$. Hence by Proposition~\ref{twosilt-fgtor}, $\silt^2_{N}(\C)$ is in bijection with $\silt^2_{M}(\C)$. The equivalence of (i) and (ii) then follows by Lemma~\ref{AM}.
\end{proof}

\subsection{Derived preprojective algebras}
\label{ss:derived-preprojective-algebra}
In this subsection we show that the perfect derived category of a derived preprojective algebra associated with a quiver is silting-discrete if and only if the quiver is Dynkin.

\smallskip
Let $Q$ be a finite quiver. Define a graded quiver $\tilde{Q}$ by
\begin{itemize}
\item[-] $\tilde{Q}$ has the same vertices as $Q$;
\item[-] $\tilde{Q}$ has three types of arrows:
\begin{itemize}
\item[$\cdot$] the arrows of $Q$, in degree $0$,
\item[$\cdot$] $\alpha^*\colon j\to i$ in degree $-d+1$, for each arrow $\alpha\colon  i\to j$ of $Q$,
\item[$\cdot$] $t_i\colon i\to i$ in degree $-d$, for each vertex $i$ of $Q$.
\end{itemize}
\end{itemize}
According to \cite[Section 4.1]{Ke2}, the \emph{derived $(d+1)$-preprojective algebra} $\Gamma:=\Gamma_{d+1}(Q)$ is the dg algebra $(K\tilde{Q},d)$, where $K\tilde{Q}$ is the graded path algebra of $\tilde{Q}$ and $d$ is the unique $K$-linear differential which satisfies the graded Leibnitz rule
\[
d(ab)=d(a)b+(-1)^{p}ad(b),
\]
where $a$ is homogeneous of degree $p$, and which takes the following values
\begin{itemize}
\item[$\cdot$] $d(e_i)=0$ for any vertex $i$ of $Q$, where $e_i$ is the trivial path at $i$,  
\item[$\cdot$] $d(\alpha)=0$ for any arrow $\alpha$ of $Q$,
\item[$\cdot$] $d(\alpha^*)=0$ for any arrow $\alpha^*$ of $Q$,
\item[$\cdot$] $d(t_i)=e_i\sum_\alpha (\alpha\alpha^*-\alpha^*\alpha)e_i$ for any vertex $i$ of $Q$, where the sum is over all arrows of $Q$.
\end{itemize}
Note that if $d=1$, then $H^{0}(\Gamma)$ is the preprojective algebra associated with $Q$, and if $d\ge 2$, then $H^{0}(\Gamma)$ is the path algebra of $Q$. See \cite{Ringel98} for the definition of preprojective algebras.

\begin{lemma}\label{lem:Hom-finiteness-and-cy-property-for-derived-preprojective-algebras}
The following three conditions are equivalent:
\begin{itemize}
\item[(i)]
$\per(\Gamma)$ is a Hom-finite Krull--Schmidt triangulated $K$-category,
\item[(ii)] 
$H^0(\Gamma)$ is finite-dimensional,
\item[(iii)]
$d=1$ and $Q$ is Dynkin, or 
$d\geq 2$ and $Q$ has no oriented cycles.
\end{itemize}
If these conditions are satisfied, then $(\per(\Gamma),\Dfd(\Gamma),\Gamma)$ is a $(d+1)$-Calabi--Yau triple.
\end{lemma}
\begin{proof}
According to \cite[Theorem 6.3]{Ke2}, $\Gamma$ is homologically smooth and bimodule $(d+1)$-Calabi--Yau. It is clear that $H^p(\Gamma)=0$ for $p>0$. So in view of Example~\ref{ex:CY-triple-from-CY-algebra}, the second statement is a consequence of the first statement.
Because $\per(\Gamma)$ is idempotent complete, it is Hom-finite Krull--Schmidt if and only if it is Hom-finite.

(i)$\Rightarrow$(ii): This is because $H^0(\Gamma)=\Hom_{\per(\Gamma)}(\Gamma,\Gamma)$. 

(ii)$\Rightarrow$(i): This follows from \cite[Proposition 2.5]{KaY16} (see Lemma~\ref{lem:standard-ST-pair-for-smooth-non-positive-dg-algebra}).

(ii)$\Leftrightarrow$(iii): 
If $d=1$, then $H^0(\Gamma)$ is the preprojective algebra associated with $Q$, which is finite-dimensional if and only if $Q$ is Dynkin (by \cite[Proposition 5.2]{DlabRingel80}). 
If $d\geq 2$, then $H^0(\Gamma)$ is the path algebra $KQ$ of $Q$, which is finite-dimensional if and only if $Q$ has no oriented cycles.
\end{proof}

Now we apply Theorem~\ref{equivalent thm} to perfect derived categories of derived preprojective algebras.

\begin{corollary}\label{main3-1} Let $Q$ be a finite quiver and $\Gamma=\Gamma_{d+1}(Q)$. Assume that $K$ is algebraically closed and that $H^0(\Gamma)$ is finite-dimensional.
Then $\per(\Gamma)$ is silting-discrete if and only if  $Q$ is Dynkin.
\end{corollary}

\begin{proof} By Lemma~\ref{lem:Hom-finiteness-and-cy-property-for-derived-preprojective-algebras}, we have $d=1$ and $Q$ is Dynkin, or $d\geq 2$ and $Q$ has no oriented cycles.
 
Case 1: $d\geq 2$ and $Q$ has no oriented cycles. Assume that $Q$ is a Dynkin quiver. By \cite[Corollary 3.4]{Gu11}, $\per(\Gamma)/\Dfd(\Gamma)$ is triangle equivalent to the orbit category $\U_Q$ of the bounded derived category $\Db(\mod KQ)$ under the action of the automorphism $\tau^{-1}[d]$, where $\tau$ is the Auslander-Reiten translation. By \cite[Proposition 2.2(d)]{Z08}, $\U_Q$ has only finitely many isomorphism classes of indecomposable objects, and hence $\ctilt{d}(\U_Q)$ is finite. By Theorem~\ref{equivalent thm}(a), $\per(\Gamma)$ is silting-discrete.

Assume that $Q$ is not Dynkin. Then there are infinitely many bricks in $\mod KQ$, so  $KQ$ is not $\tau$-tilting finite, by \cite[Theorem 1.10]{DIJ15}. 
Since
$\End_{\per(\Gamma)}(\Gamma)=H^0(\Gamma)=KQ$, it follows from Theorem~\ref{equivalent thm}(b) that $\per(\Gamma)$ is not silting-discrete.

Case 2: $d=1$ and $Q$ is Dynkin. It follows from \cite[Theorem 2.21]{M14} and Proposition~\ref{twosilt-fgtor} that $\silt^2_\Gamma(\per(\Gamma))$ is finite. By Theorem~\ref{equivalent thm}(b), $\per(\Gamma)$ is silting-discrete. 
\end{proof}

\subsection{Complete Ginzburg dg algebras}\label{ss:ginzburg-dg-algebras}
In this subsection we show that the perfect derived category of the complete Ginzburg dg algebra associated with a quiver with a nondegenerate potential is silting-discrete if and only if the quiver is mutation equivalent to a Dynkin quiver. We refer to \cite{DWZ08} for the definition and properties of quiver mutation and mutation of quivers with potential.

Let $Q$ be a finite quiver and $W$ a  potential. Let $\widehat{\Gamma}(Q,W)$ be the complete Ginzburg dg algebra associated with the quiver with potential $(Q,W)$, see \cite{Gi06,KeY11}. The algebra $H^0\widehat{\Gamma}(Q,W)$ is known as the \emph{Jacobian algebra}, and if it is finite-dimensional, we say that  $(Q,W)$ is \emph{Jacobi-finite}.
By definition $\widehat{\Gamma}(Q,W)$ is concentrated in non-positive degrees. According to \cite[Theorem A.17]{KeY11}, $\widehat{\Gamma}(Q,W)$ is topologically homologically smooth and 3-Calabi--Yau as a bimodule. So if $(Q,W)$ is Jacobi-finite, then by Example~\ref{ex:CY-triple-from-CY-algebra} we obtain a $3$-Calabi--Yau triple $(\per(\widehat{\Gamma}(Q,W)),\Dfd(\widehat{\Gamma}(Q,W)),\widehat{\Gamma}(Q,W))$. Similar to Lemma~\ref{lem:Hom-finiteness-and-cy-property-for-derived-preprojective-algebras}(i)$\Leftrightarrow$(ii), $\per(\widehat{\Gamma}(Q,W))$ is Hom-finite if and only if $(Q,W)$ is Jacobi-finite.
Put 
\[\U_{(Q,W)}:=\per(\widehat{\Ga}(Q,W))/\D_{\fd}(\widehat{\Ga}(Q,W)).\]

\begin{corollary}\label{main3-2}
Let $(Q,W)$ be a Jacobi-finite quiver with potential. 
\begin{itemize}
\item[(a)] 
The following conditions are equivalent.
\begin{itemize}
\item[(i)] $\per(\widehat{\Gamma}(Q,W))$ is silting-discrete.
\item[(ii)] $\ctilt{2}(\U_{(Q,W)})$ is finite,
\item[(iii)] $H^0\widehat{\Gamma}(Q,W)$ is $\tau$-tilting finite.
\end{itemize}
\item[(b)]
Assume further that $W$ is nondegenerate $($\cite{DWZ08}$)$. Then 
$\per(\widehat{\Gamma}(Q,W))$ is silting-discrete if and only if $Q$ is related to a Dynkin quiver by a finite sequence of quiver mutations. 
\end{itemize}
\end{corollary}

\begin{proof} 
(a) This follows from Theorem~\ref{equivalent thm}. 

(b)
Assume that $Q$ is related to a Dynkin quiver $Q'$ by a finite sequence of quiver mutations. 
Since $W$ is a nondegenerate potential, the quiver of the mutated quiver with potential $\mu_i(Q,W)$ is just the mutated quiver $\mu_i(Q)$ and the potential of $\mu_i(Q,W)$ is again nondenegerate. 
From the assumption, we have a finite sequence of mutations 
$\mu_{i_1}\circ\cdots\circ\mu_{i_\ell}(Q)=Q'$ and hence 
we obtain $\mu_{i_1}\circ\cdots\circ\mu_{i_\ell}(Q,W)=(Q',0)$. Since $KQ'=H^0\widehat{\Gamma}(Q',0)$ is representation-finite, it follows from \cite[Corollary 4.6]{KeY11} that $H^0\widehat{\Gamma}(Q,W)$ is representation-finite. As a consequence, $H^0\widehat{\Gamma}(Q,W)$ is $\tau$-tilting finite and by the first statement  $\per(\widehat{\Gamma}(Q,W))$ is silting-discrete.

Next assume that $Q$ is not related to any Dynkin quiver by any finite sequence of quiver mutations. Then it follows from \cite[Corollary 4.4]{Y09} (which is a consequence of \cite[Theorem 1.4]{P08}) that there is a surjection from a certain subset of $\ctilt{2}(\U_{(Q,W)})$ to clusters of $Q$, which is infinite  by \cite[Theorem 1.6]{FZ03}.
This implies that  $\ctilt{2}(\U_{(Q,W)})$ is infinite. So by the first statement $\per(\widehat{\Gamma}(Q,W))$ is not silting-discrete.
\end{proof}

\begin{example}
Let $(Q,W)$ be a Jacobi-finite quiver with potential.  
By Corollary~\ref{main3-2}, if $H^0\widehat{\Gamma}(Q,W)$ is local or representation-finite, then $\per(\widehat{\Gamma}(Q,W))$ is silting-discrete. We give one concrete example for each case:
\begin{itemize}
\item[(1)] $Q$ consists of one vertex and two loops $x$ and $y$ and $W=x^4-x^2y$. \footnote{The Jacobian algebra of this quiver with potential appears as a contraction algebra in \cite{BrownWemyss17,DonovanWemyss16}. We thank Osamu Iyama for providing these two references.}
\item[(2)] $Q$ is an oriented cycle and $W$ is a non-trivial power of the oriented cycle.
\end{itemize}
\end{example}


\section{Contractible stability spaces}\label{s:contractible-stability-spaces}

In this section we apply results in previous sections to study the contractibility of  stability spaces.

\medskip

Let $\D$ be a $K$-linear Hom-finite Krull--Schmidt triangulated category with shift functor $[1]$.
Following \cite[Sections 2.3 and 2.4]{QW14}, we define $\mathrm{Tilt}(\D)$ to be the poset of bounded $t$-structures with $(\D^{\leq 0},\D^{\geq 0})\leq (\D'^{\leq 0},\D'^{\geq 0})$ if and only if there is a finite sequence of left tilts from $(\D^{\leq 0},\D^{\geq 0})$ to $(\D'^{\leq 0},\D'^{\geq 0})$ and define $\mathrm{Tilt}_{\mathrm{alg}}(\D)$ to be the poset of algebraic $t$-structures with $(\D^{\leq 0},\D^{\geq 0})\preceq (\D'^{\leq 0},\D'^{\geq 0})$ if and only if there is a finite sequence of left tilts from $(\D^{\leq 0},\D^{\geq 0})$ to $(\D'^{\leq 0},\D'^{\geq 0})$ via algebraic $t$-structures. Following \cite[Section 4.1]{QW14}, we say that a component of $\mathrm{Tilt}(\D)$ has \emph{finite type} if each $t$-structure in it is algebraic and has only finitely many torsion classes in its heart. Let $\mathrm{Stab}(\D)$ denote the stability space of $\D$ \cite{Br1}.

\begin{theorem}\label{thm:contractibility-vs-silting-discreteness}
Let $(\C,\D)$ be an ST-pair.

\begin{itemize}
\item[(a)]
If $\C$ is silting-discrete, then $\mathrm{Tilt}(\D)$ is connected and has finite type, and $\mathrm{Stab}(\D)$ is connected and contractible.
\item[(b)]
If $\mathrm{Tilt}(\D)$ has a finite-type component which contains a $\C$-silting bounded $t$-structure, then $\C$ is silting-discrete. 
\end{itemize}
\end{theorem}
\begin{proof}
(a)
Since $\C$ is silting-discrete, it follows from Theorem~\ref{main} that all bounded $t$-structures on $\D$ are $\C$-silting and hence algebraic. So $\mathrm{Tilt}(\D)=\mathrm{Tilt}_{\mathrm{alg}}(\D)$. By \cite[Corollary 3.9]{Ai13}, every silting object of $\C$ can be obtained from $A$ by a finite sequence of irreducible mutations. This means, by Proposition~\ref{st-mutation}, that any bounded $t$-structure on $\D$ can be obtained from $(\D_A^{\leq 0},\D_A^{\geq 0})$ by a finite sequence of simple tilts. Therefore,  $\mathrm{Tilt}(\D)$ is connected. Moreover, since any bounded $t$-structure on $\D$ has the form $(\D_M^{\leq 0},\D_M^{\geq 0})$ for some silting object $M$ of $\C$, it follows from Corollary~\ref{finiteness} that the heart has only finitely many torsion classes. So $\mathrm{Tilt}(\D)$ has finite type.
Therefore by \cite[Lemma 4.3 and Theorem 4.9]{QW14}, $\mathrm{Stab}(\D)$ is connected and contractible. 

(b) Let $M$ be a silting object of $\C$ such that the corresponding $t$-structure $(\D^{\leq 0}_M,\D^{\geq 0}_M)$ belongs to the given finite-type component. Let $N$ be any silting object obtained from $M$ by a finite sequence of irreducible left mutations. It follows from Proposition~\ref{st-mutation} that $(\D_N^{\leq 0},\D_N^{\geq 0})$ is in the same component as $(\D^{\leq 0}_M,\D^{\geq 0}_M)$. So
the heart of $(\D_N^{\leq 0},\D_N^{\geq 0})$ has finitely many torsion classes. By Corollary~\ref{finiteness}, 
$\ssilt{2_{N}}\C$ is finite.  By \cite[Theorem 2.4]{AM15}, $\C$ is silting-discrete. 
\end{proof}

We remark that the $t$-discreteness of $\D$ is not a necessary condition of the contractibility of $\mathrm{Stab}(\D)$. For example, it is known that $\mathrm{Stab}(\Db(\coh(\mathbb{P}^1)))=\mathbb{C}^2$ by \cite[Theorem 1.1]{Okada06}, but $\Db(\coh(\mathbb{P}^1))$ is not $t$-discrete by Corollary~\ref{cor:projective-variety-not-t-discrete}.

As an application of Theorem~\ref{thm:contractibility-vs-silting-discreteness}, we can show that some derived categories have contractible stability spaces. The following result is independently obtained in \cite{PSZ17}.

\begin{corollary}\label{cor:contractible-derived-category}
Let $\Lambda$ be a finite-dimensional $K$-algebra. 
If $\Kb(\proj\Lambda)$ is silting-discrete, 
then $\mathrm{Stab}(\Db(\mod\Lambda))$ is contractible. 
\end{corollary}
\begin{proof}
Recall that $(\Kb(\proj\Lambda),\Db(\mod\Lambda))$ is an ST-pair (Lemma~\ref{lem:ST-pair-for-fd-algebra}). The desired result follows by applying Theorem~\ref{thm:contractibility-vs-silting-discreteness} to this ST-pair.
\end{proof}

\begin{example}\label{ex:contractibility-derived-categories}
By Corollary~\ref{cor:contractible-derived-category}, 
the stability space of $\Db(\mod \Lambda)$ is contractible if $\Lambda$ is one of the algebra in Example \ref{ex:silting-discrete-derived-categories}. In particular, Example \ref{ex:silting-discrete-derived-categories}(2) affirms the first part of \cite[Conjecture 5.8]{Q15} and Example \ref{ex:silting-discrete-derived-categories}(3) is \cite[Theorem 8.2]{BPP2}.
\end{example}

Concerning stability spaces of Calabi--Yau triangulated categories, we have the following result.

\begin{corollary}\label{cor:cluster-tilting-finite-vs-contractibility-cy-case}
Let $d\geq 1$ be an integer and let $(\C,\D)$ be a $(d+1)$-Calabi--Yau pair. 
If $\ctilt{d}(\C/\D)$ is finite, then $\mathrm{Stab}(\D)$ is contractible. 
\end{corollary}
\begin{proof}
By Theorem~\ref{equivalent thm}, $\C$ is silting-discrete. By Theorem~\ref{thm:contractibility-vs-silting-discreteness}, $\mathrm{Stab}(\D)$ is contractible.
\end{proof}

The following Calabi--Yau triangulated categories associated to Dynkin quivers have contractible stability spaces.

\begin{corollary}\label{derived pp} 
\begin{itemize}
\item[(a)]
Let $Q$ be a Dynkin quiver and 
$\Gamma=\Gamma_{d+1}(Q)$ the derived (d+1)-preprojective algebra. Then $\mathrm{Stab}(\Dfd(\Gamma))$ is contractible.
\item[(b)] Let $(Q,W)$ be a Jacobi-finite quiver with nondegenerate potential and  $\widehat{\Gamma}=\widehat{\Gamma}(Q,W)$ 
the complete Ginzburg algebra.
If $Q$ is related to a Dynkin quiver by a finite sequence of quiver mutations, then $\mathrm{Stab}(\Dfd(\widehat{\Gamma}))$ is contractible.
\end{itemize}
\end{corollary}

\begin{proof}
This is immediate from 
Corollaries \ref{main3-1}, \ref{main3-2} and \ref{cor:contractible-derived-category}.
\end{proof}

\begin{remark}
Corollary \ref{derived pp} affirms \cite[Conjecture 1.3]{I17} and the second part of \cite[Conjecture 5.8]{Q15} (see also \cite[Corollary 1.2]{I17} and \cite[Corollary 5.1]{QW14}). 
\end{remark}


\end{document}